\documentclass{article}

\usepackage[final, nonatbib]{neurips_2024}

\usepackage[utf8]{inputenc} 
\usepackage{macro_neurips} 
\usepackage[T1]{fontenc}    
\usepackage{hyperref}       
\usepackage{url}            
\usepackage{booktabs}       
\usepackage{amsfonts}       
\usepackage{nicefrac}       
\usepackage{microtype}      
\usepackage{xcolor}         

\title{Matrix Denoising with Doubly Heteroscedastic Noise: Fundamental Limits and Optimal Spectral Methods}

\author{%
  Yihan Zhang \\
  Institute of Science and Technology Austria \\
  \texttt{zephyr.z798@gmail.com} \\
  \And
  Marco Mondelli \\
  Institute of Science and Technology Austria \\
  \texttt{marco.mondelli@ist.ac.at} 
}

\begin{document}

\maketitle

\begin{abstract}
We study the matrix denoising problem of estimating the singular vectors of a rank-$1$ signal corrupted by noise with both column and row correlations. Existing works are either unable to pinpoint the exact asymptotic estimation error or, when they do so, the resulting approaches (e.g., based on whitening or singular value shrinkage) remain vastly suboptimal. On top of this, most of the literature has focused on the special case of estimating the left singular vector of the signal when the noise only possesses row correlation (one-sided heteroscedasticity). In contrast, our work establishes the information-theoretic and algorithmic limits of matrix denoising with doubly heteroscedastic noise. We characterize the exact asymptotic minimum mean square error, and design a novel spectral estimator with rigorous optimality guarantees: under a technical condition, it attains positive correlation with the signals whenever information-theoretically possible and, for one-sided heteroscedasticity, it also achieves the Bayes-optimal error. Numerical experiments demonstrate the significant advantage of our theoretically principled method with the state of the art. The proofs draw connections with statistical physics and approximate message passing, departing drastically from standard random matrix theory techniques. 
\end{abstract}



\newcounter{asmpctr} 
\setcounter{asmpctr}{\value{enumi}}

\section{Introduction}
\label{sec:intro}

Matrix denoising is a central primitive in statistics and machine learning, and the problem is to recover a signal $X\in\bbR^{n\times d}$ from an observation $ A = X + W  $ corrupted by additive noise $W$.
This finds applications across multiple domains of sciences, e.g., imaging \cite{imaging1,
imaging3}, biology \cite{
biology2,bioilogy3} and astronomy \cite{astronomy1,astronomy2}. 
When $X$ has low rank and $W$ 
i.i.d.\ entries, $A$ is the standard model for principal component analysis, typically referred to as the Johnstone spiked covariance model \cite{Johnstone_spiked}. When $n, d$ are both large and proportional, which corresponds to the most sample-efficient regime, its Bayes-optimal limits are well understood \cite{Miolane_asymm}, and it has been established how to achieve them efficiently \cite{MV_AoS}. 
Minimax/non-asymptotic guarantees are also available in special cases, such as sparse PCA \cite{SparsePCA}, Gaussian mixtures \cite{wu-zhou-2019-em} and certain joint scalings of $(n,d)$ \cite{MW_diverging}. 

However, in most applications, noise is highly structured and correlated, thereby calling for more realistic assumptions on $W$ than having i.i.d.\ entries. 
A recent line of work addresses this concern by studying matrix denoising with heteroscedastic noise \cite{Agterberg_low-rank_hetero,shrinkage_sepcov,mtx_denois_partial,PCA_hetero_shrink,DLY_shrinkage}, resting 
on two basic ideas: whitening and singular value shrinkage. 
Whitening refers to 
multiplying the data matrix by the square root of the inverse covariance, 
in order to reduce the model to one with i.i.d.\ noise; and singular value shrinkage 
retains the singular vectors of the data while deflating the singular values to correct for the noise. 
Though the exact asymptotic performance of these algorithms has been derived \cite{shrinkage_sepcov,mtx_denois_partial,PCA_hetero_shrink,DLY_shrinkage}, their optimality is yet to be determined from a Bayesian standpoint. 
In fact, we will prove that whitening and shrinkage are \emph{not} the correct way to approach Bayes optimality. 

\paragraph{Main contributions.}
We focus on the prototypical model 
$ A = X + W $, where $ X = \frac{\lambda}{n} u^* {v^*}^\top $ is a rank-$1$ signal, $\lambda$ is the signal-to-noise ratio (SNR), and $ W = \Xi^{1/2} \wt{W} \Sigma^{1/2} $ is doubly heterogeneous noise.
Here $u^*, v^*$ follow i.i.d.\ priors; 
$\wt{W}$ contains i.i.d.\ Gaussian entries; the covariance matrices $\Xi, \Sigma$ capture column and row correlations; and we consider the typical high-dimensional regime in which 
$n,d$ are both large and proportional. 
Our main results are summarized below. 

\vspace{-.5em}

\begin{enumerate}

    \item We design an efficient spectral estimator to recover $u^*, v^*$, and we provide a precise asymptotic analysis of its performance, see \Cref{thm:spec}. This estimator is given by the top singular vectors of a matrix obtained by carefully pre-processing $A$, see \Cref{eqn:A*_main}. 
    
    \item When the priors of 
    $u^*, v^*$ are standard 
    Gaussian, we show in \Cref{cor:one_sided} that the spectral estimator above is optimal in the following sense: \emph{(i)} under a technical condition, it achieves the \emph{optimal weak recovery threshold}, namely its mean square error is non-trivial as soon as this is information-theoretically possible; \emph{(ii)} it achieves the \emph{Bayes-optimal error} for $ u^* $ (resp.\ $ v^* $) when $ \Xi$ (resp.\ $ \Sigma$) is the identity. These optimality guarantees follow from rigorously obtaining the asymptotic 
    minimum mean square error (MMSE) for the estimation of the whitened signals $ \Xi^{-1/2} u^*$ and $\Sigma^{-1/2} v^*$, see \Cref{thm:IT}.  
    
\end{enumerate}

\vspace{-.5em}

Our spectral estimator only involves matrix multiplication and computing principal singular vectors. 
Practically, this can be efficiently done using standard SVD algorithms or power iteration \cite{Householder}. 
For both one-sided and double heteroscedasticity, numerical experiments in \Cref{fig:1sided,fig:2sided} show significant advantage of our spectral estimator for moderate SNRs over HeteroPCA \cite{HeteroPCA} and shrinkage-based methods, i.e.,  Whiten-Shrink-reColor \cite{PCA_hetero_shrink,Leeb_mtx_denois}, OptShrink \cite{Nadakuditi_shrinkage_rotinv}, and ScreeNOT \cite{DGR_ScreeNOT}.

\vspace{-.5em}

\paragraph{Proof techniques.} We take a completely different route from classical approaches in statistics and random matrix theory (e.g., whitening and shrinkage), and instead exploit tools from statistical physics and the theory of approximate message passing. In particular, the MMSE for the whitened signals $ \Xi^{-1/2} u^*, \Sigma^{-1/2} v^*$ is obtained via an interpolation argument \cite{BM_interp,Miolane_asymm,miolane_thesis}. 
This result allows us to derive the weak recovery threshold for estimating the true signals $ u^*, v^* $. 
Moreover, for one-sided heteroscedasticity, this MMSE coincides with that for estimating the true signal on the homoscedastic side. 
Evaluating the Bayes-optimal estimators requires solving high-dimensional integrals that are computationally intractable. To circumvent this issue, we propose an efficient spectral method that still enjoys optimality guarantees. Its design and analysis draw connections with a family of iterative algorithms called approximate message passing (AMP) \cite{Bayati_Montanari,amp-tutorial}. 
All our results 
are mathematically rigorous, with the only technical condition being ``\Cref{eqn:thr_bayes} implies $\sigma_2^* < 1$'' in \Cref{thm:spec} that we only managed to verify numerically, but not analytically; see \Cref{rk:asmp_spec}.

\vspace{-.5em}


\section{Related work}
\label{sec:related_work}

\vspace{-.5em}

Research on matrix denoising in the homoscedastic case ($ \Xi = I_n, \Sigma = I_d $) has a rich history, and in random matrix theory properties of the spectrum and eigenspaces of $A$ have been studied exhaustively. 
Most prominently, the BBP phase transition phenomenon \cite{BBAP} (and its finite-sample counterpart \cite{Nadler_finite_sample}) unveils a threshold of the SNR $\lambda$ above which a pair of outlier singular value and singular vector emerge. Under i.i.d.\ priors, the asymptotic Bayes-optimal estimation error has been derived \cite{Miolane_asymm,miolane_thesis}, rigorously justifying predictions from statistical physics \cite{Bayes_conj}. 
The proof uses the interpolation method due to Guerra \cite{Guerra}, originally developed in the context of mean-field spin glasses. 
Besides low-rank matrix estimation, this method (including its adaptive variant \cite{BM_interp} and the Aizenman--Sims--Starr scheme \cite{ASS}) has also been applied to a range of 
problems, including spiked tensor estimation \cite{mutual_info_ten}, generalized linear models \cite{Barbier_GLM}, stochastic block models \cite{SBM_multi} and group synchronization \cite{YWF_group}. 

Moving to the heteroscedastic case, an active line of work concerns optimal singular value shrinkage methods \cite{PCA_hetero_shrink,mtx_denois_partial,Leeb_mtx_denois,shrinkage_sepcov,Nadakuditi_shrinkage_rotinv,DLY_shrinkage}. These methods can be regarded as a special family of rotationally invariant estimators, which apply a univariate function $ \eta\colon\bbR_{\ge0}\to\bbR $ to each empirical singular value. 
An example widely employed by practitioners is the thresholding function $ \eta_\theta(y) = y \indicator{y > \theta} $ \cite{DGR_ScreeNOT}. 
In the presence of noise heteroscedasticity, most of these results are based on 
whitening \cite{biwhitening}. 
%
Another model of noise heterogeneity common in the literature takes 
$ W = \wt{W} \circ \Delta^{\circ 1/2} $, where $ \wt{W} $ has i.i.d.\ Gaussian entries, $\Delta$ is a deterministic block matrix with fixed (i.e., constant with respect to $n,d$) number of blocks, and $ \circ $ denotes the element-wise product. This means that  the entries of the noise are independent but non-identically distributed, and they follow the variance profile $ \Delta $. 
The corresponding low-rank perturbation $A$, known as a spiked inhomogeneous matrix, has attracted attention from both the information-theoretic \cite{IT_inhomo,Reeves_mtx_ten,IT_inhomo_univ} and the algorithmic sides \cite{WPCA,spec_inhomo,AMP_inhomo}. 
Spiked inhomogeneous matrices have some connections with the model considered in this paper: if $\Delta$ has rank $1$, such $A$ can be realized by taking $\Xi, \Sigma$ to be diagonal with suitable block structures. 
Finally, non-asymptotic results for the heteroscedastic and the inhomogeneous models have been derived in varying generality in \cite{HeteroPCA,DeflatedHeteroPCA,CWC_small,Agterberg_low-rank_hetero,wishart_nonasymp_hetero}. We highlight that our paper is the \emph{first} to establish information-theoretic and algorithmic limits for doubly heteroscedastic noise.

Our 
characterization of the spectral estimator relies on an AMP algorithm that  converges to it by performing 
power iteration. 
AMP refers to a family of iterative procedures, whose 
performance in the high-dimensional limit is precisely characterized by a low-dimensional deterministic recursion called state evolution \cite{Bayati_Montanari,Bolthausen}. Originally introduced for compressed sensing \cite{DMM09}, AMP algorithms have been developed for various settings, including low-rank estimation \cite{MV_AoS, fan2020approximate,Barbier_PCA_rotinv} and inference in generalized linear models \cite{RanganGAMP,rangan2019vector,venkataramanan2022estimation}. 
Beyond statistical estimation, AMP proves its versatility as both an efficient algorithm and a proof technique for studying e.g.\
posterior sampling \cite{sampling_spiked}, 
spectral universality \cite{spec_univ}, 
first order methods with random data \cite{DMFT}, mismatched 
estimation \cite{barbier2022price}, spectral estimators for generalized linear models \cite{zhang2023spectral,mixed-zmv-arxiv} and their combination with linear estimators \cite{mondelli2021optimalcombination}.

\section{Problem setup}
\label{sec:prelim}

Consider the following rank-$1$ rectangular matrix estimation problem with doubly heteroscedastic noise where we observe
\vspace{-.5em}
\begin{align}
A &= \frac{\lambda}{n} u^* {v^*}^\top + W \in\bbR^{n\times d} , \label{eqn:model}
\end{align}
and aim to estimate $ u^*, v^* $. 
The following assumptions are imposed throughout the paper. 
    The dimensions $ n,d\to\infty $ obey the proportional scaling $ n/d\to\delta\in(0,\infty) $, where $ \delta $ is 
    the aspect ratio. 
    The SNR $ \lambda\in[0,\infty) $ is a known constant (relative to $n,d$). 
    The signals $ (u^*, v^*) \sim P^{\ot n} \ot Q^{\ot d} $ have i.i.d.\ priors, where $ P , Q $ are distributions on $\bbR$ with mean $0$ and variance $1$. 
    The unknown noise matrix has the form $ W = \Xi^{1/2} \wt{W} \Sigma^{1/2} \in \bbR^{n\times d} $, with 
    $ \wt{W}_{i,j}\iid\cN(0,1/n) $ independent of $ (u^*, v^*) $. 
    The covariances 
    $ \Xi\in\bbR^{n\times n}, \Sigma\in\bbR^{d\times d} $ are known, deterministic,\footnote{All our results hold verbatim if $ \Xi, \Sigma $ are random matrices independent of each other and of $u^*, v^*, \wt{W}$. } 
    strictly positive definite and satisfy
\vspace{-.5em}
    \begin{align}
    \lim_{n\to\infty} \frac{1}{n} \tr(\Xi) &= \lim_{d\to\infty} \frac{1}{d} \tr(\Sigma) = 1 . \label{eqn:trace}
    \end{align}
    Their empirical spectral distributions (ESD) converge (as $n,d\to\infty$ s.t.\ $n/d\to\infty$) weakly to the laws of the random variables $ \ol{\Xi}$ and $ \ol{\Sigma} $. 
    Furthermore, $ \normtwo{\Xi}, \normtwo{\Sigma} $ are uniformly bounded over $d$. 
    The supports of $ \ol{\Xi}, \ol{\Sigma} $ are compact subsets of $ (0,\infty) $. 
    For all $ \eps>0 $, there exists $ d_0\in\bbN $ s.t.\ for all $ d\ge d_0 $, 
    \begin{align}
    \supp(\ESD(\Xi)) \subset \supp(\ol{\Xi}) + [-\eps, \eps] , \quad 
    \supp(\ESD(\Sigma)) \subset \supp(\ol{\Sigma}) + [-\eps, \eps] . \label{eqn:supp} 
    \end{align}


The trace assumption \Cref{eqn:trace} on the covariances is 
for normalization purposes since the values of the traces, if not $1$, can be absorbed into $\lambda$. 
The support assumption \Cref{eqn:supp} excludes outliers in the spectra of covariances which may contribute to undesirable spikes in $A$ \cite{shrinkage_sepcov}. 





\section{Information-theoretic limits}
\label{sec:IT}

For mathematical convenience, 
in this section, we switch to an equivalent rescaled model 
\begin{align}
Y &\coloneqq \sqrt{n} A = \sqrt{\frac{\snr}{n}} u^* {v^*}^\top + \Xi^{1/2} Z \Sigma^{1/2} \in\bbR^{n\times d} , \label{eqn:Y_mtx}
\end{align}
where $ \snr \coloneqq \lambda^2 $ and $ Z = \sqrt{n} \wt{W} $ contains i.i.d.\ elements $ Z_{i,j} \iid \cN(0,1) $. 
Abusing terminology, we refer to $\snr$ as the SNR of $Y$. 
Define also $ \alpha \coloneqq 1/\delta \in (0,\infty) $ so that $ d/n\to\alpha $. 
The scaling of the parameters in \Cref{eqn:Y_mtx} turns out to be more convenient for presenting the results in this section. 
Results for $Y$ can be easily translated to $A$ by a change of variables.

Let $\wt{u}^* \coloneqq \Xi^{-1/2} u^*$ and $
\wt{v}^* \coloneqq \Sigma^{-1/2} v^*$ denote the whitened signals. 
The main result of this section is \Cref{thm:IT}, which characterizes the performance of the matrix minimum mean square error (MMSE) associated to the estimation of $\wt{u}^*(\wt{v}^*)^\top, \wt{u}^*(\wt{u}^*)^\top$ and $\wt{v}^*(\wt{v}^*)^\top,$ via the corresponding Bayes-optimal estimators: 
\begin{align}
    \mmse_n(\snr) &\coloneqq \frac{1}{nd} \expt{ \normf{ \wt{u}^* (\wt{v}^*)^\top - \expt{ \wt{u}^* (\wt{v}^*)^\top  \mid Y } }^2 } , \label{eqn:MMSE_Y} \\
    \mmse_n^u(\snr) &\coloneqq \frac{1}{n^2} \expt{ \normf{\wt{u}^* (\wt{u}^*)^\top - \expt{\wt{u} (\wt{u}^*)^\top \mid Y}}^2 } , \label{eqn:MMSE_u} \\
    \mmse_n^v(\snr) &\coloneqq \frac{1}{d^2} \expt{ \normf{\wt{v}^* (\wt{v}^*)^\top - \expt{\wt{v}^* (\wt{v}^*)^\top \mid Y}}^2 } .
    \label{eqn:MMSE_v}
\end{align} 
Our characterization 
involves a pair of parameters $ (q_u^*, q_v^*)\in\bbR_{\ge0}^2 $ defined as the largest solution to 
\vspace{-.25em}
\begin{align}
    q_u &= \expt{\frac{\alpha\snr q_v \ol{\Xi}^{-2}}{1 + \alpha\snr q_v \ol{\Xi}^{-1}}} , \qquad 
    q_v = \expt{\frac{\snr q_u \ol{\Sigma}^{-2}}{1 + \snr q_u \ol{\Sigma}^{-1}}} . 
    \label{eqn:fp_it}
\end{align}
Here and throughout the paper, all expectations involving $\ol{\Xi}, \ol{\Sigma}$ are computed as integrals against the limiting spectral distributions of $\Xi, \Sigma$. 

The proposition below, proved in \Cref{app:pf_prop:bayes_fp_sol}, justifies the existence of the solution to \Cref{eqn:fp_it} and identifies when a non-trivial solution emerges. 

\begin{proposition}
\label{prop:bayes_fp_sol}
The fixed point equation \Cref{eqn:fp_it} always has a trivial solution $ (0, 0) $. 
There exists a non-trivial solution $ (q_u^*, q_v^*)\in\bbR_{>0}^2 $ if and only if 
\vspace{-.25em}
\begin{align}
    \alpha \snr^2 \expt{\ol{\Sigma}^{-2}} \expt{\ol{\Xi}^{-2}} &> 1 , \label{eqn:thr_bayes_change} 
\end{align}
in which case the non-trivial solution is unique. 
\end{proposition}

We are now ready to state our main result on the MMSE.

\begin{theorem}
\label{thm:IT}
Assume $ P = Q = \cN(0,1) $. 
For almost every $ \snr > 0 $, 
\vspace{-.25em}
\begin{gather}
\begin{split}
\lim_{n\to\infty} \mmse_n(\snr) &= \expt{\ol{\Xi}^{-1}} \expt{\ol{\Sigma}^{-1}} - q_u^* q_v^* , 
\end{split} \label{cor:MMSE_mtx} \\
\begin{split}
\lim_{n\to\infty} \mmse_n^u(\snr) &= \expt{\ol{\Xi}^{-1}}^2 - {q_u^*}^2 
, \qquad
\lim_{n\to\infty} \mmse_n^v(\snr) = \expt{\ol{\Sigma}^{-1}}^2 - {q_v^*}^2 . 
\end{split} \label{cor:MMSE_vec}
\end{gather}
\end{theorem}

We note that 
\vspace{-.25em}
\begin{equation}\label{eq:trivial}
\lim_{n\to\infty}\frac{1}{nd} \expt{ \normf{ \wt{u}^* (\wt{v}^*)^\top }^2 }=\lim_{n\to\infty}\frac{1}{nd} \expt{ \normtwo{ \wt{u}^*  }^2 }\expt{ \normtwo{ \wt{v}^* }^2 }=\expt{\ol{\Xi}^{-1}} \expt{\ol{\Sigma}^{-1}},
\end{equation}
where the last step follows from \Cref{prop:quadratic_form}. This quantity represents the trivial error in the estimation of $\wt{u}^*(\wt{v}^*)^\top$, which is achieved by the all-0 estimator. Analogous considerations hold for $\wt{u}^*(\wt{u}^*)^\top$ and $\wt{v}^*(\wt{v}^*)^\top$, for which the trivial estimation error is $\expt{\ol{\Xi}^{-1}}^2$ and $\expt{\ol{\Sigma}^{-1}}^2$, respectively. Thus, \Cref{prop:bayes_fp_sol} and \Cref{thm:IT} identify \Cref{eqn:thr_bayes_change} as the condition for non-trivial estimation, and the smallest $\gamma$ that satisfies \Cref{eqn:thr_bayes_change} gives the 
\emph{weak recovery threshold}.

We show below that the weak recovery threshold is the same for the estimation of the true signals $u^*{v^*}^\top, u^*{u^*}^\top$ and $v^*{v^*}^\top$. In this case, since the signal priors are Gaussian, using the same passages as in \Cref{eq:trivial} one has that the trivial estimation error for $u^*{v^*}^\top, u^*{u^*}^\top$ and $v^*{v^*}^\top$ is always equal to $1$.

\begin{corollary}
\label{cor:weak_rec_thr}
Assume $ P = Q = \cN(0,1) $. The MMSE associated to the estimation of $u^*{v^*}^\top$ is non-trivial, i.e, 
\vspace{-1em}
\begin{align}
\lim_{n\to\infty}\frac{1}{nd} \expt{ \normf{u^* {v^*}^\top - \expt{u^* {v^*}^\top \mid Y}}^2 }<1 \label{eqn:MMSE_uv_thr} 
\end{align}
if and only if \Cref{eqn:thr_bayes_change} holds. The same result holds for the MMSE 
of $u^*{u^*}^\top$ and $v^*{v^*}^\top$. 
\end{corollary}


\paragraph{Proof strategy.}
To derive the characterizations in \Cref{thm:IT}, we 
write the posterior distribution of $ u^*, v^* $ given $Y$ %
in a Gibbs form, i.e., its density is the exponential of a Hamiltonian normalized by a partition function. 
The interpolation argument relates the log-partition function (also referred to as the `free energy') of the posterior to that of the posteriors of two Gaussian location models. 
Since i.i.d.\ Gaussianity is key to this approach, the challenge is to handle noise covariances. 
Our idea is to incorporate the covariances into the priors. 
In terms of the Hamiltonian, the model is equivalent to the estimation of the whitened signals $ \Xi^{-1/2} u^*, \Sigma^{-1/2} v^*$, whose priors have covariances, in the presence of i.i.d.\ Gaussian noise. 
We then manage to carry out the interpolation argument for the equivalent model and evaluate the free energy of the corresponding Gaussian location models. 

Specifically, let us starting by writing down the expression of the posterior distribution after setting up some notation. 
For $ u\in\bbR^n, v\in\bbR^d $, let $ \wt{u} \coloneqq \Xi^{-1/2} u , \wt{v} \coloneqq \Sigma^{-1/2} v $.
Define the densities
\begin{align}
\diff\wt{P}(\wt{u}) \coloneqq \sqrt{\det(\Xi)} \diff P^{\ot n}(\Xi^{1/2} \wt{u}) , \qquad
\diff\wt{Q}(\wt{v}) \coloneqq \sqrt{\det(\Sigma)} \diff Q^{\ot d}(\Sigma^{1/2} \wt{v}) , \notag 
\end{align}
where the determinant factors ensure that 
the integrals equal $1$. 
With $ P = Q = \cN(0,1) $, we have $ \wt{P} = \cN(0_n, \Xi^{-1}), \wt{Q} = \cN(0_d, \Sigma^{-1}) $, and 
%
from Bayes' rule the posterior 
of $ (u^*, v^*) $ given $Y$ is 
\begin{align}
\diff P(u, v \mid Y) &= \frac{1}{\cZ_n(\snr)} \exp\paren{ \cH_n(\Xi^{-1/2} u, \Sigma^{-1/2} v) } \diff P^{\ot n}(u) \diff Q^{\ot d}(v) , \label{eqn:posterior}
\end{align}
where the Hamiltonian and the partition function are given respectively by 
\begin{align}
&\cH_n(\wt{u}, \wt{v}) \coloneqq \sqrt{\frac{\snr}{n}} \wt{u}^\top Z \wt{v} + \frac{\snr}{n} \wt{u}^\top \wt{u}^* \wt{v}^\top \wt{v}^* - \frac{\snr}{2n} \normtwo{\wt{u}}^2 \normtwo{\wt{v}}^2 , \label{eqn:Hamil} \\
&\cZ_n(\snr) \coloneqq \hspace{-.4em}\int\hspace{-.4em} 
\int\hspace{-.2em} 
\exp\paren{ \cH_n(\Xi^{-1/2} u, \Sigma^{-1/2} v) } \diff P^{\ot n}(u) \diff Q^{\ot d}(v) 
=\hspace{-.4em}\int\hspace{-.4em} 
\int\hspace{-.2em} 
\exp\paren{ \cH_n(\wt{u}, \wt{v}) } \diff \wt{P}(\wt{u}) \diff \wt{Q}(\wt{v}) . \label{eqn:part_Y} 
\end{align}
Define the free energy as 
\vspace{-.5em}
\begin{align}
\cF_n(\snr) &\coloneqq \frac{1}{n} \expt{ \log\cZ_n(\snr) } . \label{eqn:energy_Y}
\end{align}

The major technical step 
is to characterize $\cF_n(\snr)$ in the large $n$ limit in terms of a bivariate functional $\cF$ introduced below. This is the core component to derive the MMSE characterization. 

For a positive random variable $ \ol{\Sigma} $ subject to the conditions in \Cref{sec:prelim}, let
\begin{align}
\psi_{\ol{\Sigma}}(\snr) &\coloneqq \frac{1}{2} \paren{ \snr \expt{\ol{\Sigma}^{-1}} - \expt{\log\paren{1 + \snr\ol{\Sigma}^{-1}}} } . \label{eqn:psi_Sigma}
\end{align}
As shown in \Cref{app:gauss_ch}, $\psi_{\ol{\Sigma}}(\snr)$ is the limiting free energy of a Gaussian channel, in which one wishes to estimate $x^*\in \bbR^n$ from the observation $Y = \sqrt{\snr} x^* + \Sigma^{1/2} Z$ corrupted by anisotropic Gaussian noise with covariance $\Sigma$. 
Using \Cref{eqn:psi_Sigma}, let us define the replica symmetric potential $\cF$:
\begin{align}
\cF(q_u, q_v) &\coloneqq \psi_{\ol{\Xi}}(\alpha \snr q_v) + \alpha \psi_{\ol{\Sigma}}(\snr q_u) - \frac{\alpha \snr}{2} q_u q_v , \notag 
\end{align}
and the set of critical points of $\cF$: 
\begin{align}
\cC(\snr, \alpha) &\coloneqq \brace{ (q_u, q_v)\in\bbR_{\ge0}^2 : \partial_1 \cF(q_u, q_v) = 0 , \partial_2 \cF(q_u, q_v) = 0 } \notag \\
&= \brace{ (q_u, q_v)\in\bbR_{\ge0}^2 : q_u = 2 \psi_{\ol{\Xi}}'(\alpha\snr q_v) , q_v = 2\psi_{\ol{\Sigma}}'(\snr q_u) } \label{eqn:cQ_solve} \\
&= \brace{ (q_u, q_v)\in\bbR_{\ge0}^2 : (q_u, q_v) \textnormal{ solves } \Cref{eqn:fp_it} } , 
\notag 
\end{align}
where the last equality is a direct calculation of $\psi_{\ol{\Xi}}', \psi_{\ol{\Sigma}}'$. The following result, proved in \Cref{app:pf}, shows that the limit of $ \cF_n(\snr) $ is given by a dimension-free variational problem involving $ \cF(q_u, q_v) $. 

\vspace{.25em}

\begin{theorem}[Free energy]
\label{thm:free_energy}
Assume $ P = Q = \cN(0,1) $. Then, we have
\begin{align}
\lim_{n\to\infty} \cF_n(\snr) &= \sup_{q_v\ge0} \inf_{q_u\ge0} \cF(q_u, q_v) 
= \sup_{(q_u, q_v)\in\cC(\snr, \alpha)} \cF(q_u, q_v)
, \notag 
\end{align}
and $ \sup_{q_v} \inf_{q_u} $ and $ \sup_{(q_u, q_v)} $ are achieved by the same $ (q_u^*, q_v^*) $ in \Cref{prop:bayes_fp_sol}. 
\end{theorem}

\vspace{.25em}

\begin{remark}[Equivalent models]
\label{rk:free_energy_equiv}
Informally, the above result says that the matrix model \Cref{eqn:Y_mtx} is equivalent at the level of Hamiltonian to the following two statistically uncorrelated vector models:
\begin{align}
    Y^u &\coloneqq \sqrt{\alpha \snr q_v^*}  u^* + \Xi^{1/2} Z^u \in\bbR^n , \quad 
    Y^v \coloneqq \sqrt{\snr q_u^*}  v^* + \Sigma^{1/2} Z^v \in \bbR^d , \label{eqn:Yuv} 
\end{align}
with $ q_u^*, q_v^* $ the largest solution to \Cref{eqn:fp_it} and 
 $   (u^*, v^*, Z^u, Z^v) \sim P^{\ot n} \ot Q^{\ot d} \ot \cN(0_n, I_n) \ot \cN(0_d, I_d) .$ 
\end{remark}

\vspace{.25em}

\begin{remark}[Gaussian priors]
\label{rk:free_energy_prior}
\Cref{thm:free_energy} crucially relies on having Gaussian priors $P, Q$. 
This assumption is mainly used to derive single-letter (i.e., dimension-free) expressions of the free energy of the vector models in \Cref{eqn:Yuv} which, under Gaussian priors, are nothing but Gaussian integrals. 
The free energy, and hence the MMSE, are expected to be sensitive to the priors. 
Indeed, this is already the case in the homoscedastic setting $\Xi = I_n, \Sigma = I_d$ \cite{Miolane_asymm}. 
An extension towards general i.i.d.\ priors is a challenging open problem and,  
in fact, without posing additional assumptions on $\Xi, \Sigma$, it is unclear whether a single-letter expression for free energy and MMSE is possible. 
\end{remark}

At this point, the MMSE can be derived from the above characterization of free energy. 
Indeed, let
\begin{align}
\cD(\alpha) &\coloneqq \brace{ \snr > 0 : \cF \textnormal{ has a unique maximizer } (q_u^*, q_v^*) \textnormal{ over } \cC(\snr, \alpha) } . \notag 
\end{align}
The envelope theorem \cite[Corollary 4]{envelope} ensures that $ \cD(\alpha) $ is equal to $ \bbR_{>0} $ up to a countable set. 
Using algebraic relations between free energy and MMSE, we prove \Cref{cor:MMSE_mtx,cor:MMSE_vec} for all $ \snr\in\cD(\alpha) $ (and, thus, for almost every $ \snr > 0 $). 
Then, using the Nishimori identity (\Cref{prop:nishimori}) and the fact that the ESDs of $\Xi, \Sigma$ are upper and lower bounded by constants independent of $n$ and $d$, \Cref{cor:weak_rec_thr} also follows. The formal arguments are contained in \Cref{app:pfmain}. 

\section{Spectral estimator}
\label{sec:spec}

This section introduces a spectral estimator that meets the weak recovery threshold and, for one-sided heteroscedasticity, attains the Bayes-optimal error. 
%
Suppose 
that the following condition holds
\begin{align}
\frac{\lambda^4}{\delta} \expt{\ol{\Sigma}^{-2}} \expt{\ol{\Xi}^{-2}} &> 1 , \label{eqn:thr_bayes} 
\end{align}
which is equivalent to \Cref{eqn:thr_bayes_change}. 
Under this condition, the fixed point equations \Cref{eqn:fp_it} have a unique pair of positive solutions $ (q_u^*, q_v^*) $. For convenience, we also define the rescalings $\mu^* \coloneqq \lambda q_v^* / \delta , \nu^* \coloneqq \lambda q_u^*$, and the auxiliary quantities 
\begin{equation}\label{eq:aux}
    b^* \coloneqq \frac{1}{\delta} \expt{ \frac{\lambda}{\lambda \nu^* + \ol{\Sigma}} }, \qquad 
    c^* \coloneqq \expt{\frac{\lambda}{\lambda \mu^* + \ol{\Xi}}}.
\end{equation}
Now, we pre-process the data matrix $A$ as 
\begin{align}
    A^* \coloneqq \lambda (\lambda (\mu^* + b^*) I_n + \Xi)^{-1/2} \Xi^{-1/2} A \Sigma^{-1/2} (\lambda (\nu^* + c^*) I_d + \Sigma)^{-1/2} , \label{eqn:A*_main} 
\end{align}
from which we obtain the spectral estimators 
\begin{subequations}
\begin{align}
\begin{split}
    \wh{u} &\coloneqq \eta_u \sqrt{n}  \frac{\Xi^{1/2} (\lambda (\mu^* + b^*) I_n + \Xi)^{-1/2} (\lambda \mu^* I_n + \Xi) u_1(A^*)}{\normtwo{\Xi^{1/2} (\lambda (\mu^* + b^*) I_n + \Xi)^{-1/2} (\lambda \mu^* I_n + \Xi) u_1(A^*)}} , 
\end{split} \label{eqn:uhat} \\
\begin{split}
    \wh{v} &\coloneqq \eta_v \sqrt{d}  \frac{\Sigma^{1/2} (\lambda(\nu^* + c^*) I_d + \Sigma)^{-1/2} (\lambda\nu^* I_d + \Sigma) v_1(A^*)}{\normtwo{\Sigma^{1/2} (\lambda(\nu^* + c^*) I_d + \Sigma)^{-1/2} (\lambda\nu^* I_d + \Sigma) v_1(A^*)}} ,  
\end{split} \label{eqn:vhat}
\end{align}
\label{eqn:def_spec}
\end{subequations}
where $ u_1(\cdot)/v_1(\cdot) $ denote the top left/right singular vectors 
and 
\begin{gather}
\begin{split}
    \eta_u &\coloneqq \sqrt{ \frac{\lambda\mu^* }{\lambda\mu^* + 1} } , \qquad 
    \eta_v \coloneqq \sqrt{ \frac{\lambda\nu^* }{\lambda\nu^* + 1} } . 
\end{split} \label{eqn:eta_uv} 
\end{gather} 
Note that $ \eta_u, \eta_v >0$, provided that \Cref{eqn:thr_bayes} holds. The pre-processing of $A$ in \Cref{eqn:A*_main} and the form of the spectral estimators in \Cref{eqn:def_spec} come from the derivation of a suitable AMP algorithm, and they are discussed at the end of the section. We finally defer to \Cref{sec:bulk} the definition of the scalar quantity $\sigma_2^*$ 
obtained via a fixed point equation depending only on $ \ol{\Xi}, \ol{\Sigma}, \lambda, \delta $, see  \Cref{eq:sigma2} for details.

Our main result, \Cref{thm:spec}, shows that, under the criticality condition \Cref{eqn:thr_bayes}, 
the matrix $A^*$ exhibits a spectral gap between the top two singular values, and it characterizes the performance of the spectral estimators in \Cref{eqn:def_spec}, proving that they  achieve weak recovery of $u^*$ and $v^*$, respectively. 

\begin{theorem}
\label{thm:spec}
Suppose that \Cref{eqn:thr_bayes} 
holds and that, for any $c>0$,
    \begin{align}\label{eq:cmild}
        \lim_{\beta\downarrow s} \expt{ \frac{\ol{\Sigma}^*}{\beta - c \ol{\Sigma}^*} }
        = \lim_{\beta\downarrow s} \expt{ \paren{\frac{\ol{\Sigma}^*}{ \beta - c \ol{\Sigma}^* }}^2 }
        = \infty ,\quad\lim_{\alpha\downarrow\sup\supp(\ol{\Xi}^*)} \expt{\frac{\ol{\Xi}^*}{\alpha - \ol{\Xi}^*}}
        = \infty, 
    \end{align}
    where $\ol{\Xi}^* \coloneqq \frac{\lambda}{\lambda(\mu^* + b^*) + \ol{\Xi}}$, $\ol{\Sigma}^* \coloneqq \frac{\lambda}{\lambda(\nu^* + c^*) + \ol{\Sigma}}$ and $ s \coloneqq c \cdot \sup\supp(\ol{\Sigma}^*) $. 
Let $ A^*,  \wh{u}, \wh{v}, \sigma_2^* $ be defined in \Cref{eqn:A*_main,eqn:def_spec,eq:sigma2}, and $ \sigma_i(A^*) $ denote the $i$-th largest singular value of $A^*$. 
Then, 
if $ \sigma_2^* < 1 $, the following limits hold in probability:
\begin{gather}
\begin{split}
    \lim_{n\to\infty} \sigma_1(A^*) = 1
    > \sigma_2^* = \lim_{n\to\infty} \sigma_2(A^*) , 
\end{split} \label{eqn:sing_char} \\
\begin{split}
    \lim_{n\to\infty} \frac{\abs{\inprod{\wh{u}}{u^*}}}{\normtwo{\wh{u}} \normtwo{u^*}}
    = \eta_u , \quad 
    \lim_{d\to\infty} \frac{\abs{\inprod{\wh{v}}{v^*}}}{\normtwo{\wh{v}} \normtwo{v^*}}
    = \eta_v  
\end{split} \label{eqn:ol_char} \\
\begin{split}
    \lim_{n\to\infty} \frac{1}{n^2} \normf{ u^* {u^*}^\top - \wh{u} \wh{u}^\top }^2
    = 1 - \eta_u^4 , 
\quad 
    \lim_{d\to\infty} \frac{1}{d^2} \normf{ v^* {v^*}^\top - \wh{v} \wh{v}^\top }^2
    = 1 - \eta_v^4 , 
\end{split} \label{eqn:mmse_v_char} \\
\begin{split}
    \lim_{n\to\infty} \frac{1}{nd} \normf{ u^* {v^*}^\top - \wh{u} \wh{v}^\top }^2
    = 1 - \eta_u^2 \eta_v^2 . 
\end{split} \label{eqn:mmse_uv_char}
\end{gather}
\end{theorem}

\begin{remark}[Assumptions]
\label{rk:asmp_spec}
To guarantee a spectral gap for $A^*$ and the weak recoverability of $u^*, v^*$ via the proposed spectral method, we also require the algebraic condition $ \sigma_2^* < 1 $. 
We conjecture that this condition is implied by \Cref{eqn:thr_bayes}, and we have verified that this is the case in all our numerical experiments (see \Cref{fig:sing} for two concrete examples). 
The additional assumption  \Cref{eq:cmild} is a mild regularity condition on the covariances. 
It ensures that the densities of $ \ol{\Xi}^*, \ol{\Sigma}^* $ decay sufficiently slowly at the edges of the support, so that $ \sigma_2^* $ is well-posed \cite{zhang2023spectral}. 
\end{remark}

\begin{figure}
    \centering
    \begin{subfigure}[t]{0.48\linewidth}        
        \centering
        \includegraphics[width=0.677\linewidth]{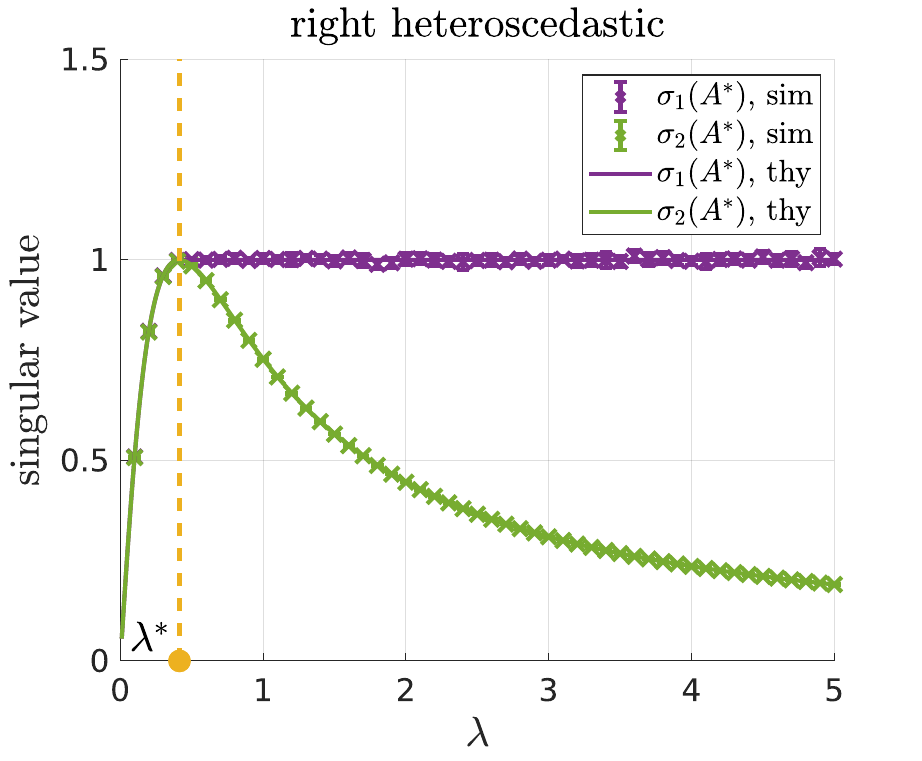}
        \caption{$ \Xi = I_n $ and $ \Sigma $ a Toeplitz matrix with $\rho = 0.9$.}
        \label{fig:fig_sing_spec_amp_bayes_1sided}
    \end{subfigure}
    \quad 
    \begin{subfigure}[t]{0.48\linewidth}
        \centering
        \includegraphics[width=0.677\linewidth]{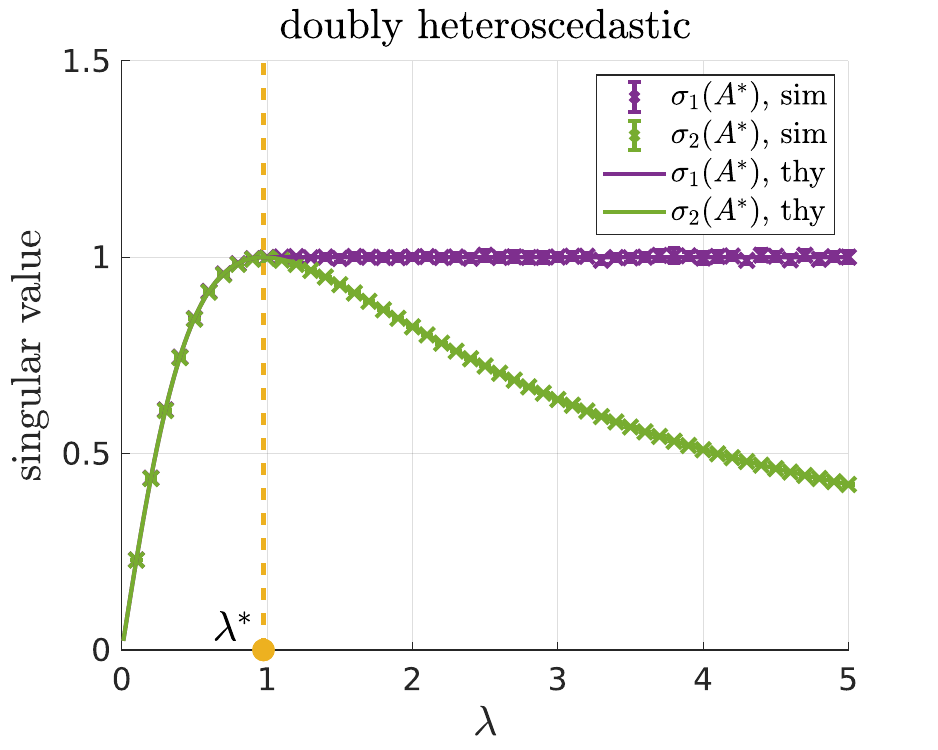}
        \caption{$ \Xi $ a circulant matrix with $ c = 0.1, \ell = 5 $ and $ \Sigma $ a Toeplitz matrix with $ \rho = 0.5 $. }
        \label{fig:fig_sing_spec_amp_bayes}
    \end{subfigure}
    \caption{Top two singular values of $ A^* $ in \Cref{eqn:A*_main}, where $ d=4000, \delta = 4 $ and each simulation is averaged over $10$ i.i.d.\ trials. The singular values computed experimentally (`sim' in the legends and $ \times $ in the plots) closely match our theoretical prediction in \Cref{eqn:sing_char} (`thy' in the legends and solid curves with the same color in the plots). The threshold $ \lambda^* $ is such that equality holds in \Cref{eqn:thr_bayes}. We note that the green curve corresponding to  $\sigma_2^*$ is smaller than $1$ for $\lambda>\lambda^*$, i.e., when \Cref{eqn:thr_bayes} holds.}
    \vspace{-.75em}
    \label{fig:sing}
\end{figure}

\begin{remark}[Signal priors]
\label{rk:spec_prior}
\Cref{thm:spec} does not require the prior distributions $P, Q$ to be Gaussian, and it is valid for any i.i.d.\ prior with mean $0$ and variance $1$. 
\end{remark}

On the one hand, \Cref{cor:weak_rec_thr} shows that, if \Cref{eqn:thr_bayes} is violated, the problem is information-theoretically impossible, i.e., no estimator achieves non-trivial error. 
On the other hand, \Cref{thm:spec} exhibits a pair of estimators that achieves non-trivial error as soon as \Cref{eqn:thr_bayes} holds -- under the additional assumption $\sigma_2^*<1$ which we conjecture to be implied by \Cref{eqn:thr_bayes}. Thus, the spectral method in \Cref{eqn:def_spec} is optimal in terms of weak recovery threshold. 
Though such estimators do not attain the optimal error, when both priors are Gaussian and $ \Xi = I_n $, $\wh{u}\wh{u}^\top$ is the Bayes-optimal estimate for $u^*{u^*}^\top$. 

\begin{corollary}
\label{cor:one_sided}
Assume $ P = Q = \cN(0,1) $, and consider the setting of \Cref{thm:spec} with the additional assumption $ \Xi = I_n $. 
Then, $ \eta_u = \sqrt{q_u^*} $, i.e., $\wh{u}\wh{u}^\top$ achieves the MMSE for $u^*{u^*}^\top$. 
\end{corollary}

The claim readily follows by noting that, when $ \Xi = I_n $, the first equation in \Cref{eqn:fp_it} becomes
\begin{align}
    q_u^* &= \frac{\alpha\snr q_v^*}{1 + \alpha\snr q_v^*}
    = \frac{(\lambda^2 / \delta) (\delta\mu^* / \lambda)}{1 + (\lambda^2 / \delta) (\delta\mu^* / \lambda)}
    = \frac{\lambda\mu^*}{1 + \lambda\mu^*}
    = \eta_u^2 , \notag 
\end{align}
where the last equality is by the definition \Cref{eqn:eta_uv} of $ \eta_u $. Let us highlight that, even if 
$ \Xi = I_n $, 
$\wh{u}$ still makes non-trivial use of the other covariance $\Sigma^{1/2}$. At the information-theoretic level, this is reflected by the fact that $\Sigma^{1/2}$ enters $q_u^*$ through the fixed point equations \Cref{eqn:fp_it}. Therefore, even though the matrix model in \Cref{eqn:Y_mtx} is equivalent to a pair of uncorrelated vector models in \Cref{eqn:Yuv} in the sense of the free energy, the tasks of estimating $u^*$ and $v^*$ cannot be decoupled. 

\begin{figure}[tb]
    \centering
    \begin{subfigure}{0.325\linewidth}
        \centering
        \includegraphics[width=\linewidth]{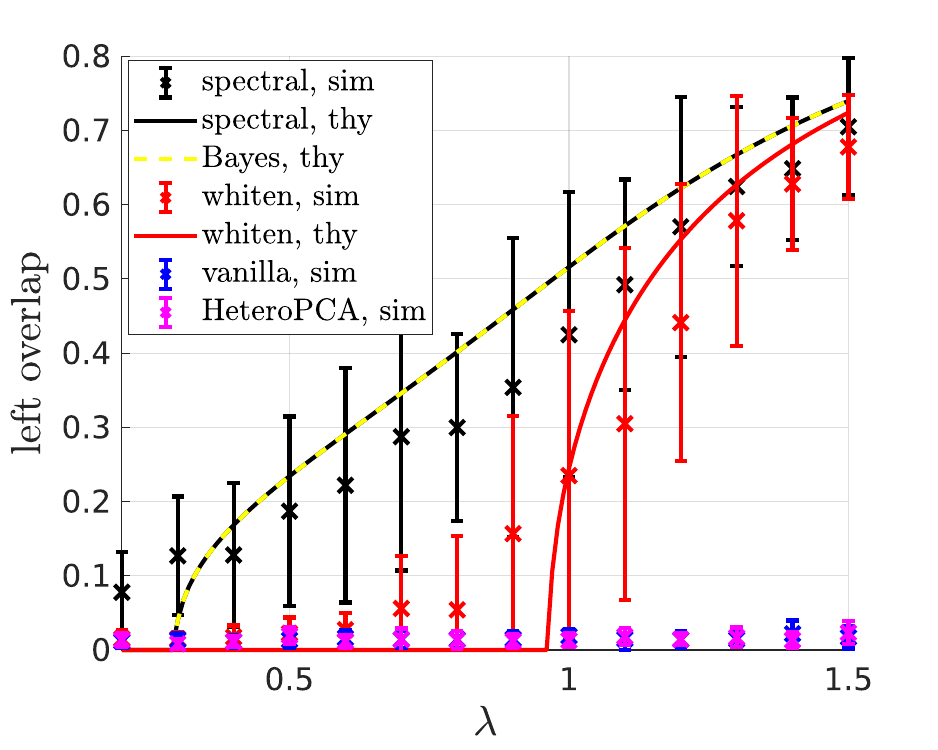}
        \caption{Normalized correlation with $u^*$}
        \label{fig:fig_compare_overlap_1sided_zoomin_left}
    \end{subfigure}
    \begin{subfigure}{0.325\linewidth}
        \centering
        \includegraphics[width=\linewidth]{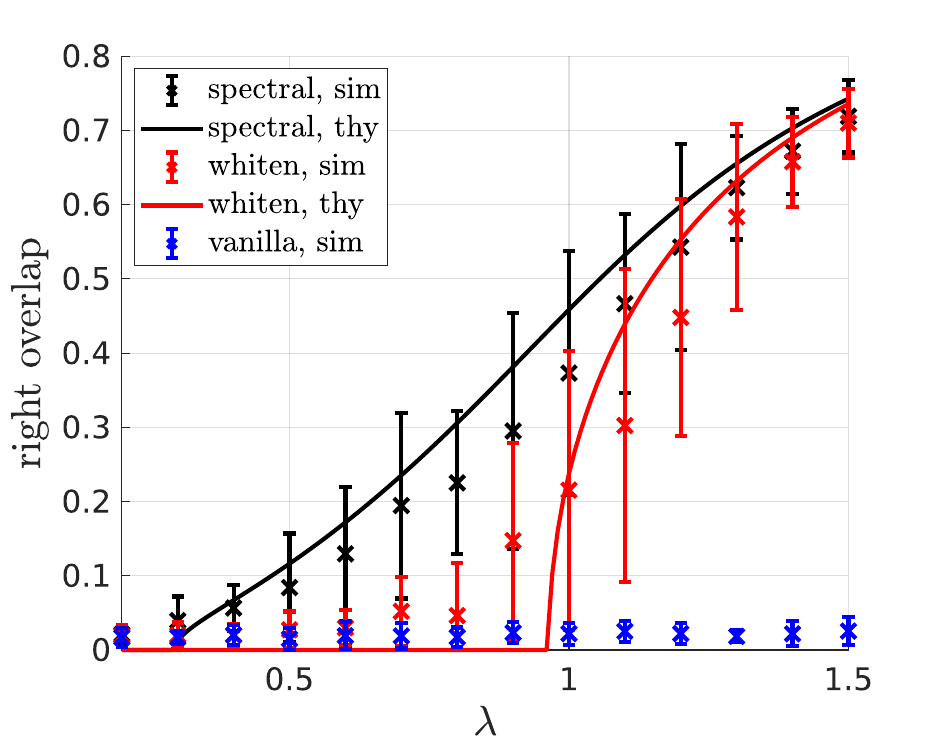}
        \caption{Normalized correlation with $v^*$}
        \label{fig:fig_compare_overlap_1sided_zoomin_right}
    \end{subfigure}
    \begin{subfigure}{0.325\linewidth}
        \centering
        \includegraphics[width=\linewidth]{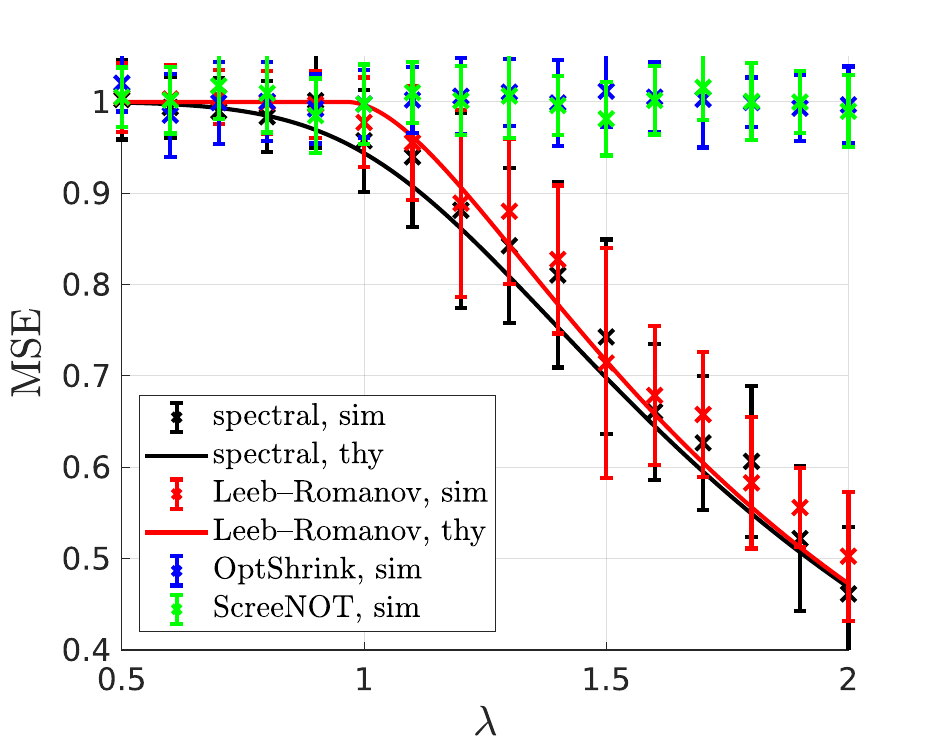}
        \caption{Matrix MSE for $u^*{v^*}^\top$}
            
\label{fig:fig_compare_1sided_zoomin}
    \end{subfigure}
   \vspace{-.25em} \caption{Performance comparison when $ \Xi = I_n $ and $ \Sigma $ is a circulant matrix. The numerical results closely follow the predictions of \Cref{thm:spec}, and our spectral estimators in \Cref{eqn:def_spec} outperform all other methods (Leeb--Romanov, OptShrink, ScreeNOT, and HeteroPCA), especially at low SNR.}
    \vspace{-1em}
    \label{fig:1sided}
\end{figure}

\begin{figure}[tb]
    \centering
    \begin{subfigure}{0.325\linewidth}
        \centering
        \includegraphics[width=\linewidth]{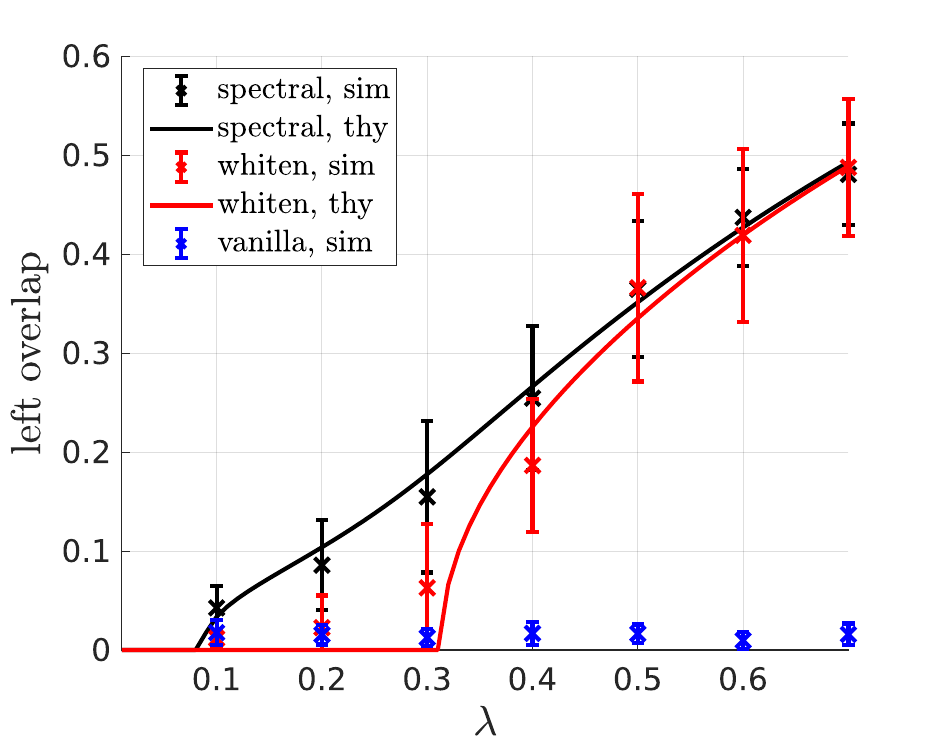}
        \caption{Normalized correlation with $u^*$}
        \label{fig:fig_compare_overlap_zoomin_left}
    \end{subfigure}
    \begin{subfigure}{0.325\linewidth}
        \centering
        \includegraphics[width=\linewidth]{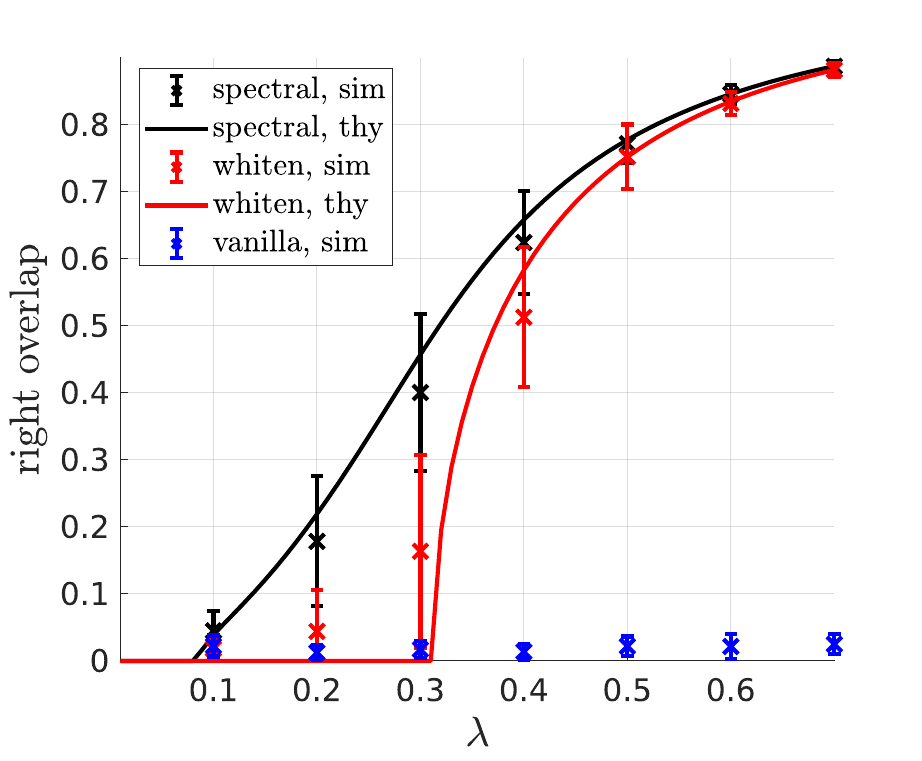}
        \caption{Normalized correlation with $v^*$}
        \label{fig:fig_compare_overlap_zoomin_right}
    \end{subfigure}
    \begin{subfigure}{0.325\linewidth}
        \centering
        \includegraphics[width=\linewidth]{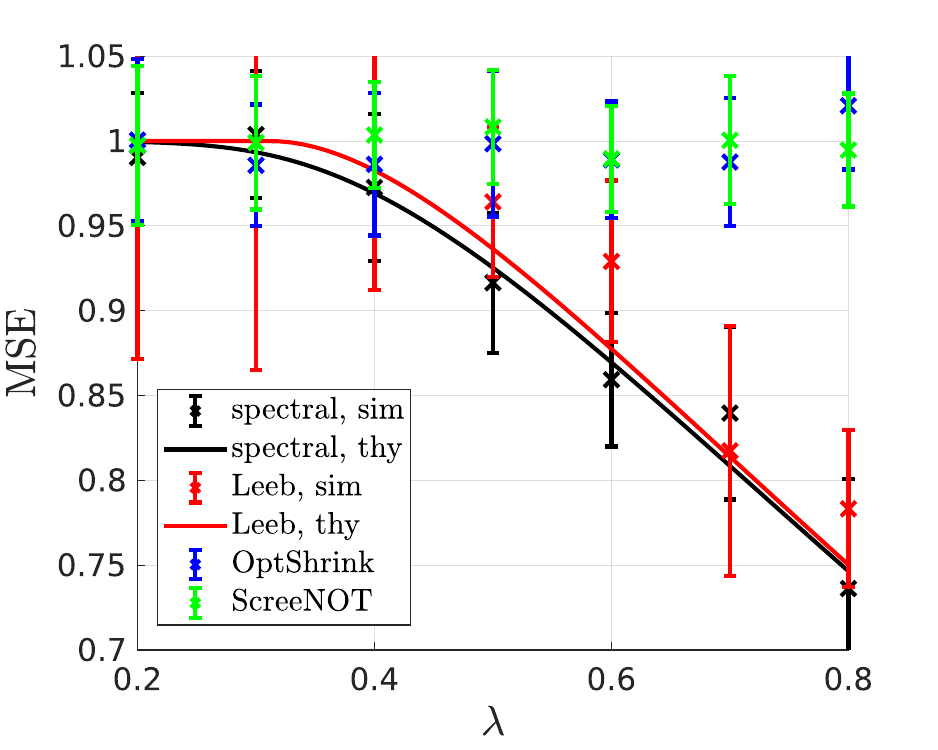}
        \caption{Matrix MSE for $u^*{v^*}^\top$}
        \label{fig:fig_compare_zoomin}
    \end{subfigure}
   \vspace{-.25em}     \caption{Performance comparison when $ \Xi $ is a Toeplitz matrix and $\Sigma$ is circulant. The numerical results closely follow the predictions of \Cref{thm:spec}, and our spectral estimators in \Cref{eqn:def_spec} outperform all other methods (Leeb, OptShrink, and ScreeNOT), especially at low SNR.}
    \vspace{-1em}
    \label{fig:2sided}
\end{figure}

\begin{figure}[htbp]
    \centering
    \begin{subfigure}[t]{\linewidth}        
        \centering
        \includegraphics[width=0.32\linewidth]{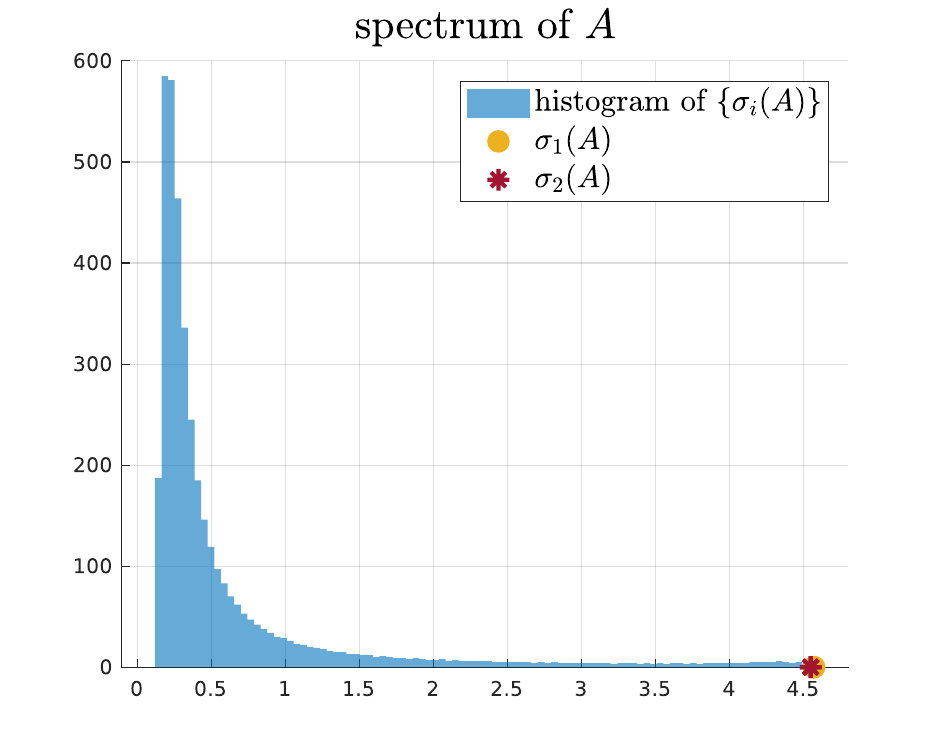}
        \qquad\qquad 
        \includegraphics[width=0.32\linewidth]{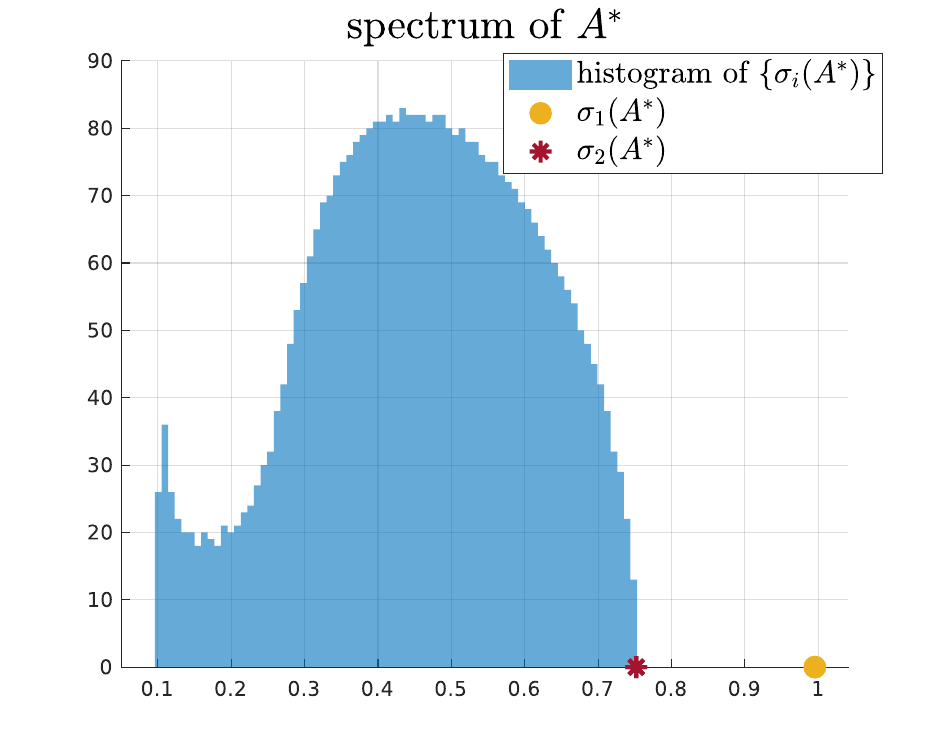}
        \caption{$\lambda = 1$, $ \Xi = I_n $ and $ \Sigma $ a Toeplitz matrix with $\rho = 0.9$.}
        \label{fig:spike_1sided_opt}
    \end{subfigure}
    \\
    \begin{subfigure}[t]{\linewidth}
        \centering
        \includegraphics[width=0.32\linewidth]{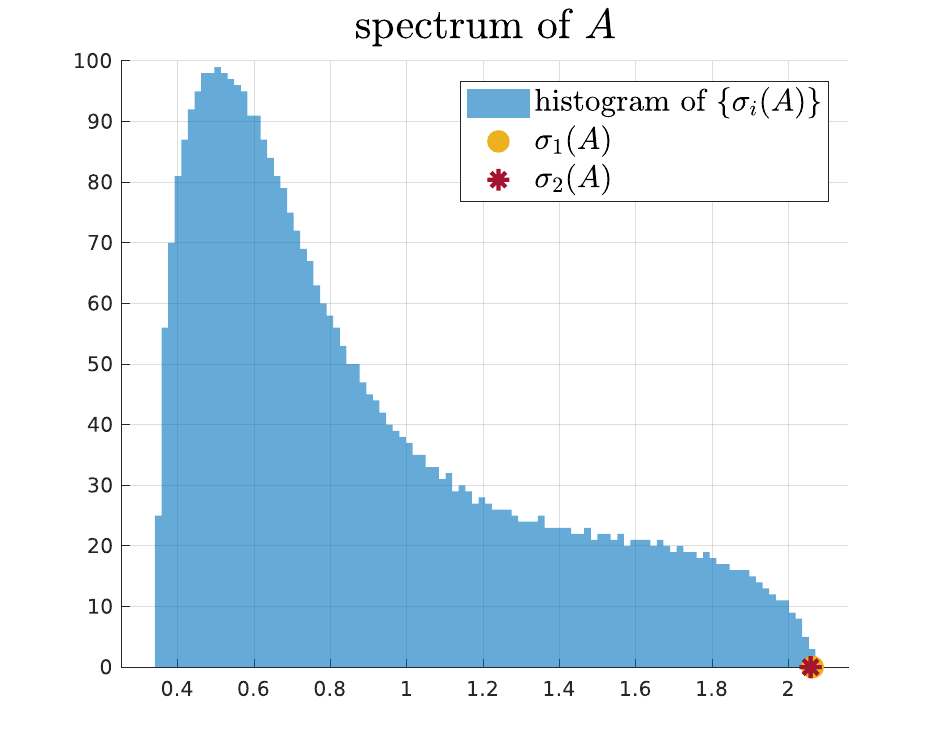}
        \qquad\qquad  
        \includegraphics[width=0.32\linewidth]{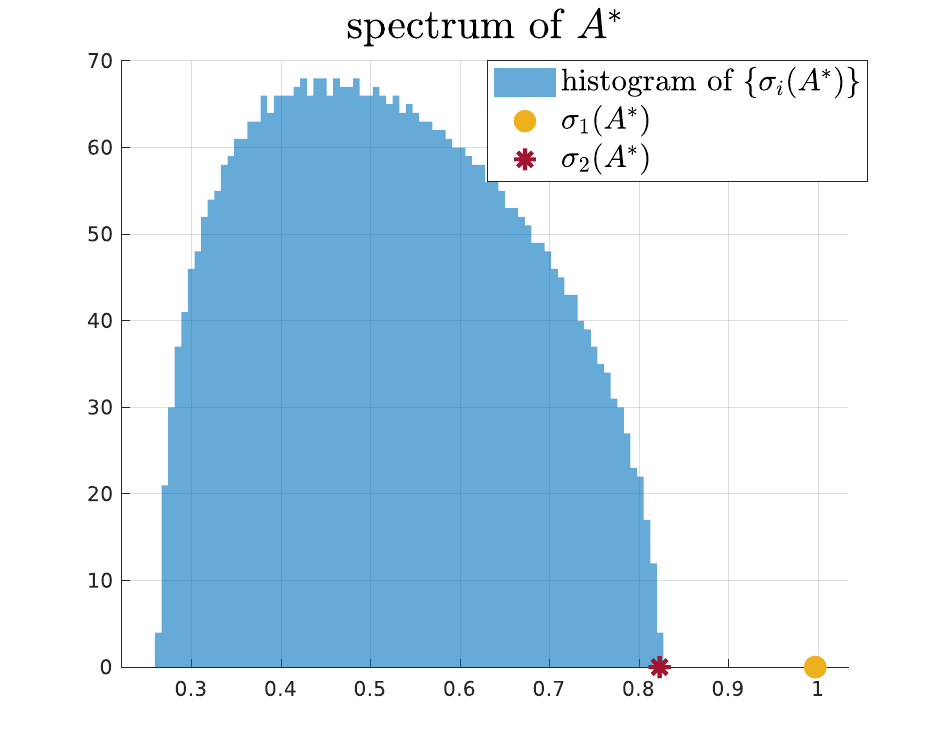}
        \caption{$\lambda = 2$, $ \Xi $ a circulant matrix with $ c = 0.1, \ell = 5 $ and $ \Sigma $ a Toeplitz matrix with $ \rho = 0.5 $. }
        \label{fig:spike_opt}
    \end{subfigure}
    \caption{Spectra of $A$ and $A^*$ averaged over $10$ i.i.d.\ trials, where $ d = 4000, \delta = 4 $. An outlier singular value emerges in the spectrum of $A^*$ due to the pre-processing on $A$. }
    \label{fig:spike}
\end{figure}

\vspace{-.5em}

\paragraph{Numerical experiments.} \Cref{fig:1sided,fig:2sided} demonstrate the advantage of our method over existing approaches, and they display an accurate agreement between simulations (`sim' in the legends and $ \times $ in the plots) 
and the theoretical predictions of \Cref{thm:spec} (`thy' in the legends and solid curves with the same color in the plots), both plotted as a function of $\lambda$. 
In both figures, $n = 4000, d = 2000$ (so $ \delta = 2 $), and $ P = Q = \cN(0,1) $. 
Each data point is computed from $20$ i.i.d.\ trials and error bars are reported at $1$ standard deviation. 
We let $ \Xi$ be either the identity or a Toeplitz matrix \cite{ZhangZhang,Javanmard_Montanari_18,Toeplitz_PTRF}, i.e., $ \Xi_{i,j} = \rho^{\abs{i - j}} $ with $\rho = 0.9$. We let $ \Sigma$ be a circulant matrix \cite{Javanmard_Montanari_14_TIT,Javanmard_Montanari_14_JMLR}: 
the first row has $1$ in the first position, $c = 0.0078$ in the second through $(\ell+1)$-st position and in the last $\ell$ positions ($\ell = 300$), with the remaining entries being  $0$; for $2\le i\le d$, the $i$-th row is a cyclic shift of the $(i-1)$-st row to the right by $1$ position. 
Both matrices satisfy \Cref{eq:cmild} and the conditions of \Cref{sec:prelim}. 
    
Our spectral estimator outperforms all other approaches: Leeb--Romanov \cite{PCA_hetero_shrink}, OptShrink \cite{Nadakuditi_shrinkage_rotinv}, ScreeNOT \cite{DGR_ScreeNOT}, and HeteroPCA \cite{HeteroPCA} in the one-sided heteroscedastic case (\Cref{fig:1sided}); Leeb \cite{Leeb_mtx_denois}, OptShrink, and ScreeNOT in the doubly heteroscedastic case (\Cref{fig:2sided}). 
When computing the normalized correlation with the signals (left/right overlap), the performance of Leeb--Romanov and Leeb is the same as the estimators $ \Xi^{1/2} u_1(\Xi^{-1/2} A \Sigma^{-1/2}), \Sigma^{1/2} v_1(\Xi^{-1/2} A \Sigma^{-1/2}) $ referred to as `whiten' in \Cref{fig:fig_compare_overlap_1sided_zoomin_left,fig:fig_compare_overlap_1sided_zoomin_right}; 
the performance of OptShrink and ScreeNOT is the same as the estimators $ u_1(A), v_1(A) $ referred to as `vanilla' in \Cref{fig:fig_compare_overlap_zoomin_left,fig:fig_compare_overlap_zoomin_right}. The advantage of our approach (in black) is especially significant at low SNR; as SNR increases, Leeb-Romanov and Leeb (in red) achieve similar performance; a much larger SNR ($>2$ and $>3$ in \Cref{fig:1sided,fig:2sided}) is required by HeteroPCA, OptShrink and ScreeNOT (in magenta, blue and green) to perform comparably.

Finally, \Cref{fig:spike} shows the presence of spectral outliers in $A^*$ and their absence in $A$ at a fixed $\lambda$. 

\vspace{-.5em}

\paragraph{Proof strategy.} 
The design and analysis of the spectral estimator in \Cref{eqn:def_spec} comprise two steps, detailed in \Cref{app:analysis_spec}. The \emph{first step} is to present an AMP algorithm dubbed Bayes-AMP for matrix denoising with doubly heteroscedastic noise. 
Specifically, 
its iterates are updated as 
\vspace{-.5em}
\begin{align}
u^t &= \Xi^{-1} A \Sigma^{-1} \wt{v}^t - b_t \Xi^{-1} \wt{u}^{t-1} , \quad 
\wt{u}^t = g_t^*(u^t) , \quad 
c_t = \frac{1}{n} \tr((\nabla g_t^*(u^t)) \Xi^{-1}) ,  \label{eqn:BAMPit} 
\\
v^{t+1} &= \Sigma^{-1} A^\top \Xi^{-1} \wt{u}^t - c_t \Sigma^{-1} \wt{v}^t , \quad
\wt{v}^{t+1} = f_{t+1}^*(v^{t+1}) , \quad
b_{t+1} = \frac{1}{n} \tr((\nabla f_{t+1}^*(v^{t+1})) \Sigma^{-1}) , \notag
\end{align}
where $\nabla$ denotes the Jacobian matrix, and the functions $ g_t^*, f_{t+1}^* $ are specified below in \Cref{eq:conddenoi}. As common in AMP algorithms, the iterates \Cref{eqn:BAMPit} are accompanied with a state evolution which accurately tracks their behavior via a simple deterministic recursion: the joint empirical distribution of $(u^*, v^*, u^t, v^{t+1})$ converges to the random variables $(U^*, V^*, U_t, V_{t+1})$, see  \Cref{prop:SE_bayes} for a formal statement and the recursive description of the laws of such random variables. Then, the name `Bayes-AMP' is motivated by the fact that $ g_t^*, f_{t+1}^* $ are the posterior-mean denoisers given by
\vspace{-.25em}
\begin{align}\label{eq:conddenoi}
g_t^*(u) &\coloneqq \expt{ U^* \mid U_t = u } , \qquad 
f_{t+1}^*(v) \coloneqq \expt{ V^* \mid V_{t+1} = v } . 
\end{align}
Remarkably, Bayes-AMP operates on $ \Xi^{-1} A \Sigma^{-1} $, as opposed to the widely adopted ansatz of considering the whitened matrix $ 
\Xi^{-1/2} A \Sigma^{-1/2} $. The advantage of operating on $ \Xi^{-1} A \Sigma^{-1} $ is that the fixed point of the corresponding state evolution matches the extremizers of the free energy in \Cref{eqn:fp_it}. This would \emph{not} be the case if Bayes-AMP used the whitening $ \Xi^{-1/2} A \Sigma^{-1/2} $. 
Indeed, one can repeat the analysis of an AMP that operates on $\Xi^{-1/2} A \Sigma^{-1/2}$. The fixed point equations of the resulting state evolution do not match the information-theoretically optimal one in \Cref{eqn:fp_it}. In particular, the weak recovery threshold coming out of this approach is strictly larger than the optimal one in \Cref{eqn:thr_bayes_change}, as long as at least one of $\Xi, \Sigma$ is not a multiple of the identity. Since these derivations led to suboptimal results, the details were left out from the paper.

The design of Bayes-AMP and the proof of its state evolution follow a two-step reduction detailed in \Cref{app:pf_BAMP_SE}. 
Using a change of variables, we show in \Cref{app:pf_SE_bayes}  that Bayes-AMP can be realized by an auxiliary AMP with non-separable denoising functions (meaning that $ g_t, f_{t+1} $ cannot be written as univariate functions applied component-wise) operating on $ \Xi^{-1/2} A \Sigma^{-1/2} = \frac{\lambda}{n} \wt{u}^* (\wt{v}^*)^\top + \wt{W} $. 
Then, in \Cref{sec:AMP} we simulate the auxiliary AMP using a standard AMP operating on the i.i.d.\ Gaussian matrix $ \wt{W} $, whose state evolution has been established in \cite{AMP_nonsep,GraphAMP}. 

However, Bayes-AMP by itself is not a practical algorithm since it needs a warm start, i.e., an initialization that achieves non-trivial error. Thus, the \emph{second step} is to design a spectral estimator that solves the fixed point equation of Bayes-AMP, which turns out to be an eigen-equation for $A^*$.

To offer the readers an intuition on how the spectral estimators arise from Bayes-AMP, 
we now heuristically derive the form \Cref{eqn:A*_main} of $A^*$ and the expression \Cref{eqn:def_spec} of the spectral estimator. To do so, we note that the large-$n$ limits of $ c_t, b_{t+1} $ coincide with the auxiliary quantities $c^*, b^*$ defined in \Cref{eq:aux}. Furthermore, when the priors of $u^*, v^*$ are Gaussian, \Cref{eq:conddenoi} reduces to 
$$ g_t^*(u) = \lambda (\lambda \mu^* \Xi^{-1} + I_n)^{-1} u , \qquad f_{t+1}^*(v) = \lambda (\lambda \nu^* \Sigma^{-1} + I_d)^{-1} v, $$
where we recall that $\mu^* = \lambda q_v^* / \delta$ and $\nu^* = \lambda q_u^*$ are rescalings of the non-trivial solution $(q_u^*, q_v^*)$ of \Cref{eqn:fp_it}. Denoting by $u,v$ the fixed points of the iteration \Cref{eqn:BAMPit}, after some manipulations we have
\vspace{-.5em}
\begin{equation*}
    \mathfrak{g}(\Xi) u = A^* \mathfrak{f}(\Sigma) v , \qquad
 \mathfrak{f}(\Sigma) v = {A^*}^\top \mathfrak{g}(\Xi) u,
\end{equation*}  
where $A^*$ is given in \Cref{eqn:A*_main} and
\vspace{-.5em}
 \begin{align}
     \mathfrak{g}(\Xi) &\coloneqq \sqrt{\lambda} (\lambda(\mu^* + b^*) I_n + \Xi)^{1/2} (\lambda\mu^* I_n + \Xi)^{-1} \Xi^{1/2} , \notag \\
     \mathfrak{f}(\Sigma) &\coloneqq \sqrt{\lambda} (\lambda(\nu^* + c^*) I_d + \Sigma)^{1/2} (\lambda\nu^* I_d + \Sigma)^{-1} \Sigma^{1/2} . 
     \notag 
 \end{align}
This suggests that $ A^* $ has top singular value equal to $1$ and $ (\mathfrak{g}(\Xi) u, \mathfrak{f}(\Sigma) v) $ are aligned with the corresponding singular vectors $(u_1(A^*), v_1(A^*))$. 
Moreover, state evolution implies that the distribution of the fixed point $(u,v)$ is close to that of $$  ( \mu^* \Xi^{-1} u^* + \sqrt{\mu^*/\lambda} w_u, \nu^* \Sigma^{-1} v^* + \sqrt{\nu^*/\lambda} w_v ) , $$ with $ (w_u, w_v) \sim \cN(0_n, \Xi^{-1}) \otimes \cN(0_d, \Sigma^{-1}) $ independent of $ u^*,v^* $. 
Thus, to obtain estimates of $(u^*, v^*)$, we take $ (\Xi \mathfrak{g}(\Xi)^{-1} u_1(A^*), \Sigma \mathfrak{f}(\Sigma)^{-1} v_1(A^*) )$ and suitably rescale their norm, which leads to the expressions in \Cref{eqn:def_spec}. 
More details on the above heuristics are discussed in \Cref{sec:AMP_spec}.

The most outstanding step remains to make the heuristics rigorous. This involves proving that $ \Xi u^t, \Sigma v^{t+1} $ are aligned with the proposed spectral estimator, which allows for a performance characterization via state evolution. The formal argument is carried out in  \Cref{app:pf_spec}. 


\section{Concluding remarks}
\label{sec:dis}

In this work, we establish  information-theoretic limits and propose an efficient spectral method with optimality guarantees, for matrix estimation with doubly heteroscedastic noise.
On the one hand, under Gaussian priors, we give a rigorous characterization of the MMSE; on the other hand, we present a spectral estimator that \emph{(i)} achieves the information-theoretic weak recovery threshold, and \emph{(ii)} is Bayes-optimal for the estimation of one of the signals, when the noise is heteroscedastic only on the other side. While our analysis focuses on rank-1 estimation, we expect that all results admit proper extensions to rank-$r$ signals, where $r$ is a constant independent on $n, d$.

The design and analysis of the spectral estimator draws connections with approximate message passing and, along the way, we introduce a Bayes-AMP algorithm which could be of independent interest. In this paper, we employ Bayes-AMP solely as a proof technique. However, one could use the spectral method designed here as an initialization of Bayes-AMP itself, after suitably correcting its iterates. This strategy has been successfully carried out for i.i.d.\ Gaussian noise in  \cite{MV_AoS} and for rotationally invariant noise  in \cite{mondelli2021pca,zhong2021approximate}. Bayes-AMP is well equipped to exploit signal priors  more informative than the Gaussian one, and AMP algorithms are known to achieve the information-theoretically optimal estimation error for low-rank matrix inference \cite{MV_AoS,Barbier_PCA_rotinv}. Nevertheless, we point out two obstacles towards doing so in the presence of doubly heteroscedastic noise. First, for general priors, establishing the information-theoretic limits remains a challenging open problem, and it is unclear whether a low-dimensional characterization of the free energy (and, hence, of the MMSE) is 
possible. Second, even for Gaussian priors, Bayes-AMP reduces to the proposed spectral estimator, which is not Bayes-optimal for the general case of doubly heteroscedastic noise. 




Finally, the proposed spectral estimator makes non-trivial use of the covariances $\Xi, \Sigma$, which are assumed to be known. 
When such matrices possess additional structure -- 
e.g., they are sparse \cite{CaiZhou_sparse}, their inverses are sparse \cite{Toeplitz_PTRF} or they are circulant or Toeplitz \cite{Yuan_cov} -- their consistent estimation is possible, see also the survey \cite{CaiRenZhou}. 
However, in general, $\Xi, \Sigma$ cannot be consistently estimated from the data when $n$ and $d$ grow proportionally. 
Thus, a challenging open problem is to construct estimators that retain comparable performance without knowing the noise covariances. 
The paper \cite{mtx_denois_partial} addresses the challenge of unknown covariances by considering a modified model where one additionally observes an independent copy of noise. The statistician can then estimate the covariance from the noise-only observation and use it as a surrogate of the true covariance for estimating the signals from the spiked model. It is possible to derive similar results in the doubly heteroskedastic setting considered in our paper.
If the covariances are completely unknown, then our model (with Gaussian priors) is equivalent to a spiked matrix model with a certain bi-rotationally invariant noise. This problem is expected to exhibit rather different behaviors than when covariances are known, see \cite{PCA_rotinv_IT,PCA_rotinv_OAMP} for recent progress on understanding the statistical and computational limits for such models.







\newpage

\begin{ack}
YZ thanks Shashank Vatedka for discussions at the early stage of this project. 
MM thanks Jean Barbier for sharing his insights into the interpolation argument. 
This research is partially supported by the 2019 Lopez-Loreta Prize and by the Interdisciplinary Projects Committee (IPC) at the Institute of Science and Technology Austria (ISTA). 
This work was done in part while the authors were visiting the Simons Institute for the Theory of Computing.


\end{ack}

\bibliographystyle{plain}
\bibliography{ref} 

\begin{thebibliography}{10}

\bibitem{Agterberg_low-rank_hetero}
Joshua Agterberg, Zachary Lubberts, and Carey~E. Priebe.
\newblock Entrywise estimation of singular vectors of low-rank matrices with
  heteroskedasticity and dependence.
\newblock {\em IEEE Trans. Inform. Theory}, 68(7):4618--4650, 2022.

\bibitem{ASS}
Michael Aizenman, Robert Sims, and Shannon~L. Starr.
\newblock Extended variational principle for the sherrington-kirkpatrick
  spin-glass model.
\newblock {\em Phys. Rev. B}, 68:214403, Dec 2003.

\bibitem{BaiYin}
Z.~D. Bai and Y.~Q. Yin.
\newblock Limit of the smallest eigenvalue of a large-dimensional sample
  covariance matrix.
\newblock {\em Ann. Probab.}, 21(3):1275--1294, 1993.

\bibitem{BBAP}
Jinho Baik, G\'{e}rard Ben~Arous, and Sandrine P\'{e}ch\'{e}.
\newblock Phase transition of the largest eigenvalue for nonnull complex sample
  covariance matrices.
\newblock {\em Ann. Probab.}, 33(5):1643--1697, 2005.

\bibitem{astronomy2}
Stephen Bailey.
\newblock Principal component analysis with noisy and/or missing data.
\newblock {\em Publications of the Astronomical Society of the Pacific},
  124(919):1015, sep 2012.

\bibitem{Barbier_PCA_rotinv}
Jean Barbier, Francesco Camilli, Marco Mondelli, and Manuel S\'{a}enz.
\newblock Fundamental limits in structured principal component analysis and how
  to reach them.
\newblock {\em Proc. Natl. Acad. Sci. USA}, 120(30):Paper No. e2302028120, 7,
  2023.

\bibitem{PCA_rotinv_IT}
Jean Barbier, Francesco Camilli, Marco Mondelli, and Yizhou Xu.
\newblock Information limits and thouless-anderson-palmer equations for spiked
  matrix models with structured noise.
\newblock {\em CoRR}, abs/2405.20993, 2024.

\bibitem{barbier2022price}
Jean Barbier, TianQi Hou, Marco Mondelli, and Manuel Saenz.
\newblock The price of ignorance: how much does it cost to forget noise
  structure in low-rank matrix estimation?
\newblock In {\em Advances in Neural Information Processing Systems},
  volume~35, pages 36733--36747, 2022.

\bibitem{Barbier_GLM}
Jean Barbier, Florent Krzakala, Nicolas Macris, L\'{e}o Miolane, and Lenka
  Zdeborov\'{a}.
\newblock Optimal errors and phase transitions in high-dimensional generalized
  linear models.
\newblock {\em Proc. Natl. Acad. Sci. USA}, 116(12):5451--5460, 2019.

\bibitem{BM_interp}
Jean Barbier and Nicolas Macris.
\newblock The adaptive interpolation method: a simple scheme to prove replica
  formulas in {B}ayesian inference.
\newblock {\em Probab. Theory Related Fields}, 174(3-4):1133--1185, 2019.

\bibitem{Bayati_Montanari}
Mohsen Bayati and Andrea Montanari.
\newblock The dynamics of message passing on dense graphs, with applications to
  compressed sensing.
\newblock {\em IEEE Trans. Inform. Theory}, 57(2):764--785, 2011.

\bibitem{IT_inhomo}
Joshua~K. Behne and Galen Reeves.
\newblock Fundamental limits for rank-one matrix estimation with groupwise
  heteroskedasticity.
\newblock In {\em International Conference on Artificial Intelligence and
  Statistics}, pages 8650--8672, 2022.

\bibitem{AMP_nonsep}
Rapha\"{e}l Berthier, Andrea Montanari, and Phan-Minh Nguyen.
\newblock State evolution for approximate message passing with non-separable
  functions.
\newblock {\em Inf. Inference}, 9(1):33--79, 2020.

\bibitem{biology2}
Tejal Bhamre, Teng Zhang, and Amit Singer.
\newblock Denoising and covariance estimation of single particle cryo-em
  images.
\newblock {\em Journal of Structural Biology}, 195(1):72--81, 2016.

\bibitem{Bolthausen}
Erwin Bolthausen.
\newblock An iterative construction of solutions of the {TAP} equations for the
  {S}herrington-{K}irkpatrick model.
\newblock {\em Comm. Math. Phys.}, 325(1):333--366, 2014.

\bibitem{BLM_book}
St\'{e}phane Boucheron, G\'{a}bor Lugosi, and Pascal Massart.
\newblock {\em Concentration inequalities}.
\newblock Oxford University Press, Oxford, 2013.

\bibitem{wishart_nonasymp_hetero}
T.~Tony Cai, Rungang Han, and Anru~R. Zhang.
\newblock On the non-asymptotic concentration of heteroskedastic {W}ishart-type
  matrix.
\newblock {\em Electron. J. Probab.}, 27:Paper No. 29, 40, 2022.

\bibitem{SparsePCA}
T.~Tony Cai, Zongming Ma, and Yihong Wu.
\newblock Sparse {PCA}: optimal rates and adaptive estimation.
\newblock {\em Ann. Statist.}, 41(6):3074--3110, 2013.

\bibitem{Toeplitz_PTRF}
T.~Tony Cai, Zhao Ren, and Harrison~H. Zhou.
\newblock Optimal rates of convergence for estimating {T}oeplitz covariance
  matrices.
\newblock {\em Probab. Theory Related Fields}, 156(1-2):101--143, 2013.

\bibitem{CaiRenZhou}
T.~Tony Cai, Zhao Ren, and Harrison~H. Zhou.
\newblock Estimating structured high-dimensional covariance and precision
  matrices: optimal rates and adaptive estimation.
\newblock {\em Electron. J. Stat.}, 10(1):1--59, 2016.

\bibitem{CaiZhou_sparse}
T.~Tony Cai and Harrison~H. Zhou.
\newblock Optimal rates of convergence for sparse covariance matrix estimation.
\newblock {\em Ann. Statist.}, 40(5):2389--2420, 2012.

\bibitem{DMFT}
Michael Celentano, Chen Cheng, and Andrea Montanari.
\newblock The high-dimensional asymptotics of first order methods with random
  data.
\newblock {\em arXiv preprint arXiv:2112.07572}, 2021.

\bibitem{CWC_small}
Chen Cheng, Yuting Wei, and Yuxin Chen.
\newblock Tackling small eigen-gaps: fine-grained eigenvector estimation and
  inference under heteroscedastic noise.
\newblock {\em IEEE Trans. Inform. Theory}, 67(11):7380--7419, 2021.

\bibitem{imaging1}
Lucilio Cordero-Grande, Daan Christiaens, Jana Hutter, Anthony~N. Price, and
  Jo~V. Hajnal.
\newblock Complex diffusion-weighted image estimation via matrix recovery under
  general noise models.
\newblock {\em NeuroImage}, 200:391--404, 2019.

\bibitem{CouilletHachem}
Romain Couillet and Walid Hachem.
\newblock Analysis of the limiting spectral measure of large random matrices of
  the separable covariance type.
\newblock {\em Random Matrices Theory Appl.}, 3(4):1450016, 23, 2014.

\bibitem{DLY_shrinkage}
Xiucai Ding, Yun Li, and Fan Yang.
\newblock Eigenvector distributions and optimal shrinkage estimators for large
  covariance and precision matrices.
\newblock {\em arXiv preprint arXiv:2404.14751}, 2024.

\bibitem{DGR_ScreeNOT}
David Donoho, Matan Gavish, and Elad Romanov.
\newblock {\it {S}cree{NOT}}: exact {MSE}-optimal singular value thresholding
  in correlated noise.
\newblock {\em Ann. Statist.}, 51(1):122--148, 2023.

\bibitem{DMM09}
David~L. Donoho, Arian Maleki, and Andrea Montanari.
\newblock {Message passing algorithms for compressed sensing}.
\newblock {\em Proceedings of the National Academy of Sciences},
  106:18914--18919, 2009.

\bibitem{PCA_rotinv_OAMP}
Rishabh Dudeja, Songbin Liu, and Junjie Ma.
\newblock Optimality of approximate message passing algorithms for spiked
  matrix models with rotationally invariant noise.
\newblock {\em CoRR}, abs/2405.18081, 2024.

\bibitem{spec_univ}
Rishabh Dudeja, Subhabrata Sen, and Yue~M Lu.
\newblock Spectral universality of regularized linear regression with nearly
  deterministic sensing matrices.
\newblock {\em IEEE Transactions on Information Theory}, 2024.

\bibitem{fan2020approximate}
Zhou Fan.
\newblock Approximate message passing algorithms for rotationally invariant
  matrices.
\newblock {\em The Annals of Statistics}, 50(1):197--224, 2022.

\bibitem{amp-tutorial}
Oliver~Y Feng, Ramji Venkataramanan, Cynthia Rush, Richard~J Samworth, et~al.
\newblock A unifying tutorial on approximate message passing.
\newblock {\em Foundations and Trends{\textregistered} in Machine Learning},
  15(4):335--536, 2022.

\bibitem{mtx_denois_partial}
Matan Gavish, William Leeb, and Elad Romanov.
\newblock Matrix denoising with partial noise statistics: optimal singular
  value shrinkage of spiked {F}-matrices.
\newblock {\em Inf. Inference}, 12(3):Paper No. iaad028, 46, 2023.

\bibitem{GraphAMP}
C\'{e}dric Gerbelot and Rapha\"{e}l Berthier.
\newblock Graph-based approximate message passing iterations.
\newblock {\em Inf. Inference}, 12(4):Paper No. iaad020, 67, 2023.

\bibitem{Guerra}
Francesco Guerra.
\newblock Broken replica symmetry bounds in the mean field spin glass model.
\newblock {\em Comm. Math. Phys.}, 233(1):1--12, 2003.

\bibitem{IT_inhomo_univ}
Alice Guionnet, Justin Ko, Florent Krzakala, and Lenka Zdeborov{\'a}.
\newblock Low-rank matrix estimation with inhomogeneous noise.
\newblock {\em arXiv preprint arXiv:2208.05918}, 2022.

\bibitem{ODE_book}
Philip Hartman.
\newblock {\em Ordinary differential equations}, volume~38.
\newblock Society for Industrial and Applied Mathematics (SIAM), Philadelphia,
  PA, 2002.

\bibitem{WPCA}
David Hong, Fan Yang, Jeffrey~A. Fessler, and Laura Balzano.
\newblock Optimally weighted {PCA} for high-dimensional heteroscedastic data.
\newblock {\em SIAM J. Math. Data Sci.}, 5(1):222--250, 2023.

\bibitem{Javanmard_Montanari_14_JMLR}
Adel Javanmard and Andrea Montanari.
\newblock Confidence intervals and hypothesis testing for high-dimensional
  regression.
\newblock {\em J. Mach. Learn. Res.}, 15:2869--2909, 2014.

\bibitem{Javanmard_Montanari_14_TIT}
Adel Javanmard and Andrea Montanari.
\newblock Hypothesis testing in high-dimensional regression under the
  {G}aussian random design model: asymptotic theory.
\newblock {\em IEEE Trans. Inform. Theory}, 60(10):6522--6554, 2014.

\bibitem{Javanmard_Montanari_18}
Adel Javanmard and Andrea Montanari.
\newblock Debiasing the {L}asso: optimal sample size for {G}aussian designs.
\newblock {\em Ann. Statist.}, 46(6A):2593--2622, 2018.

\bibitem{Johnstone_spiked}
Iain~M. Johnstone.
\newblock On the distribution of the largest eigenvalue in principal components
  analysis.
\newblock {\em Ann. Statist.}, 29(2):295--327, 2001.

\bibitem{biwhitening}
Boris Landa, Thomas T. C.~K. Zhang, and Yuval Kluger.
\newblock Biwhitening reveals the rank of a count matrix.
\newblock {\em SIAM J. Math. Data Sci.}, 4(4):1420--1446, 2022.

\bibitem{PCA_hetero_shrink}
William Leeb and Elad Romanov.
\newblock Optimal spectral shrinkage and {PCA} with heteroscedastic noise.
\newblock {\em IEEE Trans. Inform. Theory}, 67(5):3009--3037, 2021.

\bibitem{Leeb_mtx_denois}
William~E. Leeb.
\newblock Matrix denoising for weighted loss functions and heterogeneous
  signals.
\newblock {\em SIAM J. Math. Data Sci.}, 3(3):987--1012, 2021.

\bibitem{bioilogy3}
Jeffrey~T. Leek.
\newblock Asymptotic conditional singular value decomposition for
  high-dimensional genomic data.
\newblock {\em Biometrics}, 67(2):344--352, 2011.

\bibitem{Bayes_conj}
Thibault Lesieur, Florent Krzakala, and Lenka Zdeborová.
\newblock Mmse of probabilistic low-rank matrix estimation: Universality with
  respect to the output channel.
\newblock In {\em 2015 53rd Annual Allerton Conference on Communication,
  Control, and Computing (Allerton)}, pages 680--687, 2015.

\bibitem{Householder}
Yue~M. Lu.
\newblock Householder dice: a matrix-free algorithm for simulating dynamics on
  {G}aussian and random orthogonal ensembles.
\newblock {\em IEEE Trans. Inform. Theory}, 67(12):8264--8272, 2021.

\bibitem{mutual_info_ten}
Cl\'{e}ment Luneau, Jean Barbier, and Nicolas Macris.
\newblock Mutual information for low-rank even-order symmetric tensor
  estimation.
\newblock {\em Inf. Inference}, 10(4):1167--1207, 2021.

\bibitem{spec_inhomo}
Pierre Mergny, Justin Ko, and Florent Krzakala.
\newblock Spectral phase transition and optimal pca in block-structured spiked
  models.
\newblock {\em arXiv preprint arXiv:2403.03695}, 2024.

\bibitem{envelope}
Paul Milgrom and Ilya Segal.
\newblock Envelope theorems for arbitrary choice sets.
\newblock {\em Econometrica}, 70(2):583--601, 2002.

\bibitem{Miolane_asymm}
L{\'e}o Miolane.
\newblock Fundamental limits of low-rank matrix estimation: the non-symmetric
  case.
\newblock {\em arXiv preprint arXiv:1702.00473}, 2017.

\bibitem{miolane_thesis}
L{\'e}o Miolane.
\newblock {\em {Fundamental limits of inference: A statistical physics
  approach.}}
\newblock Theses, {Ecole normale sup{\'e}rieure - ENS PARIS ; Inria Paris},
  June 2019.

\bibitem{mondelli2021optimalcombination}
Marco Mondelli, Christos Thrampoulidis, and Ramji Venkataramanan.
\newblock Optimal combination of linear and spectral estimators for generalized
  linear models.
\newblock {\em Foundations of Computational Mathematics}, pages 1--54, 2021.

\bibitem{mondelli2021pca}
Marco Mondelli and Ramji Venkataramanan.
\newblock Pca initialization for approximate message passing in rotationally
  invariant models.
\newblock In {\em Advances in Neural Information Processing Systems},
  volume~34, pages 29616--29629, 2021.

\bibitem{MV_AoS}
Andrea Montanari and Ramji Venkataramanan.
\newblock Estimation of low-rank matrices via approximate message passing.
\newblock {\em Ann. Statist.}, 49(1):321--345, 2021.

\bibitem{MW_diverging}
Andrea Montanari and Yuchen Wu.
\newblock Fundamental limits of low-rank matrix estimation with diverging
  aspect ratios.
\newblock {\em arXiv preprint arXiv:2211.00488}, 2022.

\bibitem{sampling_spiked}
Andrea Montanari and Yuchen Wu.
\newblock Posterior sampling from the spiked models via diffusion processes.
\newblock {\em arXiv preprint arXiv:2304.11449}, 2023.

\bibitem{Nadakuditi_shrinkage_rotinv}
Raj~Rao Nadakuditi.
\newblock Opt{S}hrink: an algorithm for improved low-rank signal matrix
  denoising by optimal, data-driven singular value shrinkage.
\newblock {\em IEEE Trans. Inform. Theory}, 60(5):3002--3018, 2014.

\bibitem{Nadler_finite_sample}
Boaz Nadler.
\newblock Finite sample approximation results for principal component analysis:
  a matrix perturbation approach.
\newblock {\em Ann. Statist.}, 36(6):2791--2817, 2008.

\bibitem{AMP_inhomo}
Aleksandr Pak, Justin Ko, and Florent Krzakala.
\newblock Optimal algorithms for the inhomogeneous spiked wigner model.
\newblock In {\em Advances in Neural Information Processing Systems},
  volume~36, pages 76409--76424, 2023.

\bibitem{Panchenko_book}
Dmitry Panchenko.
\newblock {\em The {S}herrington-{K}irkpatrick model}.
\newblock Springer Monographs in Mathematics. Springer, New York, 2013.

\bibitem{imaging3}
Henrik Pedersen, Sebastian Kozerke, Steffen Ringgaard, Kay Nehrke, and Won~Yong
  Kim.
\newblock k-t pca: Temporally constrained k-t blast reconstruction using
  principal component analysis.
\newblock {\em Magnetic Resonance in Medicine}, 62(3):706--716, 2009.

\bibitem{RanganGAMP}
S.~Rangan.
\newblock Generalized approximate message passing for estimation with random
  linear mixing.
\newblock In {\em IEEE International Symposium on Information Theory (ISIT)},
  2011.

\bibitem{rangan2019vector}
Sundeep Rangan, Philip Schniter, and Alyson~K Fletcher.
\newblock Vector approximate message passing.
\newblock {\em IEEE Transactions on Information Theory}, 65(10):6664--6684,
  2019.

\bibitem{Reeves_mtx_ten}
Galen Reeves.
\newblock Information-theoretic limits for the matrix tensor product.
\newblock {\em IEEE Journal on Selected Areas in Information Theory},
  1(3):777--798, 2020.

\bibitem{Rockafellar_book}
R.~Tyrrell Rockafellar.
\newblock {\em Convex analysis}.
\newblock Princeton Landmarks in Mathematics. Princeton University Press,
  Princeton, NJ, 1997.

\bibitem{stein_proc}
Charles~M. Stein.
\newblock Estimation of the mean of a multivariate normal distribution.
\newblock {\em The Annals of Statistics}, 9(6):1135--1151, 1981.

\bibitem{shrinkage_sepcov}
Pei-Chun Su and Hau-Tieng Wu.
\newblock Data-driven optimal shrinkage of singular values under
  high-dimensional noise with separable covariance structure.
\newblock {\em arXiv preprint arXiv:2207.03466}, 2022.

\bibitem{astronomy1}
O.~Tamuz, T.~Mazeh, and S.~Zucker.
\newblock {Correcting systematic effects in a large set of photometric light
  curves}.
\newblock {\em Monthly Notices of the Royal Astronomical Society},
  356(4):1466--1470, 02 2005.

\bibitem{venkataramanan2022estimation}
Ramji Venkataramanan, Kevin K{\"o}gler, and Marco Mondelli.
\newblock Estimation in rotationally invariant generalized linear models via
  approximate message passing.
\newblock In {\em International Conference on Machine Learning (ICML)}, 2022.

\bibitem{wu-zhou-2019-em}
Yihong Wu and Harrison~H. Zhou.
\newblock Randomly initialized {EM} algorithm for two-component {G}aussian
  mixture achieves near optimality in {$O(\sqrt n)$} iterations.
\newblock {\em Math. Stat. Learn.}, 4(3-4):143--220, 2021.

\bibitem{YWF_group}
Kaylee~Y Yang, Timothy~LH Wee, and Zhou Fan.
\newblock Asymptotic mutual information in quadratic estimation problems over
  compact groups.
\newblock {\em arXiv preprint arXiv:2404.10169}, 2024.

\bibitem{SBM_multi}
Xiaodong Yang, Buyu Lin, and Subhabrata Sen.
\newblock Fundamental limits of community detection from multi-view data:
  multi-layer, dynamic and partially labeled block models.
\newblock {\em arXiv preprint arXiv:2401.08167}, 2024.

\bibitem{Yuan_cov}
Ming Yuan.
\newblock High dimensional inverse covariance matrix estimation via linear
  programming.
\newblock {\em J. Mach. Learn. Res.}, 11:2261--2286, 2010.

\bibitem{HeteroPCA}
Anru~R. Zhang, T.~Tony Cai, and Yihong Wu.
\newblock Heteroskedastic {PCA}: algorithm, optimality, and applications.
\newblock {\em Ann. Statist.}, 50(1):53--80, 2022.

\bibitem{ZhangZhang}
Cun-Hui Zhang and Stephanie~S. Zhang.
\newblock Confidence intervals for low dimensional parameters in high
  dimensional linear models.
\newblock {\em J. R. Stat. Soc. Ser. B. Stat. Methodol.}, 76(1):217--242, 2014.

\bibitem{Zhang_Thesis_RMT}
Lixin Zhang.
\newblock {\em Spectral analysis of large dimentional random matrices}.
\newblock PhD thesis, National University of Singapore, 2007.

\bibitem{zhang2023spectral}
Yihan Zhang, Hong~Chang Ji, Ramji Venkataramanan, and Marco Mondelli.
\newblock Spectral estimators for structured generalized linear models via
  approximate message passing.
\newblock {\em arXiv preprint arXiv:2308.14507}, 2023.

\bibitem{mixed-zmv-arxiv}
Yihan Zhang, Marco Mondelli, and Ramji Venkataramanan.
\newblock Precise asymptotics for spectral methods in mixed generalized linear
  models.
\newblock {\em arXiv preprint arXiv:2211.11368}, 2022.

\bibitem{zhong2021approximate}
Xinyi Zhong, Tianhao Wang, and Zhou Fan.
\newblock Approximate message passing for orthogonally invariant ensembles:
  Multivariate non-linearities and spectral initialization.
\newblock {\em arXiv preprint arXiv:2110.02318}, 2021.

\bibitem{DeflatedHeteroPCA}
Yuchen Zhou and Yuxin Chen.
\newblock Deflated heteropca: Overcoming the curse of ill-conditioning in
  heteroskedastic pca.
\newblock {\em arXiv preprint arXiv:2303.06198}, 2023.

\end{thebibliography}


\newpage
\appendix

\paragraph{Notation.} 
All vectors are column vectors. 
The singular values of a matrix $A\in\bbR^{n\times d}$ (where $n\ge d$ without loss of generality) are denoted by $ \sigma_1(A) \ge \cdots \ge \sigma_d(A) \ge 0 $ and the corresponding left/right singular vectors are denoted by $ u_1(A), \cdots, u_d(A) \in\bbS^{n-1} $ and $ v_1(A), \cdots, v_d(A)\in\bbS^{d-1} $. 
The (real) eigenvalues of a symmetric matrix $B\in\bbR^{d\times d}$ are denoted by $ \lambda_1(B) \ge \cdots \ge \lambda_d(B) $ and the corresponding eigenvectors are denoted by $ v_1(B), \cdots, v_d(B) \in \bbS^{d-1} $ (which will not be confused with the right singular vectors, whenever they are different, since we will never talk about both simultaneously for a square asymmetric matrix). 
We generally put overlines on capital letters to indicate a scalar random variable, e.g., $ \ol{X}\in\bbR $, whose support is denoted by $ \supp(\ol{X}) $. 
The limit/liminf/limsup in probability are denoted by $\plim, \pliminf, \plimsup$. 
The product distribution whose $i$-th ($i\in[k]$) marginal is given by $P_i$ is denoted by $ P_1 \ot \cdots \ot P_k $, with the shorthand $P^{\ot k}$ when all $P_i$'s are equal to $P$. 
The gradient of $f\colon\bbR^n\to\bbR$, or with abuse of notation, the Jacobian matrix of $ F\colon\bbR^n\to\bbR^d $ are denoted by $ \nabla f\in\bbR^n, \nabla F\in\bbR^{d\times n} $. 
The partial derivative of $ f(x_1, \cdots, x_n) $ with respect to $x_i$ is denoted by either $ \frac{\partial}{\partial x_i} f(x_1, \cdots, x_n) $ or $ \partial_i f(x_1, \cdots, x_n) $. 
All $\log$ and $\exp$ are to the base $e$. 
We use the standard notation of $\sup(S), \inf(S)$ for a subset $S\subset\bbR$. 
We generally use $C>0$ to denote a sufficiently large constant independent of $n,d$.
Its dependence on other parameters will be specified, though its value may change across passages. 
We use the standard big O notation.


\section{Proof of \Cref{prop:bayes_fp_sol}}
\label{app:pf_prop:bayes_fp_sol}

We eliminate $ q_u $ and write a fixed point equation only involving $ q_v $: 
\begin{align}
q_v &= \expt{ \frac{\snr \expt{ \frac{\alpha \snr q_v \ol{\Xi}^{-2}}{\alpha \snr q_v \ol{\Xi}^{-1} + 1} } \ol{\Sigma}^{-2}}{\snr \expt{ \frac{\alpha \snr q_v \ol{\Xi}^{-2}}{\alpha \snr q_v \ol{\Xi}^{-1} + 1} } \ol{\Sigma}^{-1} + 1} } . \notag 
\end{align}
Denote the RHS by $ f(q_v) $. 
Recall that we are only interested in non-negative solutions $ (q_u, q_v) $. 
So let us restrict attention on $f$ to the domain $ \bbR_{\ge0} $.
We have $ f(0) = 0 $ and
\begin{align}
f'(q_u) &= \expt{\frac{
    \snr \ol{\Sigma}^{-2}
}{
    \paren{ \snr \expt{ \frac{\alpha\snr q_v\ol{\Xi}^{-2}}{\alpha\snr q_v\ol{\Xi}^{-1} + 1} } \ol{\Sigma}^{-1} + 1 }^2
}} \expt{\frac{
    \alpha \snr \ol{\Xi}^{-2}
}{
    \paren{ \alpha\snr q_v\ol{\Xi}^{-1} + 1 }^2
}} > 0 , \notag \\
f'(0) &= \alpha \snr^2 \expt{\ol{\Sigma}^{-2}} \expt{\ol{\Xi}^{-2}} , \notag \\
f''(q_v) &= -2 \left(
    \expt{\frac{\alpha\snr\ol{\Xi}^{-2}}{\paren{ \alpha\snr q_v\ol{\Xi}^{-1} + 1 }^2}}^2
    \expt{\frac{\snr^2\ol{\Sigma}^{-3}}{\paren{\snr \expt{\frac{\alpha\snr q_v\ol{\Xi}^{-2}}{\alpha\snr q_v\ol{\Xi}^{-1} + 1}} \ol{\Sigma}^{-1} + 1}^3}} \right. \notag \\
&\qquad\qquad \left. 
    + \expt{\frac{\snr\ol{\Sigma}^{-2}}{\paren{\snr \expt{\frac{\alpha\snr q_v\ol{\Xi}^{-2}}{\alpha\snr q_v\ol{\Xi}^{-1} + 1}} \ol{\Sigma}^{-1} + 1}^2}}
    \expt{\frac{\alpha^2 \snr^2 \ol{\Xi}^{-3}}{\paren{\alpha\snr q_v\ol{\Xi}^{-1} + 1}^3}}
\right) < 0 , \notag \\
\lim_{q_v\to\infty} f(q_v) &= \expt{\frac{\snr \expt{\ol{\Xi}^{-1}} \ol{\Sigma}^{-2}}{\snr \expt{\ol{\Xi}^{-1}} \ol{\Sigma}^{-1} + 1}} \in (0,\infty) . \notag 
\end{align}
It then becomes evident that a non-trivial fixed point $ q_v > 0 $ exists if and only if $ f'(0) > 1 $ and in this case, the non-trivial fixed point is unique. 

Finally, by the first equation in \Cref{eqn:fp_it}, there is a non-trivial fixed point $ q_u $ if and only if there is a non-trivial fixed point $ q_v $, which completes the proof. 

\section{Auxiliary Gaussian channel}
\label{app:gauss_ch}

We formally introduce here the auxiliary model mentioned in \Cref{sec:IT}. 
Consider a Gaussian channel with blocklength $n$, input $ x^* $, output $ Y $, anisotropic Gaussian noise $ \Sigma^{1/2} Z $ and SNR $\snr$: 
\begin{align}
Y &= \sqrt{\snr} x^* + \Sigma^{1/2} Z \in \bbR^n , \label{eqn:gauss_ch}
\end{align}
where 
\begin{align}
(x^*, Z) &\sim P^{\ot n} \ot \cN(0_n, I_n) . \notag 
\end{align}
By similar derivations as in \Cref{sec:IT}, the posterior distribution of $ x^* $ given $ Y $ can be written as
\begin{align}
\diff P(x \mid Y) &= \frac{1}{Z_n(\snr)} \exp\paren{ H_n(x) } \diff P^{\ot n}(x) , \notag 
\end{align}
where the Hamiltonian and the partition function are
\begin{align}
H_n(x) &\coloneqq \snr {x^*}^\top \Sigma^{-1} x + \sqrt{\snr} Z^\top \Sigma^{-1/2} x - \frac{\snr}{2} x^\top \Sigma^{-1} x , \notag \\
Z_n(\snr) &\coloneqq \int_{\bbR^n} \exp\paren{ H_n(x) } \diff P^{\ot n}(x) . \notag 
\end{align}
Define the free energy as 
\begin{align}
F_n(\snr) &\coloneqq \frac{1}{n} \expt{ \log Z_n(\snr) } . \notag 
\end{align}

With $ P = \cN(0, 1) $, $ Z_n(\snr) $ becomes a Gaussian integral that can be computed as below using \Cref{prop:gauss_int}:
\begin{align}
Z_n(\snr) &= \frac{1}{\sqrt{\det(\snr \Sigma^{-1} + I_n)}} \exp\paren{ \frac{1}{2} \paren{\snr \Sigma^{-1} x^* + \sqrt{\snr} \Sigma^{-1/2} Z}^\top \paren{\snr \Sigma^{-1} + I_n}^{-1} \paren{\snr \Sigma^{-1} x^* + \sqrt{\snr} \Sigma^{-1/2} Z} } . \notag 
\end{align}
Therefore, by \Cref{prop:quadratic_form}, 
\begin{align}
\plim_{n\to\infty} F_n(\snr) &= -\frac{1}{2} \expt{\log\paren{ \snr \ol{\Sigma}^{-1} + 1 }} + \frac{1}{2} \snr^2 \expt{\ol{\Sigma}^{-2} \paren{\snr \ol{\Sigma}^{-1} + 1}^{-1}} + \frac{1}{2} \snr \expt{\ol{\Sigma}^{-1} \paren{\snr \ol{\Sigma}^{-1} + 1}^{-1}} \notag \\
&= \frac{1}{2} \paren{ \snr \expt{\ol{\Sigma}^{-1}} - \expt{\log\paren{1 + \snr\ol{\Sigma}^{-1}}} } . \label{eqn:lim_energy_gauss} 
\end{align}
The above functional is nothing but $\psi_{\ol{\Sigma}}(\snr)$ introduced in \Cref{eqn:psi_Sigma} which will play an important role in characterizing the free energy of the original model \Cref{eqn:Y_mtx}.

\section{Proof of \Cref{thm:free_energy}}\label{app:pf}

Before diving into the proof, we make further notation adjustments for the ease of applying the interpolation argument. 
Specifically, we will henceforth assume $ \snr = 1 $ by incorporating the actual value of $\snr$ into the prior distributions $ P , Q $, 
\begin{align}
\int_\bbR x^2 \diff P(x) &= \snr , \quad 
\int_\bbR x^2 \diff Q(x) = 1 . \notag
\end{align}
This is obviously equivalent to the previous setting. 
So we can drop the dependence on $\snr$ and write $ \mmse_n, \cZ_n, \cF_n $ for $ \mmse_n(\snr), \cZ_n(\snr), \cF_n(\snr) $ defined in \Cref{eqn:MMSE_Y,eqn:part_Y,eqn:energy_Y}. 

We will also assume that $ \Xi, \Sigma $ are diagonal. 
This is without loss of generality since the Gaussianity of $ P, Q, \wt{W} $ ensures that both the prior distributions and the noise matrix are rotationally invariant. 
Furthermore, we truncate $ P, Q $ so that they are supported on $[-K, K]$ for a constant $K>0$. 
The approximation error in the free energy due to truncation can be made arbitrarily small if $K$ is sufficiently large, since the free energy is pseudo-Lipschitz in the prior distribution with respect to the Wasserstein-$2$ metric.

The proof follows an interpolation argument \cite{BM_interp,Miolane_asymm,miolane_thesis} with suitable modifications to take care of the noise heteroscedasticity featured by the covariances $ \Xi, \Sigma $. 
To start with, define the interpolating models:
\begin{align}
Y_t &\coloneqq \sqrt{\frac{1 - t}{n}} {u}^* {v^*}^\top + \Xi^{1/2} Z \Sigma^{1/2} \in\bbR^{n\times d} , \notag \\
Y_t^u &\coloneqq \sqrt{\alpha q_1(t)} {u}^* + \Xi^{1/2} Z^u \in\bbR^n , \notag \\
Y_t^v &\coloneqq \sqrt{q_2(t)} {v}^* + \Sigma^{1/2} Z^v \in\bbR^d , \notag 
\end{align}
where $ q_1(t), q_2(t)\ge0 $ are to be determined and
\begin{align}
(u^*, v^*, Z, Z^u, Z^v) &\sim P^{\ot n} \ot Q^{\ot d} \ot \cN(0_{nd}, I_{nd}) \ot \cN(0_n, I_n) \ot \cN(0_d, I_d) . \label{eqn:interp_distr} 
\end{align}
By definition, $ Y_0 = Y $ is the model that we would like to understand, and $ Y_t^u, Y_t^v $ are instances of Gaussian channels in \Cref{eqn:gauss_ch} whose free energy we have already understood (see \Cref{eqn:lim_energy_gauss}). 
The idea is that $ Y_t $ serves as a path parametrized by $ t\in[0,1] $ from the original model $Y$ to the target models $ (Y_1^u, Y_1^v) $. 
The crux of the interpolation argument lies in showing that $ Y_t $ and $ (Y_t^u, Y_t^v) $ are equivalent (at the level of free energy) along the path. 

To study the interpolating models $ (Y_t, Y_t^u, Y_t^v) $, define the Hamiltonian
\begin{align}
\begin{split}
\cH_{n,t}(\wt{u}, \wt{v}; q_1, q_2) &\coloneqq \sqrt{\frac{1-t}{n}} \wt{u}^\top Z \wt{v} + \frac{1-t}{n} \wt{u}^\top \wt{u}^* \wt{v}^\top \wt{v}^* - \frac{1-t}{2n} \normtwo{\wt{u}}^2 \normtwo{\wt{v}}^2 \\
&\phantom{=}~ + \alpha q_1 \wt{u}^\top \wt{u}^* + \sqrt{\alpha q_1} \wt{u}^\top Z^u - \frac{\alpha q_1}{2} \normtwo{\wt{u}}^2 \\
&\phantom{=}~ + q_2 \wt{v}^\top \wt{v}^* + \sqrt{q_2} \wt{v}^\top Z^v - \frac{q_2}{2} \normtwo{\wt{v}}^2 . 
\end{split}
\label{eqn:H_nt}
\end{align}
Then the posterior distribution of $ ({u}^*, {v}^*) $ given $ (Y_t, Y_t^u, Y_t^v) $ is 
\begin{align}
\diff P({u}, {v} \mid Y_t, Y_t^u, Y_t^v) 
&= \frac{1}{\cZ_{n,t}} \exp\paren{ \cH_{n,t}(\Xi^{-1/2} u, \Sigma^{-1/2} v; q_1(t), q_2(t)) } \diff P^{\ot n}(u) \diff Q^{\ot d}(v) . \label{eqn:posterior_interp} 
\end{align}
Let the partition function be
\begin{align}
\cZ_{n,t} &\coloneqq \int_{\bbR^d} \int_{\bbR^n} \exp\paren{ \cH_{n,t}(\Xi^{-1/2} u, \Sigma^{-1/2} v; q_1(t), q_2(t)) } \diff P^{\ot n}(u) \diff Q^{\ot d}(v) \notag \\
&= \int_{\bbR^d} \int_{\bbR^n} \exp\paren{ \cH_{n,t}(\wt{u}, \wt{v}; q_1(t), q_2(t)) } \diff \wt{P}(\wt{u}) \diff \wt{Q}(\wt{v}) . \label{eqn:Znt}
\end{align}
Define the free energy as 
\begin{align}
f_n(t) &\coloneqq \frac{1}{n} \expt{\log\cZ_{n,t}} . \label{eqn:fnt}
\end{align}

The Gibbs bracket $ \bracket{\cdot}_{n,t} $ denotes the expectation with respect to the posterior distribution in \Cref{eqn:posterior_interp}:
\begin{align}
\bracket{g(\wt{u}, \wt{v})}_{n,t} &\coloneqq \frac{1}{\cZ_{n,t}} \int_{\bbR^d} \int_{\bbR^n} g(\wt{u}, \wt{v}) \exp\paren{ \cH_{n,t}(\wt{u}, \wt{v}; q_1(t), q_2(t)) } \diff \wt{P}(\wt{u}) \diff \wt{Q}(\wt{v}) , \label{eqn:gibbs_brack} 
\end{align}
for any $ g\colon\bbR^n\times\bbR^d \to \bbR $ such that the expectation exists. 
That is, 
\begin{align}
\bracket{g(\wt{u}, \wt{v})}_{n,t} &= \expt{ g(\wt{u}^*, \wt{v}^*) \mid Y_t, Y_t^u, Y_t^v } , \notag 
\end{align}
where we recall the notation
\begin{align}
\wt{u}^* \coloneqq \Xi^{-1/2} u^* , \qquad 
\wt{v}^* \coloneqq \Sigma^{-1/2} v^* . \label{eqn:uv*_tilde} 
\end{align}
We will also use the notation $ \bracket{\cdot}_n $ for the Gibbs bracket with respect to the original posterior $ \diff P(u, v \mid Y) $ in \Cref{eqn:posterior}. 

\begin{lemma}
\label{lem:f01}
Consider $ f_n(t) $ defined in \Cref{eqn:fnt} with $ t\in\{0,1\} $. 
Assume that $ q_1(0), q_2(0) $ satisfy
\begin{align}
    q_1(0) &\ge 0 , \quad q_2(0) \ge 0 , \quad 
    \lim_{n\to\infty} q_1(0) = \lim_{n\to\infty} q_2(0) = 0 . \notag 
\end{align}
Then we have
\begin{align}
f_n(0) &= \cF_n + \cO(q_1(0) + q_2(0)) , \label{eqn:f=F} \\
\lim_{n\to\infty} f_n(1) &= \psi_{\ol{\Xi}}(\alpha q_1(1) \snr) + \alpha \psi_{\ol{\Sigma}}(q_2(1)) + o_K , \label{eqn:f1} 
\end{align}
where $ \lim_{K\to\infty} o_K = 0 $. 
\end{lemma}

\begin{proof}
To show the first statement \Cref{eqn:f=F}, let us control $ f_n(0) - \cF_n $. 
Denoting 
\begin{align}
    \cH_{n,t}'(\wt{u}, \wt{v}; q_1, q_2) &\coloneqq \alpha q_1 \wt{u}^\top \wt{u}^* + \sqrt{\alpha q_1} \wt{u}^\top Z^u - \frac{\alpha q_1}{2} \normtwo{\wt{u}}^2 
    + q_2 \wt{v}^\top \wt{v}^* + \sqrt{q_2} \wt{v}^\top Z^v - \frac{q_2}{2} \normtwo{\wt{v}}^2 \notag 
\end{align}
and recalling the Gibbs bracket notation $ \bracket{\cdot}_n $, we have
\begin{align}
    f_n(0) - \cF_n &= \frac{1}{n} \expt{ \log\frac{\cZ_{n,0}}{\cZ_n} }
    = \frac{1}{n} \expt{\log\bracket{ \exp\paren{\cH_{n,0}'(\wt{u}, \wt{v}; q_1(0), q_2(0))} }_{n}} , \label{eqn:f-F} 
\end{align}
where the outer expectation is over all randomness in $ {u}^*, {v}^*, Z^u, Z^v $. 
The second equality above follows since
\begin{align}
    \cH_{n,0}(\wt{u}, \wt{v}; q_1(0), q_2(0))
    &= \cH_n(\wt{u}, \wt{v}) + \cH_{n,0}'(\wt{u}, \wt{v}; q_1(0), q_2(0)) . 
    \notag 
\end{align}
We will derive double-sided bounds on $ f_n(0) - \cF_n $. 

To upper bound it, use Jensen's inequality $ 
\expt{\log(\cdot)} \le \log\expt{\cdot} $ on the partial expectation over $ Z^u, Z^v $ in \Cref{eqn:f-F}: 
\begin{align}
    f_n(0) - \cF_n
    &\le \frac{1}{n} \exptover{{u}^*, {v}^*}{ \log\exptover{Z^u, Z^v}{ \bracket{ \exp\paren{\cH_{n,0}'(\wt{u}, \wt{v}; q_1(0), q_2(0))} }_{n} } } \notag \\
    &= \frac{1}{n} \exptover{{u}^*, {v}^*}{ \log \bracket{ \exptover{Z^u, Z^v}{ \exp\paren{\cH_{n,0}'(\wt{u}, \wt{v}; q_1(0), q_2(0))} } }_{n} } , \notag 
\end{align}
where the equality is legit since $ \bracket{\cdot}_n $ does not depend on $ Z^u, Z^v $. 
By the Gaussian integral formula (\Cref{prop:gauss_int}), the inner expectation equals
\begin{align}
    \exptover{Z^u, Z^v}{ \exp\paren{\cH_{n,0}'(\wt{u}, \wt{v}; q_1(0), q_2(0))} }
    &= \exp\paren{ \alpha q_1(0) \wt{u}^\top \wt{u}^* + q_2(0) \wt{v}^\top \wt{v}^* } . \notag 
\end{align}
Replacing the Gibbs bracket with $\max$, we obtain an upper bound:
\begin{align}
    f_n(0) - \cF_n 
    &\le \frac{1}{n} \exptover{{u}^*, {v}^*}{
    \log  
    \max_{(\wt{u}, \wt{v})\in\Xi^{-1/2}[-K,K]^n\times\Sigma^{-1/2}[-K,K]^d} \exp\paren{ \alpha q_1(0) \wt{u}^\top \wt{u}^* + q_2(0) \wt{v}^\top \wt{v}^* }
    } \notag \\
    &\le \alpha q_1(0) \normtwo{\Xi}^{-1} K^2 + q_2(0) \normtwo{\Sigma}^{-1} K^2 \frac{d}{n} \notag \\
    &\le \frac{2 \alpha q_1(0) K^2}{\inf\supp(\ol{\Xi})} + \frac{2\alpha q_2(0) K^2}{\inf\supp(\ol{\Sigma})} , \label{eqn:f-F_ub} 
\end{align}
where the last inequality holds for all sufficiently large $n$ by \Cref{eqn:supp} and $n\asymp d$. 

To lower bound $ f_n - \cF_n $, we use Jensen's inequality again but this time on the Gibbs bracket in \Cref{eqn:f-F}: 
\begin{align}
    f_n(0) - \cF_n &\ge \frac{1}{n} \expt{ \bracket{ \cH_{n,0}'(\wt{u}, \wt{v}; q_1(0), q_2(0)) }_{n} } 
    = \frac{1}{n} \exptover{{u}^*, {v}^*}{
    \bracket{
    \exptover{Z^u, Z^v}{
    \cH_{n,0}'(\wt{u}, \wt{v}; q_1(0), q_2(0))
    }
    }_{n}
    }
    . \label{eqn:to_lb} 
\end{align}
Since $ (Z^u, Z^v) \sim \cN(0_n, I_n) \ot \cN(0_d, I_d) $, the inner expectation equals 
\begin{align}
    \exptover{Z^u, Z^v}{
    \cH_{n,0}'(\wt{u}, \wt{v}; q_1(0), q_2(0))
    }
    &= \alpha q_1(0) \wt{u}^\top \wt{u}^* - \frac{\alpha q_1(0)}{2} \normtwo{\wt{u}}^2
    + q_2(0) \wt{v}^\top \wt{v}^* - \frac{q_2(0)}{2} \normtwo{\wt{v}}^2 . \notag 
\end{align}
So 
\begin{multline}
    \max_{(\wt{u}, \wt{v})\in\Xi^{-1/2}[-K,K]^n\times\Sigma^{-1/2}[-K,K]^d} \abs{ \exptover{Z^u, Z^v}{
    \cH_{n,0}'(\wt{u}, \wt{v}; q_1(0), q_2(0))
    } } \notag \\
    \le \frac{3}{2} \alpha q_1(0) \normtwo{\Xi}^{-1} n K^2 
    + \frac{3}{2} q_2(0) \normtwo{\Sigma}^{-1} d K^2 . \notag 
\end{multline}
Using this, we obtain a lower bound on $ f_n - \cF_n $ by replacing the Gibbs bracket on the RHS of \Cref{eqn:to_lb} with $ - \max \abs{\cdot} $:
\begin{align}
    f_n(0) - \cF_n &\ge - \frac{1}{n} \exptover{{u}^*, {v}^*}{
    \max_{(\wt{u}, \wt{v})\in\Xi^{-1/2}[-K,K]^n\times\Sigma^{-1/2}[-K,K]^d} \abs{ \exptover{Z^u, Z^v}{
    \cH_{n,0}'(\wt{u}, \wt{v}; q_1(0), q_2(0))
    } }
    } \notag \\
    &\ge - \frac{3}{2} \alpha q_1(0) \normtwo{\Xi}^{-1} K^2 - \frac{3}{2} q_2(0) \normtwo{\Sigma}^{-1} K^2 \frac{d}{n} \notag \\
    &\ge - \frac{3\alpha q_1(0) K^2}{\inf\supp(\ol{\Xi})} - \frac{3 \alpha q_2(0) K^2}{\inf\supp(\ol{\Sigma})} , \label{eqn:f-F_lb} 
\end{align}
where the last inequality holds for all sufficiently large $n$ by \Cref{eqn:supp} and $n\asymp d$. 
Combining \Cref{eqn:f-F_lb,eqn:f-F_ub} gives the first result \Cref{eqn:f=F}. 

We then prove the second statement \Cref{eqn:f1}.
Since $ f_n $ is pseudo-Lipschitz as a function of the priors, up to a term $ o_K $ that vanishes as $K\to\infty$ uniformly over $n$, it suffices to ignore the truncation at $K$ and assume $ P = \cN(0,\snr), Q = \cN(0,1) $. 
From the definition \Cref{eqn:Znt} of $ \cZ_{n,t} $, we have
\begin{align}
& \cZ_{n,1} = \int_{\bbR^d} \int_{\bbR^n} \exp\paren{ \cH_{n,1}'(\Xi^{-1/2} u, \Sigma^{-1/2} v; q_1(1), q_2(1)) } \diff P^{\ot n}(u) \diff Q^{\ot d}(v) \notag \\
&= \sqrt{\frac{\det(\Xi)}{\snr^n \det(\Xi / \snr + \alpha q_1(1) I_n)}} \notag \\
&\times\exp\paren{ \frac{1}{2} \paren{ \alpha q_1(1) \wt{u}^* + \sqrt{\alpha q_1(1)} Z^u }^\top \paren{ \Xi / \snr + \alpha q_1(1) I_n }^{-1} \paren{ \alpha q_1(1) \wt{u}^* + \sqrt{\alpha q_1(1)} Z^u } } \notag \\
&\times \sqrt{\frac{\det(\Sigma)}{\det(\Sigma + q_2(1) I_d)}} \exp\paren{ \frac{1}{2} \paren{ q_2(1) \wt{v}^* + \sqrt{q_2(1)} Z^v }^\top \paren{ \Sigma + q_2(1) I_d }^{-1} \paren{ q_2(1) \wt{v}^* + \sqrt{q_2(1)} Z^v } } , \notag 
\end{align}
where in the second equality, we use the Gaussian integral formula (\Cref{prop:gauss_int}). 
Therefore
\begin{align}
\lim_{n\to\infty} f_n(1) 
&= \lim_{n\to\infty} \frac{1}{n} \expt{\log\cZ_{n,1}} \notag \\
&= \frac{1}{2} \alpha q_1(1) \snr \expt{\ol{\Xi}^{-1}} - \frac{1}{2} \expt{ \log\paren{ 1 + \alpha q_1(1) \snr \ol{\Xi}^{-1} } } \notag \\
&\quad 
+ \frac{\alpha}{2} q_2(1) \expt{\ol{\Sigma}^{-1}} 
- \frac{\alpha}{2} \expt{ \log\paren{ 1 + q_2(1) \ol{\Sigma}^{-1} } } , \notag 
\end{align}
verifying the identity \Cref{eqn:f1}. 
In the second equality, we have used \Cref{prop:quadratic_form}. 
\end{proof}

\begin{lemma}[Free energy variation]
\label{lem:free_energy_var}
For all $ t\in(0,1) $, 
\begin{align}
f_n'(t) &= \frac{\alpha}{2} q_1'(t) q_2'(t) - \frac{1}{2} \expt{ \bracket{ \paren{ \frac{\wt{u}^\top \wt{u}^*}{n} - q_2'(t) } \paren{ \frac{\wt{v}^\top \wt{v}^*}{n} - \alpha q_1'(t) } }_{n,t} } . \label{eqn:free_energy_var} 
\end{align}
\end{lemma}

\begin{proof}
From the definitions \Cref{eqn:fnt,eqn:Znt}, we compute 
\begin{align}
f_n'(t) &= \frac{1}{n} \expt{ \frac{1}{\cZ_{n,t}} \frac{\partial}{\partial t} \cZ_{n,t} } \notag \\
&= \frac{1}{n} \expt{ \frac{1}{\cZ_{n,t}} \int_{\bbR^d} \int_{\bbR^n} \paren{ \frac{\partial}{\partial t} \cH_{n,t}(\wt{u}, \wt{v}; q_1(t), q_2(t)) } \exp\paren{ \cH_{n,t}(\wt{u}, \wt{v}; q_1(t), q_2(t)) } \diff \wt{P}(\wt{u}) \diff \wt{Q}(\wt{v}) } \notag \\
&= \frac{1}{n} \expt{ \bracket{\frac{\partial}{\partial t} \cH_{n,t}(\wt{u}, \wt{v}; q_1(t), q_2(t))}_{n,t} } . \notag 
\end{align}
By the definition \Cref{eqn:H_nt}, the time derivative of the Hamiltonian is 
\begin{align}
\begin{split}
\frac{\partial}{\partial t} \cH_{n,t}(\wt{u}, \wt{v}; q_1(t), q_2(t))
&= -\frac{1}{2\sqrt{(1-t)n}} \wt{u}^\top Z \wt{v} - \frac{1}{n} \wt{u}^\top \wt{u}^* \wt{v}^\top \wt{v}^* + \frac{1}{2n} \normtwo{\wt{u}}^2 \normtwo{\wt{v}}^2 \\
&\quad + \alpha q_1'(t) \wt{u}^\top \wt{u}^* + \frac{\sqrt{\alpha}}{2\sqrt{q_1(t)}} q_1'(t) \wt{u}^\top Z^u - \frac{\alpha}{2} q_1'(t) \normtwo{\wt{u}}^2 \\
&\quad + q_2'(t) \wt{v}^\top \wt{v}^* + \frac{1}{2\sqrt{q_2(t)}} q_2'(t) \wt{v}^\top Z^v - \frac{1}{2} q_2'(t) \normtwo{\wt{v}}^2 . 
\end{split}
\label{eqn:H_t_deriv}
\end{align}
The expectation of $ \bracket{\partial_t \cH_{n,t}}_{n,t} $ can be computed using the Stein's lemma (\Cref{prop:stein}). 
Indeed, let us consider the term
\begin{align}
\begin{split}
    \expt{\bracket{\wt{u}^\top Z \wt{v}}_{n,t}}
    &= \sum_{i = 1}^n \sum_{j = 1}^d \expt{\bracket{\wt{u}_i \wt{v}_j}_{n,t} Z_{i,j}}
    = \sum_{i = 1}^n \sum_{j = 1}^d \expt{ \frac{\partial \bracket{\wt{u}_i \wt{v}_j}_{n,t}}{\partial Z_{i,j}}} \\
    &= \sum_{i = 1}^n \sum_{j = 1}^d \sqrt{\frac{1-t}{n}} \expt{ \bracket{ \wt{u}_i^2 \wt{v}_j^2 }_{n,t} }
    - \sum_{i = 1}^n \sum_{j = 1}^d \sqrt{\frac{1-t}{n}} \expt{ \bracket{ \wt{u}_i \wt{v}_j }_{n,t}^2 } \\
    &= \sqrt{\frac{1-t}{n}} \expt{ \bracket{\normtwo{\wt{u}}^2 \normtwo{\wt{v}}^2}_{n,t} }
    - \sqrt{\frac{1-t}{n}} \expt{ \bracket{ \wt{u}^\top \wt{u}^* \wt{v}^\top \wt{v}^* }_{n,t} } ,  
\end{split}
\label{eqn:u'Zv}
\end{align}
where the last step is by the Nishimori identity (\Cref{prop:nishimori}). 
So the first line of \Cref{eqn:H_t_deriv} upon taken the Gibbs bracket and the expectation becomes
\begin{align}
    -\frac{1}{2n} \expt{ \bracket{ \wt{u}^\top \wt{u}^* \wt{v}^\top \wt{v}^* }_{n,t} } . \notag 
\end{align}
Similar cancellations happen for the second and third lines of \Cref{eqn:H_t_deriv}. 
Putting them together, we obtain
\begin{align}
f'_n(t) &= -\frac{1}{2n^2} \expt{\bracket{\wt{u}^\top \wt{u}^* \wt{v}^\top \wt{v}^*}_{n,t}}
+ \frac{\alpha}{2n} q_1'(t) \expt{\bracket{\wt{u}^\top \wt{u}^*}_{n,t}}
+ \frac{1}{2n} q_2'(t) \expt{\bracket{\wt{v}^\top \wt{v}^*}_{n,t}} , \notag 
\end{align}
which is the same as \Cref{eqn:free_energy_var} with the parentheses opened up. 
\end{proof}

In what follows, our strategy is:
\begin{enumerate}
    \item\label{itm:strategy1} Show that $ \bracket{\wt{u}^\top \wt{u}^*}_{n,t} $ is concentrated around its mean $ \expt{ \bracket{\wt{u}^\top \wt{u}^*}_{n,t} } $; 

    \item\label{itm:strategy2} Choose $ q_2(t) $ to be the solution to 
    \begin{align}
    q_2'(t) &= \frac{1}{n} \expt{ \bracket{\wt{u}^\top \wt{u}^*}_{n,t} } . \notag 
    \end{align}
\end{enumerate}
Once \Cref{itm:strategy1,itm:strategy2} are done, we then have 
\begin{align}
\cF_n &\approx f_n(0) = f_n(1) - \int_0^1 f_n'(t) \diff t \notag \\
&\approx \psi_{\ol{\Xi}}(\alpha q_1(1) \snr) + \alpha \psi_{\ol{\Sigma}}(q_2(1)) 
- \int_0^1 \frac{\alpha}{2} q_1'(t) q_2'(t) \diff t
\notag \\
&\quad + \int_0^1 \frac{1}{2} \expt{ \bracket{ \paren{ \frac{\wt{u}^\top \wt{u}^*}{n} - q_2'(t) } \paren{ \frac{\wt{v}^\top \wt{v}^*}{n} - \alpha q_1'(t) } }_{n,t} } \diff t \notag \\
&\approx \psi_{\ol{\Xi}}(\alpha q_1(1) \snr) + \alpha \psi_{\ol{\Sigma}}(q_2(1)) 
- \int_0^1 \frac{\alpha}{2} q_1'(t) q_2'(t) \diff t , \notag 
\end{align}
where the first line above uses \Cref{eqn:f=F} in \Cref{lem:f01};
the second line uses \Cref{eqn:f1} in \Cref{lem:f01} and \Cref{lem:free_energy_var}; 
the third line uses \Cref{itm:strategy1,itm:strategy2}. 
This will almost lead to the desired characterization of the free energy $ \cF_n $ in \Cref{thm:free_energy}: 
\begin{align}
    \sup_{q_v\ge0} \inf_{q_u\ge0} \psi_{\ol{\Xi}}(\alpha \snr q_v) + \alpha \psi_{\ol{\Sigma}}(q_u) - \frac{\alpha}{2} q_u q_v . \notag 
\end{align}

Consider the function 
\begin{align}
    \phi_t(q_1, q_2) &= \frac{1}{n} \log \cZ_{n,t} . 
    \notag 
\end{align}
Note that $\phi_t(q_1, q_2)$ also depends on $ n, u^*, v^*, Z, Z^u, Z^v $, and $ \expt{\phi_t(q_1, q_2)} = f_n(t) $ where the expectation is over $ (u^*, v^*, Z, Z^u, Z^v) $ distributed according to \Cref{eqn:interp_distr}. 

\begin{lemma}[Free energy concentration]
\label{lem:free_energy_conc}
Fix a constant $M>0$. 
There exists a constant $C>0$ depending only on $K,M,\alpha, \ol{\Xi}, \ol{\Sigma}$ such that for any $ t\in[0,1] $, $ 0\le q_1(t),q_2(t)\le M $ and sufficiently large $n$, 
\begin{align}
    \expt{ \abs{ \phi_t(q_1, q_2) - \expt{ \phi_t(q_1, q_2) } } }
    &\le \frac{C}{\sqrt{n}} . \notag 
\end{align}
\end{lemma}

\begin{proof}
Fix $ u^*, v^* $. 
Consider $ \phi_t(q_1, q_2) $ as a function of $ (Z, Z^u, Z^v) $. 
Then 
\begin{align}
    & \normtwo{ \nabla_{(Z, Z^u, Z^v)} \phi_t(q_1, q_2) }^2 \notag \\
    &= \normtwo{ \nabla_Z \phi_t(q_1, q_2) }^2 + \normtwo{ \nabla_{Z^u} \phi_t(q_1, q_2) }^2 + \normtwo{ \nabla_{Z^v} \phi_t(q_1, q_2) }^2 \notag \\
    &= \sum_{i = 1}^n \sum_{j = 1}^d \paren{ \frac{1}{n} \expt{\bracket{\sqrt{\frac{1-t}{n}} \wt{u}_i \wt{v}_j}_{n,t}} }^2
    + \sum_{i = 1}^n \paren{ \frac{1}{n} \expt{\bracket{\sqrt{\alpha q_1} \wt{u}_i}_{n,t}} }^2
    + \sum_{j=1}^d \paren{ \frac{1}{n} \expt{\bracket{\sqrt{q_2} \wt{v}_j}_{n,t}} }^2 \notag \\
    &\le \frac{1}{n^3} \expt{ \bracket{ \normtwo{\wt{u}}^2 \normtwo{\wt{v}}^2 }_{n,t} }
    + \frac{\alpha q_1}{n^2} \expt{\bracket{\normtwo{\wt{u}}^2}_{n,t}}
    + \frac{q_2}{n^2} \expt{\bracket{\normtwo{\wt{v}}^2}_{n,t}}
    \notag \\
    &\le \frac{\alpha}{n} \normtwo{\Xi}^{-1} \normtwo{\Sigma}^{-1} K^4 + \frac{\alpha}{n} M \normtwo{\Xi}^{-1} K^2 + \frac{\alpha}{n} M \normtwo{\Sigma}^{-1} K^2 
    \le \frac{C}{n} , \notag 
\end{align}
where $C>0$ is a constant depending only on $\alpha, M, \ol{\Xi}, \ol{\Sigma}, K$. 
The penultimate line is by Cauchy--Schwarz and the last line holds for all sufficiently large $n$ by \Cref{eqn:supp} and $n\asymp d$. 
Then by the Gaussian \poincare inequality (\Cref{prop:gauss_poincare}), 
\begin{align}
    \exptover{Z, Z^u, Z^v}{\abs{ \phi_t(q_1, q_2) - \exptover{Z, Z^u, Z^v}{\phi_t(q_1, q_2)} }}
    &\le \sqrt{\varover{Z, Z^u, Z^v}{\phi_t(q_1, q_2)}}
    \le \frac{C}{\sqrt{n}} . \label{eqn:tri1} 
\end{align}

The above result holds for any fixed $u^*, v^*$. 
We then verify that $ \exptover{Z, Z^u, Z^v}{\phi_t(q_1, q_2)} $ has bounded difference as a function of $ u^*, v^* $. 
We do so by bounding the derivatives of $\exptover{Z, Z^u, Z^v}{\phi_t(q_1, q_2)}$ with respect to $ u_i^*, v_j^* $ for any $ i\in[n], j\in[d] $. 
We have
\begin{align}
    \frac{\partial}{\partial u_i^*} \exptover{Z, Z^u, Z^v}{\phi_t(q_1, q_2)}
    &= \frac{1}{n} \exptover{Z, Z^u, Z^v}{ \bracket{ \frac{\partial}{\partial u_i^*} \cH_{n,t}(\wt{u}, \wt{v}; q_1, q_2) }_{n,t} } , \label{eqn:phi_deriv} 
\end{align}
and a similar expression holds for the derivative with respect to $ v_j^* $. 
Recall the definition of $ \cH_{n,t} $ from \Cref{eqn:H_nt}. 
We have
\begin{align}
    \frac{\partial}{\partial u^*_i} \cH_{n,t}(\Xi^{-1/2} u, \Sigma^{-1/2} v; q_1, q_2)
    &= \frac{1-t}{n} v^\top \Sigma^{-1} v^* (\Xi^{-1})_{i,i} u_i
    + \alpha q_1 (\Xi^{-1})_{i,i} u_i , \notag 
\end{align}
where we have used the fact that $ \Xi $ is diagonal. 
Therefore,
\begin{align}
    \abs{ \frac{\partial}{\partial u^*_i} \cH_{n,t}(\Xi^{-1/2} u, \Sigma^{-1/2} v; q_1, q_2) }
    &\le \frac{1}{n} \cdot d K^2 \normtwo{\Sigma}^{-1} \cdot \normtwo{\Xi}^{-1} K + \alpha q_1 \cdot \normtwo{\Xi}^{-1} K
    \le C , \notag 
\end{align}
where $C>0$ is a constant depending only on $ \alpha, M, K, \ol{\Xi}, \ol{\Sigma} $. 
The last inequality holds for all sufficiently large $n$.
A similar bound holds for $ \frac{\partial}{\partial v^*_j} \cH_{n,t}(\Xi^{-1/2} u, \Sigma^{-1/2} v; q_1, q_2) $. 
This, by \Cref{eqn:phi_deriv}, implies that $ \exptover{Z, Z^u, Z^v}{\phi_t(q_1, q_2)} $ as a function of $ (u^*, v^*) $ satisfies the bounded difference property with $ c_i = C/n $ (see \Cref{prop:bdd_diff}).
So by \Cref{prop:bdd_diff}, 
\begin{align}
    \exptover{u^*, v^*}{ \abs{ \exptover{Z, Z^u, Z^v}{\phi_t(q_1, q_2)} - \exptover{u^*, v^*}{\exptover{Z, Z^u, Z^v}{\phi_t(q_1, q_2)}} } }
    &\le \sqrt{\varover{u^*, v^*}{\exptover{Z, Z^u, Z^v}{\phi_t(q_1, q_2)}}}
    \le \frac{C}{\sqrt{n}} . \label{eqn:tri2} 
\end{align}

Finally using \Cref{eqn:tri1,eqn:tri2} and the triangle inequality, 
\begin{align}
    & \expt{\abs{ \phi_t(q_1, q_2) - \expt{\phi_t(q_1, q_2)} }} \notag \\
    &\le \exptover{ u^*, v^* }{
    \exptover{ Z, Z^u, Z^v }{
        \abs{ \phi_t(q_1, q_2) - \exptover{Z,Z^u,Z^v}{\phi_t(q_1, q_2)} }
    }
    + \abs{ \exptover{Z,Z^u,Z^v}{\phi_t(q_1, q_2)} - \exptover{u^*, v^*}{\exptover{Z,Z^u,Z^v}{\phi_t(q_1, q_2)}} }
    } \notag \\
    &\le \frac{C}{\sqrt{n}} , \notag 
\end{align}
concluding the proof.
\end{proof}

Suppose $ a,b\ge0 $ are constants. 
With $ q_1 = s_n a^2, q_2 = s_n b^2 $, we can write $ \cH_{n,t} $ in \Cref{eqn:H_nt} as
\begin{align}
    \cH_{n,t}(\wt{u}, \wt{v}; s_n a^2, s_n b^2)
    &= \cH_{n,t}(\wt{u}, \wt{v}) + \cH_{n,a}^u (\wt{u}) + \cH_{n,b}^v (\wt{v}) , \notag 
\end{align}
where
\begin{align}
    \cH_{n,t}(\wt{u}, \wt{v}) &\coloneqq \sqrt{\frac{1-t}{n}} \wt{u}^\top Z \wt{v} + \frac{1-t}{n} \wt{u}^\top \wt{u}^* \wt{v}^\top \wt{v}^* - \frac{1-t}{2n} \normtwo{\wt{u}}^2 \normtwo{\wt{vv}}^2 , \notag \\
    \cH_{n,a}^u (\wt{u}) &\coloneqq \alpha s_n a^2 \wt{u}^\top \wt{u}^* + a \sqrt{\alpha s_n} \wt{u}^\top Z^u - \frac{\alpha s_n a^2}{2} \normtwo{\wt{u}}^2 , \notag \\
    \cH_{n,b}^v (\wt{v}) &\coloneqq s_n b^2 \wt{v}^\top \wt{v}^* + b \sqrt{s_n} \wt{v}^\top Z^v - \frac{s_n b^2}{2} \normtwo{\wt{v}}^2 . \notag 
\end{align}

Fix $ A\ge2 $. 
Define
\begin{align}
    &&
    \phi^u(a) &= \frac{1}{n s_n} \log \int_{\bbR^n} \exp\paren{ \cH_{n,a}^u(\wt{u}) } \diff \wt{P}(\wt{u}) , 
    &
    \phi^v(b) &= \frac{1}{n s_n} \log \int_{\bbR^d} \exp\paren{ \cH_{n,b}^v(\wt{v}) } \diff \wt{Q}(\wt{v}) , 
    &&
    \notag \\
    &&
    \xi_n^u(s_n) &\coloneqq \sup_{1/2\le a\le A + 1/2} \expt{ \abs{ \phi^u(a) - \expt{\phi^u(a)} } } , 
    &
    \xi_n^v(s_n) &\coloneqq \sup_{1/2\le b\le A + 1/2} \expt{ \abs{ \phi^v(b) - \expt{\phi^v(b)} } } , 
    &&
    \label{eqn:xi^uv} \\
    &&
    \Upsilon^u(\wt{u}) &\coloneqq \frac{1}{n s_n} \frac{\partial}{\partial a} \cH_{n,a}^u(\wt{u}) , 
    &
    \Upsilon^v(\wt{v}) &\coloneqq \frac{1}{n s_n} \frac{\partial}{\partial b} \cH_{n,b}^b(\wt{v}) . 
    &&
    \notag 
\end{align}
Denote by $ \bracket{g(u^*, v^*)}_{n,a,b} $ the conditional expectation of $ g(u^*, v^*) $ given 
\begin{align}
    & Y_t, \quad 
    Y_t^u(a) \coloneqq a \sqrt{\alpha s_n} u^* + \Xi^{1/2} Z^u , \quad 
    Y_t^v(b) \coloneqq b \sqrt{s_n} v^* + \Sigma^{1/2} Z^v , 
    \label{eqn:gibbs_nab}
\end{align}
where the expectation is with respect to the distribution in \Cref{eqn:interp_distr}.

\begin{corollary}
\label{cor:xi^uv}
Let $ s_n = n^{-1/32} $ and $ A\le\sqrt{M/s_n} - 1/2 $ for a constant $M>0$ independent of $n$. 
Then there exists a constant $ C>0 $ depending only on $K,\alpha, \ol{\Xi}$ such that for all sufficiently large $n$, 
\begin{align}
    \xi_n^u(s_n) &\le \frac{C}{s_n \sqrt{n}} . \notag 
\end{align}
\end{corollary}

\begin{proof}
Note that if $ b = 0, t = 1 $, it holds that $ \cH_{n,1}(\wt{u}, \wt{v}; s_n a^2, 0) = \cH_{n,a}^u(\wt{u}) $ and $ \phi_1(s_n a^2, 0) = s_n \phi^u(a) $. 
The conclusion then follows immediately from \Cref{lem:free_energy_conc} since by the assumption on $A$, $ q_1 = s_n a^2 \in [0, M] $ for any $ 1/2\le a\le A+1/2 $. 
\end{proof}

\begin{lemma}
\label{lem:Upsilon}
Let $ s_n = n^{-1/32} $. 
For all $ A\ge2 $, 
\begin{align}
    \frac{1}{A - 1} \int_1^A \expt{ \bracket{ \abs{ \Upsilon^u(\wt{u}) - \expt{\bracket{ \Upsilon^u(\wt{u}) }_{n,a,b}} } }_{n,a,b} } \diff a
    &\le C\paren{ \frac{1}{\sqrt{n s_n}} + \sqrt{\xi_n^u(s_n)} } , \notag 
\end{align}
where $C>0$ only depends on $ \alpha, \ol{\Xi}, K $. 
\end{lemma}

\begin{proof}
By the triangle inequality, 
\begin{align}
\begin{split}
    \expt{ \bracket{ \abs{ \Upsilon^u(\wt{u}) - \expt{\bracket{ \Upsilon^u(\wt{u}) }_{n,a,b}} } }_{n,a,b} } 
    &\le \expt{ \bracket{ \abs{ \Upsilon^u(\wt{u}) - \bracket{\Upsilon^u(\wt{u})}_{n,a,b} } }_{n,a,b} } \\
    &\quad + \expt{ \bracket{ \abs{ \bracket{\Upsilon^u(\wt{u})}_{n,a,b} - \expt{ \bracket{\Upsilon^u(\wt{u})}_{n,a,b} } } }_{n,a,b} } . 
\end{split} \label{eqn:Upsilon12}
\end{align}
We will bound the two terms on the RHS separately. 
We first bound 
\begin{align}
    & \frac{1}{A - 1} \int_1^A \expt{ \bracket{ \abs{ \Upsilon^u(\wt{u}) - \bracket{\Upsilon^u(\wt{u})}_{n,a,b} } }_{n,a,b} } \diff a \notag \\
    &\le \paren{ \frac{1}{A-1} \int_1^A \expt{ \bracket{ \paren{ \Upsilon^u(\wt{u}) - \bracket{\Upsilon^u(\wt{u})}_{n,a,b} }^2 }_{n,a,b} } \diff a }^{1/2} . \label{eqn:Upsilon1}
\end{align}
The first two derivatives of $\phi^u$ are
\begin{subequations}
\begin{align}
\begin{split}
    (\phi^u)'(a) &= \frac{1}{n s_n} \frac{\int_{\bbR^n} \exp\paren{ \cH_{n,a}^u(\wt{u}) } \frac{\partial}{\partial a} \cH_{n,a}^u(\wt{u}) \diff \wt{P}(\wt{u})}{\int_{\bbR^n} \exp\paren{ \cH_{n,a}^u(\wt{u}) } \diff \wt{P}(\wt{u})}
    = \bracket{\Upsilon^u(\wt{u})}_{n,a,b} , 
\end{split} \label{eqn:phi'} \\
\begin{split}
    (\phi^u)''(a) &= \frac{1}{n s_n} \brack{ \bracket{\paren{\frac{\partial}{\partial a} \cH_{n,a}^u(\wt{u})}^2}_{n,a,b} - \bracket{\frac{\partial}{\partial a} \cH_{n,a}^u(\wt{u})}_{n,a,b}^2 + \bracket{\frac{\partial^2}{\partial a^2} \cH_{n,a}^u(\wt{u})}_{n,a,b} } \\
    &= n s_n \brack{ \bracket{\Upsilon^u(\wt{u})^2}_{n,a,b} - \bracket{\Upsilon^u(\wt{u})}_{n,a,b}^2 }
    + \frac{\alpha}{n} \brack{ 2 \bracket{\wt{u}^\top \wt{u}^*}_{n,a,b} - \bracket{\normtwo{\wt{u}}^2}_{n,a,b} } . 
\end{split} \label{eqn:phi''}
\end{align}
\end{subequations}
Since there exists $C>0$ depending only on $ \alpha, \ol{\Xi}, K $ such that for all sufficiently large $n$, 
\begin{align}
    \abs{ \frac{\alpha}{n} \brack{ 2 \bracket{\wt{u}^\top \wt{u}^*}_{n,a,b} - \bracket{\normtwo{\wt{u}}^2}_{n,a,b} } } &\le \frac{4 \alpha K^2}{\inf\supp(\ol{\Xi})}
    \eqqcolon C , \label{eqn:err_C}
\end{align}
the second result \Cref{eqn:phi''} above implies that for all sufficiently large $n$,
\begin{align}
    \bracket{ \paren{ \Upsilon^u(\wt{u}) - \bracket{\Upsilon^u(\wt{u})}_{n,a,b} }^2 }_{n,a,b}
    &\le \frac{1}{n s_n} \paren{ (\phi^u)''(a) + C } . \notag 
\end{align}
Consequently, 
\begin{align}
    \int_1^A \expt{ \bracket{ \paren{ \Upsilon^u(\wt{u}) - \bracket{\Upsilon^u(\wt{u})}_{n,a,b} }^2 }_{n,a,b} } \diff a
    &\le \frac{1}{n s_n} \paren{ \expt{(\phi^u)'(A)} - \expt{(\phi^u)'(1)} + C(A-1) } . \label{eqn:phi'_ub} 
\end{align}
To proceed, we compute for any $a$, 
\begin{align}
    \expt{ (\phi^u)'(a) }
    &= \expt{ \bracket{\Upsilon^u(\wt{u})}_{n,a,b} } , \label{eqn:phi'_comp} 
\end{align}
which is by \Cref{eqn:phi'}.
By definition, 
\begin{align}
    \Upsilon^u(\wt{u}) &= \frac{1}{n} \paren{ 2\alpha a \wt{u}^\top \wt{u}^* + \sqrt{\alpha / s_n} \wt{u}^\top Z^u - \alpha a \normtwo{\wt{u}}^2 } , \notag 
\end{align}
whose expectation is therefore given by
\begin{align}
    \expt{\bracket{\Upsilon^u(\wt{u})}_{n,a,b}}
    &= \frac{2\alpha a}{n} \expt{ \bracket{\wt{u}^\top \wt{u}^*}_{n,a,b} } 
    + \frac{\sqrt{\alpha}}{n\sqrt{s_n}} \expt{\bracket{\wt{u}^\top Z^u}_{n,a,b}} 
    - \frac{\alpha a}{n} \expt{\bracket{\normtwo{\wt{u}}^2}_{n,a,b}}
    . \label{eqn:Upsilon_expand} 
\end{align}
Using Stein's lemma (\Cref{prop:stein}), the middle term is equal to 
\begin{align}
    \frac{\sqrt{\alpha}}{n\sqrt{s_n}} \expt{\bracket{\wt{u}^\top Z^u}_{n,a,b}}
    &= \frac{\sqrt{\alpha}}{n\sqrt{s_n}} \sum_{i = 1}^n \expt{ Z^u_i \bracket{\wt{u}_i}_{n,a,b} }
    = \frac{\sqrt{\alpha}}{n\sqrt{s_n}} \sum_{i = 1}^n \expt{ \frac{\partial}{\partial Z^u_i} \bracket{\wt{u}_i}_{n,a,b} } \notag \\
    &= \frac{a \alpha}{n} \expt{\bracket{\normtwo{\wt{u}}^2}_{n,a,b}} - \frac{a\alpha}{n} \expt{\bracket{\wt{u}^\top \wt{u}^*}_{n,a,b}} . \notag 
\end{align}
Therefore, 
\begin{align}
    \expt{\bracket{\Upsilon^u(\wt{u})}_{n,a,b}}
    &= \frac{\alpha a}{n} \expt{ \bracket{\wt{u}^\top \wt{u}^*}_{n,a,b} } , \label{eqn:Upsilon} 
\end{align}
and 
\begin{align}
    \abs{ \expt{\bracket{\Upsilon^u(\wt{u})}_{n,a,b}} }
    &\le a C , \label{eqn:phi'_bd} 
\end{align}
where $C$ depends only on $ \alpha, \ol{\Xi}, K $. 

Using the last inequality, we can further upper bound the RHS of \Cref{eqn:phi'_ub} by $ \frac{C A}{n s_n} $ for some $C$ depending only on $ \alpha, \ol{\Xi}, K $. 
Putting this back to \Cref{eqn:Upsilon1}, we get
\begin{align}
    \frac{1}{A - 1} \int_1^A \expt{ \bracket{ \abs{ \Upsilon^u(\wt{u}) - \bracket{\Upsilon^u(\wt{u})}_{n,a,b} } }_{n,a,b} } \diff a
    &\le \sqrt{\frac{C A}{n s_n (A - 1)}}
    \le \sqrt{\frac{2C}{n s_n}} , \label{eqn:Upsilon1_done} 
\end{align}
since $ A\ge2 $. 

Now it remains to bound the second term on the RHS of \Cref{eqn:Upsilon12} which can be written as 
\begin{align}
    \expt{ \bracket{ \abs{ \bracket{\Upsilon^u(\wt{u})}_{n,a,b} - \expt{ \bracket{\Upsilon^u(\wt{u})}_{n,a,b} } } }_{n,a,b} }
    &= \expt{ \abs{ (\phi^u)'(a) - \expt{(\phi^u)'(a)} } } \label{eqn:Upsilon2_todo}
\end{align}
using \Cref{eqn:phi'}. 
To further bound the RHS, consider the following two functions 
\begin{align}
    a &\mapsto \phi^u(a) + \frac{2 \alpha K^2 a^2}{\inf\supp(\ol{\Xi})} , \quad 
    a \mapsto \expt{\phi^u(a)} + \frac{2 \alpha K^2 a^2}{\inf\supp(\ol{\Xi})} . \notag 
\end{align}
They are both differentiable and convex for all sufficiently large $n$ since their second derivatives are non-negative by \Cref{eqn:phi'',eqn:err_C}. 
Applying \Cref{prop:panchenko} with the above two functions, taking the expectation and using the triangle inequality, we have that for any $ 1\le a\le A , 0<a'\le1/2 $,
\begin{align}
    \expt{ \abs{ (\phi^u)'(a) - \expt{(\phi^u)'(a)} } }
    &\le \expt{(\phi^u)'(a + a')} - \expt{(\phi^u)'(a - a')}
    + 3 \xi_n^u(s_n) / a' + \frac{8 \alpha K^2 a'}{\inf\supp(\ol{\Xi})} . \label{eqn:ub2} 
\end{align}
Then
\begin{align}
    & \int_1^A \expt{(\phi^u)'(a + a')} - \expt{(\phi^u)'(a - a')} \diff a \notag \\
    &= \paren{ \expt{\phi^u(A + a')} - \expt{\phi^u(1 + a')} }
    - \paren{ \expt{\phi^u(A - a')} - \expt{\phi^u(1 - a')} } \notag \\
    &= \paren{ \expt{\phi^u(A + a')} - \expt{\phi^u(A - a')} }
    - \paren{ \expt{\phi^u(1 + a')} - \expt{\phi^u(1 - a')} } \notag \\
    &= \int_{-a'}^{a'} \expt{(\phi^u)'(A + a)} \diff a
    - \int_{-a'}^{a'} \expt{(\phi^u)'(1 + a)} \diff a \notag \\
    &\le 4a'(A + a') C
    \le 8 a' A C , \notag 
\end{align}
where the last step is by \Cref{eqn:phi'_comp,eqn:phi'_bd}, and $C$ depends only on $\alpha, \ol{\Xi}, K$. 
Using this in \Cref{eqn:ub2}, we obtain
\begin{align}
    \int_1^A \expt{ \abs{ (\phi^u)'(a) - \expt{(\phi^u)'(a)} } } \diff a
    &\le C A \paren{ a' + \xi_n^u(s_n) / a' } , \notag 
\end{align}
and the RHS is minimized by $ a' = \sqrt{ \xi_n^u(s_n) } $ which lies in the interval $ (0, 1/2] $ for all sufficiently large $n$ due to \Cref{cor:xi^uv}. 
Using this result in \Cref{eqn:Upsilon2_todo} and integrating over $a$, we have
\begin{align}
    \frac{1}{A-1} \int_1^A \expt{ \bracket{ \abs{ \bracket{\Upsilon^u(\wt{u})}_{n,a,b} - \expt{ \bracket{\Upsilon^u(\wt{u})}_{n,a,b} } } }_{n,a,b} } \diff a
    &\le \frac{C A}{A-1} \cdot 2 \sqrt{\xi_n^u(s_n)}
    \le {4 C} \sqrt{\xi_n^u(s_n)} , \label{eqn:Upsilon2_done} 
\end{align} 
since $A\ge2$.  

Finally, combining \Cref{eqn:Upsilon1_done,eqn:Upsilon2_done,eqn:Upsilon12} proves the lemma. 
\end{proof}

\begin{lemma}
\label{lem:u12_conc}
There exists $C>0$ depending only on $\alpha, K, \ol{\Xi}$ such that for any $A\ge2$, 
\begin{align}
    \frac{1}{A - 1} \int_1^A \frac{1}{n^2} \expt{ \bracket{ \paren{ (\wt{u}^{(1)})^\top \wt{u}^{(2)} - \expt{\bracket{ (\wt{u}^{(1)})^\top \wt{u}^{(2)} }_{n,a,b}} }^2 }_{n,a,b} } \diff a
    &\le C\paren{ \frac{1}{\sqrt{n s_n}} + \sqrt{\xi_n^u(s_n)} } , \notag 
\end{align}
where $ \wt{u}^{(1)} = \Xi^{-1/2} u^{(1)} , \wt{u}^{(2)} = \Xi^{-1/2} u^{(2)} $ with $ u^{(1)}, u^{(2)} $ being two i.i.d.\ copies from the conditional law of $ u^*, v^* $ given \Cref{eqn:gibbs_nab}. 
\end{lemma}

\begin{proof}
Since $ \supp(P), \supp(Q), \normtwo{\Xi}^{-1} $ are all bounded, by the triangle inequality, 
\begin{align}
    & \frac{1}{n} \abs{ \expt{\bracket{ \Upsilon^u(\wt{u}^{(1)}) (\wt{u}^{(1)})^\top \wt{u}^{(2)} }_{n,a,b}} - \expt{\bracket{\Upsilon^u(\wt{u})}_{n,a,b}} \expt{\bracket{(\wt{u}^{(1)})^\top \wt{u}^{(2)}}_{n,a,b}} } \label{eqn:abs_comp} \\
    &\le \normtwo{\Xi}^{-1} K^2 \expt{ \bracket{\abs{ \Upsilon^u(\wt{u}) - \expt{\bracket{\Upsilon^u(\wt{u})}_{n,a,b}} }}_{n,a,b} } . \label{eqn:ineq} 
\end{align}

On the other hand, let us compute \Cref{eqn:abs_comp} on the LHS of the above inequality. 
Recall from \Cref{eqn:Upsilon} in the proof of \Cref{lem:Upsilon} that 
\begin{align}
    \expt{\bracket{\Upsilon^u(\wt{u}^{(1)})}_{n,a,b}} = \frac{\alpha a}{n} \expt{ \bracket{ (\wt{u}^{(1)})^\top \wt{u}^{(2)} }_{n,a,b} } . 
    \label{eqn:Upsilon_again}
\end{align}
Similar to \Cref{eqn:Upsilon_expand}, we have
\begin{multline}
    \expt{\bracket{ \Upsilon^u(\wt{u}^{(1)}) (\wt{u}^{(1)})^\top \wt{u}^{(2)} }_{n,a,b}} 
    = \frac{2\alpha a}{n} \expt{\bracket{ (\wt{u}^{(1)})^\top \wt{u}^* (\wt{u}^{(1)})^\top \wt{u}^{(2)} }_{n,a,b}} \notag \\
    + \frac{\sqrt{\alpha}}{n\sqrt{s_n}} \expt{\bracket{ (\wt{u}^{(1)})^\top Z^u (\wt{u}^{(1)})^\top \wt{u}^{(2)} }_{n,a,b}} 
    - \frac{\alpha a}{n} \expt{\bracket{\normtwo{\wt{u}^{(1)}}^2 (\wt{u}^{(1)})^\top \wt{u}^{(2)}}_{n,a,b}} . \notag 
\end{multline}
Using Stein's lemma and following similar derivations leading to \Cref{eqn:Upsilon}, the second term on the RHS equals
\begin{multline}
    \frac{\sqrt{\alpha}}{n\sqrt{s_n}} \expt{\bracket{ (\wt{u}^{(1)})^\top Z^u (\wt{u}^{(1)})^\top \wt{u}^{(2)} }_{n,a,b}}
    = \frac{a \alpha}{n} \expt{\bracket{ \normtwo{\wt{u}^{(1)}}^2 (\wt{u}^{(1)})^\top \wt{u}^{(2)}}_{n,a,b}} \notag \\
    - \frac{2 a \alpha}{n} \expt{ \bracket{ (\wt{u}^{(1)})^\top \wt{u}^{(3)} (\wt{u}^{(1)})^\top \wt{u}^{(2)} }_{n,a,b} } 
    + \frac{a \alpha}{n} \expt{ \bracket{\paren{ (\wt{u}^{(1)})^\top \wt{u}^{(2)} }^2}_{n,a,b} } . \notag 
\end{multline}
Therefore, 
\begin{align}
    \expt{\bracket{ \Upsilon^u(\wt{u}^{(1)}) (\wt{u}^{(1)})^\top \wt{u}^{(2)} }_{n,a,b}} 
    &= \frac{\alpha a}{n} \expt{\bracket{ \paren{(\wt{u}^{(1)})^\top \wt{u}^{(2)}}^2 }_{n,a,b}} . \label{eqn:Upsilon_u12} 
\end{align}

Putting \Cref{eqn:Upsilon_again,eqn:Upsilon_u12} together and using Nishimori identity (\Cref{prop:nishimori}), we have that \Cref{eqn:abs_comp} equals
\begin{multline}
    \frac{\alpha a}{n^2} \abs{ \expt{\bracket{ \paren{(\wt{u}^{(1)})^\top \wt{u}^{(2)}}^2 }_{n,a,b}} - \expt{ \bracket{(\wt{u}^{(1)})^\top \wt{u}^{(2)}}_{n,a,b} }^2 } \notag \\
    = \frac{\alpha a}{n^2} \abs{ \expt{\bracket{\paren{ (\wt{u}^{(1)})^\top \wt{u}^{(2)} - \expt{\bracket{(\wt{u}^{(1)})^\top \wt{u}^{(2)}}_{n,a,b}} }^2}_{n,a,b}} } . 
    \notag 
\end{multline}
So by the inequality \Cref{eqn:ineq}, for any $a\ge1$, 
\begin{align}
    \frac{1}{n^2} \expt{\bracket{\paren{ (\wt{u}^{(1)})^\top \wt{u}^{(2)} - \expt{\bracket{(\wt{u}^{(1)})^\top \wt{u}^{(2)}}_{n,a,b}} }^2}_{n,a,b}} 
    &\le \frac{\normtwo{\Xi}^{-1} K^2}{\alpha} \expt{ \bracket{\abs{ \Upsilon^u(\wt{u}) - \expt{\bracket{\Upsilon^u(\wt{u})}_{n,a,b}} }}_{n,a,b} } . \notag 
\end{align}
Integrating over $a\in[1,A]$ and invoking \Cref{lem:Upsilon} concludes the proof.
\end{proof}

\begin{lemma}[Overlap concentration]
\label{lem:ol_conc}
Let $ R_1, R_2\colon[0,1]\times\bbR_{>0}^2\to\bbR_{\ge0} $ be two continuous bounded functions such that their partial derivatives with respect to the second and third arguments are continuous and non-negative. 
Let $ s_n = n^{-1/32} $. 
For $ (\eps_1, \eps_2)\in[1,2]^2 $, slightly abusing notation, let $ q_1(\cdot, \eps_1, \eps_2), q_2(\cdot, \eps_1, \eps_2) $ be the unique solution to 
\begin{align}
\begin{cases}
q_1(0) = s_n \eps_1 \\
q_2(0) = s_n \eps_2
\end{cases} , \qquad 
\begin{cases}
q_1'(t) = R_1(t, q_1(t), q_2(t)) \\
q_2'(t) = R_2(t, q_1(t), q_2(t)) 
\end{cases} . \label{eqn:cauchy_problem} 
\end{align}
Then there exists a constant $  C>0$ depending only on $K$, $ \alpha $, $ \norminf{R_1} $, $ \norminf{R_2} $, $ \ol{\Xi} $ such that for every $ t\in[0,1] $, 
\begin{align}
\int_1^2 \int_1^2 \frac{1}{n^2} \expt{\bracket{\paren{\wt{u}^\top \wt{u}^* - \expt{\bracket{\wt{u}^\top \wt{u}^*}_{n,t}}}^2}_{n,t}} \diff\eps_1 \diff\eps_2 \le C n^{-1/8} . \notag 
\end{align}
\end{lemma}

\begin{proof}
The existence and uniqueness of the solution to the Cauchy problem \Cref{eqn:cauchy_problem} is a direct consequence of the Cauchy--Lipschitz theorem \cite[Theorem 3.1, Chapter V]{ODE_book}.
For any $t\in[0,1]$, the function $ Q_t(\eps_1, \eps_2) = (q_1(t, \eps_1, \eps_2), q_2(t, \eps_1, \eps_2)) $ is a $C^1$-diffeomorphism. 
Its Jacobian determinant is given by the Liouville formula \cite[Corollary 3.1, Chapter V]{ODE_book} and can be lower bounded as 
\begin{align}
    J(\eps_1, \eps_2) &\coloneqq \det(\nabla Q_t(\eps_1, \eps_2)) \notag \\
    &= s_n^2 \exp\paren{ \int_0^t \partial_2 R_1(s, Q_s(\eps_1, \eps_2)) \diff s + \int_0^t \partial_3 R_2(s, Q_s(\eps_1, \eps_2)) \diff s }
    \ge s_n^2 , \label{eqn:jacobian_det} 
\end{align}
since the partial derivatives are non-negative by assumptions. 

We then view the RHS below as a function of $ q_1, q_2 $ and denote it by
\begin{align}
    V(q_1, q_2) &= \expt{\bracket{\paren{\wt{u}^\top \wt{u}^* - \expt{\bracket{\wt{u}^\top \wt{u}^*}_{n,t}}}^2}_{n,t}}. \label{eqn:V} 
\end{align}
Denote $ \Omega_t = Q_t([1,2]^2)/s_n $ and $ M \coloneqq \max\brace{ \norminf{R_1}, \norminf{R_2} } + 2 $. 
Since $ R_1, R_2\ge0 $ by assumptions, $ q_1, q_2 $ are non-decreasing in $t$ by \Cref{eqn:cauchy_problem}. 
So for $ i\in\{1,2\} $, and any $t\in[0,1]$ and $ (\eps_1, \eps_2)\in[1,2]^2 $, 
\begin{align}
    q_i(t, \eps_1, \eps_2) / s_n 
    &\ge q_i(0, \eps_1, \eps_2) / s_n = \eps_i \ge 1 , \notag \\ 
    q_i(t, \eps_1, \eps_2) / s_n
    &\le q_i(1, \eps_1, \eps_2) / s_n
    = s_n^{-1} \int_0^1 q_i'(t, \eps_1, \eps_2) \diff t + s_n^{-1} q_i(0, \eps_1, \eps_2)
    \le s_n^{-1} (\norminf{R_i} + 2) . \notag 
\end{align}
We obtain the relation $ \Omega_t\subset[1, M/s_n]^2 $ for any $t\in[0,1]$. 

Next, using the change of variable $ (r_1, r_2) = Q_t(\eps_1, \eps_2)/s_n $, we have
\begin{align}
    & \int_1^2 \int_1^2 \frac{1}{n^2} \expt{\bracket{\paren{\wt{u}^\top \wt{u}^* - \expt{\bracket{\wt{u}^\top \wt{u}^*}_{n,t}}}^2}_{n,t}} \diff\eps_1 \diff\eps_2 \notag \\
    &= \frac{1}{n^2} \int_1^2 \int_1^2 V(q_1(t, \eps_1, \eps_2), q_2(t, \eps_1, \eps_2)) \diff\eps_1\diff\eps_2 \notag \\
    &= \frac{1}{n^2} \int_{\Omega_t} \frac{V(s_n r_1, s_n r_2) s_n^2}{J(Q_t^{-1}(s_n r_1, s_n r_2))} \diff(r_1, r_2) \notag \\
    &\le \frac{1}{n^2} \int_1^{M/s_n} \int_1^{M/s_n} V(s_n r_1, s_n r_2) \diff r_1 \diff r_2 , \label{eqn:back2} 
\end{align}
where the last step is by \Cref{eqn:jacobian_det}. 
Further applying the change of variable $ r_1 = a^2 $, we have that for all $ r_2\ge1 $, 
\begin{align}
    \frac{1}{n^2} \int_1^{M/s_n} V(s_n r_1, s_n r_2) \diff r_1
    &= \frac{1}{n^2} \int_1^{\sqrt{M/s_n}} V(s_n a^2, s_n r_2) 2a \diff a
    \le \frac{2\sqrt{M/s_n}}{n^2} \int_1^{\sqrt{M/s_n}} V(s_n a^2, s_n r_2) \diff a . \label{eqn:back1} 
\end{align}
Recalling $ V $ from \Cref{eqn:V}, we recognize that 
\begin{align}
    V(s_n a^2, s_n r_2) &= 
    \expt{ \bracket{ \paren{ \wt{u}^\top \wt{u}^* - \expt{\bracket{\wt{u}^\top \wt{u}^*}_{n,a,\sqrt{r_2}}} }^2 }_{n,a,\sqrt{r_2}} } \notag \\
    &= \expt{ \bracket{ \paren{ (\wt{u}^{(1)})^\top \wt{u}^{(2)} - \expt{\bracket{(\wt{u}^{(1)})^\top \wt{u}^{(2)}}_{n,a,\sqrt{r_2}}} }^2 }_{n,a,\sqrt{r_2}} } , 
    \notag 
\end{align}
where $ \bracket{\cdot}_{n,a,b} $ is defined in \Cref{eqn:gibbs_nab} and the second equality is by Nishimori identity (\Cref{prop:nishimori}).
Since $\sqrt{M/s_n} \ge 2$ for all sufficiently large $n$, applying \Cref{lem:u12_conc}, we get
\begin{align}
    \frac{1}{\sqrt{M/s_n} - 1} \int_1^{\sqrt{M/s_n}} \frac{1}{n^2} V(s_n a^2, s_n r_2) \diff a
    &\le C \paren{ \frac{1}{\sqrt{n s_n}} + \sqrt{\xi_n^u(s_n)} } , \notag 
\end{align}
where $C>0$ depends only on $ \alpha, K, \ol{\Xi} $ and $ \xi_n^u(s_n) $ is given in \Cref{eqn:xi^uv} with $A = \sqrt{M/s_n}$. 
Using this back in \Cref{eqn:back1} and then in \Cref{eqn:back2}, we obtain
\begin{align}
    \int_1^2 \int_1^2 \frac{1}{n^2} \expt{\bracket{\paren{\wt{u}^\top \wt{u}^* - \expt{\bracket{\wt{u}^\top \wt{u}^*}_{n,t}}}^2}_{n,t}} \diff\eps_1 \diff\eps_2
    &\le 2 C (M/s_n)^2 (1/\sqrt{n s_n} + \sqrt{\xi_n^u(s_n)}) . \notag 
\end{align}
\Cref{cor:xi^uv} guarantees that $ \xi_n^u(s_n) \le \frac{C}{s_n \sqrt{n}} $ for some $C>0$ depending only on $\alpha, K, \ol{\Xi}$. 
By the choice of $s_n$, we can finally upper bound (up to a positive constant depending only on $ \alpha, K, \ol{\Xi}, M $) the RHS above by
\begin{align}
    \frac{1}{s_n^2} \paren{ \frac{1}{\sqrt{n s_n}} + \frac{1}{\sqrt{s_n \sqrt{n}}} }
    \le \frac{1}{s_n^{2.5}} \cdot \frac{2}{n^{1/4}}
    = 2 \cdot n^{-11/64}
    \le 2 \cdot n^{-1/8} , \notag 
\end{align}
for all sufficiently large $n$, which completes the proof. 
\end{proof}

Recalling $ \bracket{\cdot}_{n,t} $ defined in \Cref{eqn:gibbs_brack}, let us identify $ \expt{\bracket{\wt{u}^\top \wt{u}^*}_{n,t}} $ as a function of $ (t, q_1, q_2) $: 
\begin{align}
\frac{1}{n} \expt{\bracket{\wt{u}^\top \wt{u}^*}_{n,t}} &= \Delta(t, q_1, q_2) . \label{eqn:Delta}
\end{align}
Note that $ \Delta $ is continuous, non-negative (by Nishimori identity) on $ [0,1] \times \bbR_{\ge0}^2 $ and bounded by $ K^2 $.
Its partial derivatives with respect to the second and third arguments are continuous and non-negative, since the correlation between $ \wt{u}^* $ and $ \bracket{\wt{u}}_{n,t} $ is a non-decreasing function of the SNRs $ q_1, q_2 $. 

\begin{lemma}[Fundamental sum rule]
\label{lem:sum_rule}
In the setting of \Cref{lem:ol_conc}, for $ (\eps_1, \eps_2)\in[1,2]^2 $, let $ q_1(t, \eps_1, \eps_2), q_2(t, \eps_1, \eps_2) $ be the solution to \Cref{eqn:cauchy_problem} with $ R_2 = \Delta $ defined in \Cref{eqn:Delta}.
Then we have
\begin{align}
\cF_n &= \int_1^2 \int_1^2 \int_0^1 \psi_{\ol{\Xi}}(\alpha \snr q_1(1, \eps_1, \eps_2)) + \alpha \psi_{\ol{\Sigma}}(q_2(1, \eps_1, \eps_2)) 
- \frac{\alpha}{2} q_1'(t, \eps_1, \eps_2) q_2'(t, \eps_1, \eps_2) \diff t \diff\eps_1\diff\eps_2 + o(1) . \notag 
\end{align} 
\end{lemma}

\begin{proof}
Fix $ (\eps_1, \eps_2)\in[1,2]^2 $.
By the choice $ R_2 = \Delta $, 
\begin{align}
    q_2'(t, \eps_1, \eps_2) &= \Delta(t, q_1(t, \eps_1, \eps_2), q_2(t, \eps_1, \eps_2))
    = \frac{1}{n} \expt{\bracket{\wt{u}^\top \wt{u}^*}_{n,t}} . \notag 
\end{align}
Plugging this into \Cref{eqn:free_energy_var}, integrating over $(\eps_1, \eps_2)\in[1,2]^2$ and applying \Cref{lem:ol_conc}, we have 
\begin{align}
    f_n'(t) &= \frac{\alpha}{2} q_1'(t) q_2'(t) + o(1) , \notag 
\end{align}
where $ o(1) \to 0 $ as $n\to\infty$, uniformly over $ t $. 
By \Cref{lem:f01}, we conclude
\begin{align}
    \cF_n &= \int_1^2 \int_1^2 f_n(0) \diff\eps_1 \diff\eps_2 + o(1)
    = \int_1^2 \int_1^2 \paren{ f_n(1) - \int_0^1 f_n'(t) \diff t } \diff\eps_1 \diff\eps_2 + o(1) \notag \\
    &= \int_1^2 \int_1^2 \int_0^1 \psi_{\ol{\Xi}}(\alpha \snr q_1(1, \eps_1, \eps_2)) + \alpha \psi_{\ol{\Sigma}}(q_2(1, \eps_1, \eps_2)) 
    - \frac{\alpha}{2} q_1'(t, \eps_1, \eps_2) q_2'(t, \eps_1, \eps_2) \diff t \diff\eps_1\diff\eps_2 + o(1) , \notag 
\end{align}
as desired. 
Here the first and last equalities are by \Cref{eqn:f=F,eqn:f1}, respectively. 
\end{proof}

Finally, we prove a pair of matching upper and lower bounds, completing the proof of \Cref{thm:free_energy}. 

\begin{lemma}[Lower bound]
\label{lem:lb}
Let $ s_n = n^{-1/32} $ and $ R_2 = \Delta $. 
Then 
\begin{align}
\liminf_{n\to\infty} \cF_n &\ge \sup_{q_v\ge0} \inf_{q_u\ge0} \cF( q_u , q_v ) . \notag 
\end{align}
\end{lemma}

\begin{proof}
Fix an arbitrary $ q_v\ge0 $. 
Let $ R_1 = q_v $. 
Then $ q_1(t, \eps_1, \eps_2) = s_n \eps_1 + t q_v $ and \Cref{lem:sum_rule} gives
\begin{align}
    \cF_n &= \int_1^2 \int_1^2 \int_0^1 \psi_{\ol{\Xi}}(\alpha \snr (s_n\eps_1 + q_v)) + \alpha \psi_{\ol{\Sigma}}(q_2(1, \eps_1, \eps_2)) - \frac{\alpha}{2} q_v q_2'(t, \eps_1, \eps_2) \diff t \diff\eps_1 \diff\eps_2 + o(1) \notag \\
    &= \int_1^2 \int_1^2 \psi_{\ol{\Xi}}(\alpha \snr q_v) + \alpha\psi_{\ol{\Sigma}}(q_2(1, \eps_1, \eps_2)) - \frac{\alpha}{2} q_v q_2(1, \eps_1, \eps_2) \diff\eps_1 \diff\eps_2 + o(1) \notag \\
    &\ge \inf_{q_2\ge0} \psi_{\ol{\Xi}} (\alpha \snr q_v) + \alpha \psi_{\ol{\Sigma}}(q_2) - \frac{\alpha}{2} q_v q_2 + o(1)
    = \inf_{q_u\ge0} \cF(q_u, q_v) + o(1) , \notag 
\end{align}
where the second line holds since $\psi_{\ol{\Xi}}$ is Lipschitz and $ q_2(0, \eps_1, \eps_2) = s_n \eps_2 = o(1) $. 
This completes the proof since the above lower bound holds for all $q_v\ge0$. 
\end{proof}

\begin{lemma}[Upper bound]
\label{lem:ub}
Let $ s_n = n^{-1/32} $ and $ R_2 = \Delta $. 
Then
\begin{align}
\limsup_{n\to\infty} \cF_n &\le \sup_{q_v\ge0} \inf_{q_u\ge0} \cF( q_u , q_v ) . \notag 
\end{align}
\end{lemma}

\begin{proof}
We apply \Cref{lem:sum_rule} with 
\begin{align}
    R_1(t, q_1, q_2) &= 2\alpha \psi_{\ol{\Sigma}}' \paren{ \Delta(t, q_1, q_2) } . \label{eqn:R1} 
\end{align}

Since $ \psi_{\ol{\Xi}} $ is Lipschitz and convex, 
\begin{align}
    \psi_{\ol{\Xi}}(\alpha \snr q_1(1, \eps_1, \eps_2))
    &= \psi_{\ol{\Xi}}\paren{ \alpha \snr (q_1(1, \eps_1, \eps_2) - q_1(0, \eps_1, \eps_2)) } + o(1)
    = \psi_{\ol{\Xi}} \paren{ \alpha \snr \int_0^1 q_1'(t, \eps_1, \eps_2) \diff t } + o(1) \notag \\
    &\le \int_0^1 \psi_{\ol{\Xi}}(\alpha \snr q_1'(t, \eps_1, \eps_2)) \diff t + o(1) , \notag 
\end{align}
and similarly 
\begin{align}
    \psi_{\ol{\Sigma}}(q_2(1, \eps_1, \eps_2)) &\le \int_0^1 \psi_{\ol{\Sigma}}(q_2'(t, \eps_1, \eps_2)) \diff t + o(1) . \notag 
\end{align}
Now \Cref{lem:sum_rule} implies
\begin{align}
    \cF_n &\le \int_1^2 \int_1^2 \int_0^1 \psi_{\ol{\Xi}}(\alpha \snr q_1'(t, \eps_1, \eps_2)) + \alpha \psi_{\ol{\Sigma}}(q_2'(t, \eps_1, \eps_2)) - \frac{\alpha}{2} q_1'(t, \eps_1, \eps_2) q_2'(t, \eps_1, \eps_2) \diff t \diff\eps_1 \diff\eps_2 + o(1) \notag \\
    &= \int_1^2 \int_1^2 \int_0^1 \cG(q_2'(t, \eps_1, \eps_2), q_1'(t, \eps_1, \eps_2)) \diff t \diff\eps_1 \diff\eps_2 + o(1) , \notag 
\end{align}
where
\begin{align}
    \cG(q_u, q_v) &\coloneqq \psi_{\ol{\Xi}}(\alpha \snr q_v) + \alpha \psi_{\ol{\Sigma}}(q_u) - \frac{\alpha}{2} q_u q_v . \notag 
\end{align}
With the choice of $R_1$ in \Cref{eqn:R1} and $ R_2 = \Delta $, the ODE in \Cref{lem:ol_conc} gives 
\begin{align}
    q_1'(t, \eps_1, \eps_2) &= 2\alpha \psi_{\ol{\Sigma}}'(q_2'(t, \eps_1, \eps_2)) , \notag 
\end{align}
which corresponds to the criticality condition of $ \cG $ with respect to $q_u$: 
\begin{align}
    \partial_1 \cG(q_2'(t, \eps_1, \eps_2), q_1'(t, \eps_1, \eps_2)) &= 0 . \notag 
\end{align}
Since  $ \psi_{\ol{\Sigma}} $ is convex and $ -\frac{\alpha}{2} q_u q_v $ is linear in $ q_2 $, we have that $ \cG $ is convex in $ q_u $. 
Therefore
\begin{align}
    \cG(q_2'(t, \eps_1, \eps_2), q_1'(t, \eps_1, \eps_2))
    &= \inf_{q_u\ge0} \cG(q_u, q_1'(t, \eps_1, \eps_2))
    = \inf_{q_u\ge0} \cF(q_u, q_1'(t, \eps_1, \eps_2))
    \le \sup_{q_v\ge0} \inf_{q_u\ge0} \cF(q_u, q_v) , \notag
\end{align}
which completes the proof. 
\end{proof}

\section{Proofs of \Cref{thm:IT} and \Cref{cor:weak_rec_thr}}\label{app:pfmain}

\subsection{Proof of \Cref{cor:MMSE_mtx}}
\label{app:pf_cor_MMSE_mtx}

We compute the derivative of $ \cF_n(\snr) $: 
\begin{align}
    \cF_n'(\snr) &= \frac{1}{n} \expt{\frac{\cZ_n'(\snr)}{\cZ_n(\snr)}}
    = \frac{1}{n} \expt{ \bracket{
    \frac{1}{2\sqrt{\snr n}} \wt{u}^\top Z \wt{v}
    + \frac{1}{n} \wt{u}^\top \wt{u}^* \wt{v}^\top \wt{v}^*
    - \frac{1}{2n} \wt{u}^\top \wt{u} \wt{v}^\top \wt{v}
    }_n } \notag \\
    &= \frac{1}{2n^2} \expt{ \bracket{ \wt{u}^\top \wt{u}^* \wt{v}^\top \wt{v}^* }_n } , \label{eqn:F'} 
\end{align}
where the last step follows similar calculations in \Cref{eqn:u'Zv}. 
Since $ \cF_n(\snr) \to \sup_{q_v} \inf_{q_u} \cF(q_u, q_v) $ as $n\to\infty$, we have $ \cF_n'(\snr) \to \frac{\partial}{\partial\snr} \sup_{q_v} \inf_{q_u} \cF(q_u, q_v) $.
To compute the RHS, note that 
\begin{align}
    \sup_{q_v\ge0} \inf_{q_u\ge0} \cF(q_u, q_v)
    &= \sup_{q_v\ge0} \brace{
        \psi_{\ol{\Xi}}(\alpha \snr q_v)
        + \inf_{q_u\ge0} \brace{
            \alpha \psi_{\ol{\Sigma}}(\snr q_u) - \frac{\alpha}{2} \snr q_u q_v
        }
    } , \notag 
\end{align}
and the value of the infimum does not depend on $\snr$. 
Therefore, we have 
\begin{align}
    \frac{\partial}{\partial\snr} \sup_{q_v\ge0} \inf_{q_u\ge0} \cF(q_u, q_v)
    &= \psi_{\ol{\Xi}}'(\alpha \snr q_v^*) \alpha q_v^*
    = \frac{\alpha q_u^* q_v^*}{2} , 
    \label{eqn:F'_lim} 
\end{align}
where first equality is by the envelope theorem from \cite[Corollary 4]{envelope} and the last equality follows since the extremizers $ q_u^*, q_v^* $ solve a pair of equations in \Cref{eqn:cQ_solve}. 

On the other hand, we relate $ \cF_n'(\snr) $ to the MMSE \Cref{eqn:MMSE_Y} as follows: 
\begin{align}
    \mmse_n(\snr) &= \frac{1}{nd} \expt{ \normf{ {\wt{u}^*} (\wt{v}^*)^\top - \bracket{ \wt{u} \wt{v}^\top }_n }^2 }
    = \frac{1}{nd} \expt{ \normtwo{\wt{u}^*}^2 \normtwo{\wt{v}^*}^2 + \normf{\bracket{\wt{u} \wt{v}^\top}_n}^2 - 2 (\wt{u}^*)^\top \bracket{\wt{u} \wt{v}^\top}_n \wt{v}^* } \notag \\
    &= \frac{\tr(\Xi^{-1})}{n} \frac{\tr(\Sigma^{-1})}{d} - \frac{1}{nd} \expt{ \bracket{ \wt{u}^\top \wt{u}^* \wt{v}^\top \wt{v}^* }_n } , \label{eqn:MMSE_expand} 
\end{align}
where the last step is by Nishimori identity (\Cref{prop:nishimori}). 
Combining the above with \Cref{eqn:F',eqn:F'_lim}, we conclude
\begin{align}
    \mmse_n(\snr)
    &\to \expt{\ol{\Xi}^{-1}} \expt{\ol{\Sigma}^{-1}} - q_u^* q_v^* , \notag 
\end{align}
as claimed.

\subsection{Proof of \Cref{cor:MMSE_vec}}
\label{app:pf_cor_MMSE_vec}

Recall $Y$ from \Cref{eqn:Y_mtx} and define for some $\snr'\ge0$,
\begin{align}
    Y' \coloneqq \sqrt{\frac{\snr'}{n}} u^* {u^*}^\top + \Xi^{1/2} Z' \Xi^{1/2} , 
    \label{eqn:Y'}
\end{align}
where $ Z' \in \bbR^{n\times n} $ is a symmetric random matrix independent of $ u^*, v^* $ with $ Z'_{i,i} \iid \cN(0,2) $ and $ Z'_{i,j} \iid \cN(0,1) $ for all $ 1\le i<j \le n $. 
By similar derivations as before, the free energy associated with $ (Y, Y') $ is given by 
\begin{align}
    \sfF_n(\snr, \snr') &= \frac{1}{n} \bbE\Bigg[ \log \int_{\bbR^d} \int_{\bbR^n} 
        \exp\bigg( \cH_n(\Xi^{-1/2} u, \Sigma^{-1/2} v) \notag \\
        &\qquad\qquad\qquad\qquad\qquad + \frac{1}{2} \sqrt{\frac{\snr'}{n}} u^\top \Xi^{-1} Y' \Xi^{-1} u - \frac{\snr'}{4n} (u^\top \Xi^{-1} u)^2 \bigg)
    \diff P^{\ot n}(u) \diff Q^{\ot n}(v) \Bigg] , \notag 
\end{align}
where $\cH_n$ is given in \Cref{eqn:Hamil}. 
Denote by $ \bbracket{\cdot}_n $ the Gibbs bracket with respect to the corresponding Hamiltonian. 
Let
\begin{align}
    \sfF(\snr, \snr') &\coloneqq \sup_{q_u, q_v\ge0} \frac{\snr'}{4} q_u^2 + \frac{\alpha \snr}{2} q_u q_v - \psi_{\ol{\Xi}}^* \paren{ \frac{q_u}{2} } - \alpha \psi_{\ol{\Sigma}}^* \paren{ \frac{q_v}{2} } , \label{eqn:sF_lim}
\end{align}
where $ f^* $ denotes the monotone conjugate of a convex non-decreasing function $ f\colon\bbR_{\ge0} \to \bbR $; see \Cref{def:conjugate}. 
Basic properties of monotone conjugate can be found in \cite[\S12]{Rockafellar_book}. 

The following lemma, proved in \Cref{app:pf_lem_sF}, characterizes the high-dimensional limit of $ \sfF_n(\snr, \snr') $. 

\begin{lemma}
\label{lem:sF}
For all $ \snr, \snr' \ge0 $, 
\begin{align}
    \lim_{n\to\infty} \sfF_n(\snr, \snr') &= \sfF(\snr, \snr') , \notag 
\end{align}
\end{lemma}

Let us show how \Cref{cor:MMSE_vec} can be derived from \Cref{lem:sF}. 

\begin{proof}[Proof of \Cref{cor:MMSE_vec}]
Let 
\begin{align}
    \mmse_n^u(\snr, \snr') &\coloneqq \frac{1}{n^2} \expt{ \normf{ \wt{u}^* (\wt{u}^*)^\top - \expt{ \wt{u}^* (\wt{u}^*)^\top \mid Y, Y' } }^2 } . \notag 
\end{align}
Following similar derivations as in \Cref{app:pf_cor_MMSE_mtx}, one can verify the following two identities:
\begin{align}
    \partial_2 \sfF_n(\snr, \snr')
    &= \frac{1}{4n^2} \expt{ \bbracket{ \paren{\wt{u}^\top \wt{u}^*}^2 }_n } , \quad 
    \mmse^u_n(\snr, \snr') = \frac{\tr(\Xi^{-1})^2}{n^2} - \frac{1}{n^2} \expt{ \bbracket{ \paren{\wt{u}^\top \wt{u}^*}^2 }_n } . \label{eqn:MMSE_ol} 
\end{align}
Therefore, the MMSE in \Cref{eqn:MMSE_u} can be written as
\begin{align}
    \mmse_n^u(\snr) &= \lim_{\snr'\downarrow0} \mmse_n^u(\snr, \snr')
    = \frac{\tr(\Xi^{-1})^2}{n^2} - 4 \lim_{\snr'\downarrow0} \partial_2 \sfF_n(\snr, \snr') . \notag 
\end{align}
By \Cref{lem:sF} and \Cref{prop:deriv}, 
\begin{align}
    \limsup_{n\to\infty} \lim_{\snr'\downarrow0} \partial_2 \sfF_n(\snr, \snr')
    &\le \lim_{\snr'\downarrow0} \partial_2 \sfF(\snr, \snr') . \notag 
\end{align}
The envelope theorem from \cite[Corollary 4]{envelope} allows us to compute the RHS:
\begin{align}
    \lim_{\snr'\downarrow0} \partial_2 \sfF(\snr, \snr')
    &= \frac{{q_u^*}^2}{4} , \notag
\end{align}
where $ (q_u^*, q_v^*) $ are the maximizer of 
\begin{align}
    \sup_{q_u, q_v\ge0} \frac{\alpha}{2} \snr q_u q_v - \psi_{\ol{\Xi}}^* \paren{ \frac{q_u}{2} } - \alpha \psi_{\ol{\Sigma}}^* \paren{ \frac{q_v}{2} }
    &= 
    \sup_{(q_u, q_v)\in\cC(\snr, \alpha)} \psi_{\ol{\Xi}}(\alpha\snr q_v) + \alpha\psi_{\ol{\Sigma}}(\snr q_u) - \frac{\alpha\snr}{2} q_uq_v
    , \notag 
\end{align}
where the equality is by \Cref{prop:dual}. 
Putting the above together, we have
\begin{align}
\begin{split}
    \limsup_{n\to\infty} \frac{1}{n^2} \expt{ \bracket{ \paren{\wt{u}^\top \wt{u}^*}^2 }_n }
    &= \limsup_{n\to\infty} \lim_{\snr'\downarrow0} \frac{1}{n^2} \expt{ \bbracket{ \paren{\wt{u}^\top \wt{u}^*}^2 }_n }
    \le {q_u^*}^2 , \\
    \liminf_{n\to\infty} \mmse_n^u(\snr)
    &\ge \expt{\ol{\Xi}^{-1}}^2 - {q_u^*}^2 . 
\end{split}    
\label{eqn:MMSE_lb} 
\end{align}
By a symmetric argument, we also have
\begin{align}
    \limsup_{d\to\infty} \lim_{\snr'\downarrow0} \frac{1}{d^2} \expt{ \bracket{ \paren{\wt{v}^\top \wt{v}^*}^2 }_n }
    \le {q_v^*}^2 . \label{eqn:ol_ub_v}
\end{align}

To find an upper bound on $ \mmse_n^u(\snr) $, note that by \Cref{cor:MMSE_mtx} and \Cref{eqn:MMSE_expand}: 
\begin{align}
    (q_u^* q_v^*)^2 &= \lim_{n\to\infty} \frac{1}{(nd)^2} \expt{\bracket{ \wt{u}^\top \wt{u}^* \wt{v}^\top \wt{v}^* }_n}^2
    \le \liminf_{n\to\infty} \paren{ \frac{1}{n^2} \expt{\bracket{\paren{\wt{u}^\top \wt{u}^*}^2}_n} }
    \paren{ \frac{1}{d^2} \expt{\bracket{\paren{\wt{v}^\top \wt{v}^*}^2}_n} } . \notag 
\end{align}
This combined with \Cref{eqn:MMSE_lb,eqn:ol_ub_v} implies
\begin{align}
    \lim_{n\to\infty} \frac{1}{n^2} \expt{ \bracket{ \paren{\wt{u}^\top \wt{u}^*}^2 }_n }
    &= {q_u^*}^2 , \quad 
    \lim_{d\to\infty} \frac{1}{d^2} \expt{ \bracket{ \paren{\wt{v}^\top \wt{v}^*}^2 }_n }
    = {q_v^*}^2 , \notag 
\end{align}
which concludes the proof in view of the relation
\begin{align}
    \mmse_n^u(\snr) &= \frac{\tr(\Xi^{-1})^2}{n^2} - \frac{1}{n^2} \expt{ \bracket{ \paren{\wt{u}^\top \wt{u}^*}^2 }_n } . \qedhere
\end{align}
\end{proof}

\subsection{Proof of \Cref{lem:sF}}
\label{app:pf_lem_sF}

We assume $ \snr = \snr' = 1 $ by formally absorbing them into $P,Q$:
\begin{align}
    \int_{\bbR} x^2 \diff P(x) &= \sqrt{\snr'} , \quad 
    \int_{\bbR} x^2 \diff Q(x) = \frac{\snr}{\sqrt{\snr'}} , \notag 
\end{align}
so that we can drop the dependence on $ \snr, \snr' $ in notation such as $ \sfF_n, \sfF $. 
We then truncate $P,Q$ at a sufficiently large constant $K>0$ so that they have bounded supports. 

Recall $Y$ from \Cref{eqn:Y_mtx} and define for $ r\ge0 $, $ Y'' \coloneqq \sqrt{r}  u^* + \Xi^{1/2} Z'' $ where $ Z''\sim\cN(0_n, I_n) $ is independent of everything else. 
The free energy $ \wt{\sfF}_n $ associated with $ (Y, Y'') $ is 
\begin{align}
    \wt{\sfF}_n &= \frac{1}{n} \expt{\log \int_{\bbR^d} \int_{\bbR^n} \exp\paren{ \cH_n(\Xi^{-1/2} u, \Sigma^{-1/2} v) + r u^\top \Xi^{-1} u^* + \sqrt{r} u^\top \Xi^{-1/2} Z'' - \frac{r}{2} u^\top \Xi^{-1} u } \diff P^{\ot n}(u) \diff Q^{\ot d}(v) } , \notag 
\end{align}
where $ \cH_n $ is given in \Cref{eqn:Hamil}. 
A straightforward adaptation of the proof of \Cref{thm:free_energy} yields a characterization of the limit of $ \wt{\sfF}_n $. 
Let
\begin{align}
    \wt{\sfF}(r) &\coloneqq \sup_{q_v\ge0} \inf_{q_u\ge0} \psi_{\ol{\Xi}}\paren{ \sqrt{\snr'} (\alpha q_v + r) } + \alpha \psi_{\ol{\Sigma}}\paren{ \frac{\snr}{\sqrt{\snr'}} q_u } - \frac{\alpha}{2} q_u q_v . \label{eqn:Fr}
\end{align}

\begin{lemma}
\label{lem:sfF_hat}
For all $r\ge0$, 
\begin{align}
    \lim_{n\to\infty} \wt{\sfF}_n(r) &= \wt{\sfF}(r) . \notag 
\end{align}
\end{lemma}

\begin{proof}
To obtain the result, we execute the interpolation argument as in the proof of \Cref{thm:free_energy} on the Hamiltonian of the following interpolating models: 
\begin{align}
    Y_t &= \sqrt{\frac{1-t}{n}}  {u^*} {v^*}^\top + \Xi^{1/2} Z \Sigma^{1/2} , \notag \\
    Y_t^u &= \sqrt{\alpha q_1(t)}  u^* + \Xi^{1/2} Z^u , \notag \\
    Y_t^v &= \sqrt{q_2(t)}  v^* + \Sigma^{1/2} Z^v , \notag \\
    Y'' &= \sqrt{r}  u^* + \Xi^{1/2} Z'' . \notag 
\end{align}
All steps in the proof of \Cref{thm:free_energy} carry over. 
\end{proof}

The above lemma allows us to derive the following auxiliary characterization of $ \sfF_n $. 

\begin{lemma}
\begin{align}
    \lim_{n\to\infty} \sfF_n &= \sup_{r\ge0} \wt{\sfF}(r) - \frac{r^2}{4} . \label{eqn:sF_ub} 
\end{align}
\end{lemma}

\begin{proof}
Let $ r\colon[0,1] \to \bbR_{\ge0} $ be a differentiable function. 
For $ t\in[0,1] $, consider $ (Y, Y_t', Y_t'') $ where $ Y $ is given in \Cref{eqn:Y_mtx} and $Y_t' , Y_t''$ are defined as
\begin{align}
    Y_t' &= \sqrt{\frac{1-t}{n}} u^* {u^*}^\top + \Xi^{1/2} Z' \Xi^{1/2} , \quad 
    Y_t'' = \sqrt{r(t)}  u^* + \Xi^{1/2} Z'' , \notag 
\end{align}
with $ Z''\sim\cN(0_n, I_n) $ independent of everything else. 

Similar to \Cref{eqn:fnt,eqn:gibbs_brack}, denote by 
\begin{align}
    \sff_n(t) &\coloneqq \frac{1}{n} \bbE\Bigg[
        \int_{\bbR^d} \int_{\bbR^n}
        \exp\bigg(
            \sqrt{\frac{1}{n}} \wt{u}^\top Z \wt{v} + \frac{1}{n} \wt{u}^\top \wt{u}^* \wt{v}^\top \wt{v}^* - \frac{1}{2n} \normtwo{\wt{u}}^2 \normtwo{\wt{v}}^2  \notag \\
            &\qquad\qquad\qquad\qquad\quad + \frac{1}{2} \sqrt{\frac{1-t}{n}} \wt{u}^\top Z' \wt{u} + \frac{1-t}{2n} (\wt{u}^\top \wt{u}^*)^2 - \frac{1-t}{4n} \normtwo{\wt{u}}^4 \notag \\
            &\qquad\qquad\qquad\qquad\quad + r(t) \wt{u}^\top \wt{u}^* + \sqrt{r(t)} \wt{u}^\top Z'' - \frac{r(t)}{2} \normtwo{\wt{u}}^2
        \bigg)
        \diff\wt{P}(\wt{u}) \diff\wt{Q}(\wt{v})
    \Bigg] \notag 
\end{align}
the free energy associated with $ (Y, Y_t', Y_t'') $ and by $ \bbracket{\cdot}_{n,t} $ the conditional expectation with respect to the corresponding Gibbs measure.
The rest of the proof follows the skeleton of \Cref{thm:free_energy} and we only highlight the differences. 

Parallel to \Cref{lem:f01}, we have that if $ r(0)\ge0 $ and $ \lim_{n\to\infty} r(0) = 0 $, then 
\begin{align}
    \sff_n(0) &= \sfF_n + \cO(r(0)) , \quad 
    \sff_n(1) = \wt{\sfF}_n(r(1)) . \notag 
\end{align}
The analogue of \Cref{lem:free_energy_var} gives 
\begin{align}
    \sff_n'(t) &= - \frac{1}{4} \expt{\bbracket{\paren{ \frac{\wt{u}^\top \wt{u}^*}{n} - r'(t) }^2}_{n,t}} + \frac{r'(t)^2}{4} . \label{eqn:fnt'} 
\end{align}

We now show a pair of matching lower and upper bounds on $ \sfF_n $. 
First comes the lower bound. 
Using \Cref{eqn:fnt'} with $ r(t) = r t $ for a constant $r\ge0$, we have 
\begin{align}
    \sfF_n &= \sff_n(0)
    = \sff_n(1) - \int_0^1 \sff_n'(t) \diff t
    \ge \wt{\sfF}_n(r) - \frac{r^2}{4} . \notag 
\end{align}
By \Cref{lem:sfF_hat}, this implies the lower bound
\begin{align}
    \liminf_{n\to\infty} \sfF_n &\ge \sup_{r\ge0} \wt{\sfF}(r) - \frac{r^2}{4} . \notag 
\end{align}

Next we show a matching upper bound. 
Let $ \eps\in[1,2] $ and slightly abusing notation, denote by $ r(t; \eps) $ the solution to 
\begin{align}
    r'(t) &= \frac{1}{n} \expt{\bbracket{\wt{u}^\top \wt{u}^*}_{n,t}} , \quad 
    r(0) = s_n \eps , \notag 
\end{align}
where $ s_n = n^{-1/32} $. 
The analogue of \Cref{lem:ol_conc} gives 
\begin{align}
    \int_1^2 \frac{1}{n^2} \expt{\bbracket{\paren{ \wt{u}^\top \wt{u}^* - \expt{\bracket{\wt{u}^\top \wt{u}^*}_{n,t}} }^2}_{n,t}} \diff\eps &\le \frac{C}{n^{1/8}} , \label{eqn:ol_conc} 
\end{align}
for a constant $C>0$ independent of $n$. 
Using \Cref{eqn:fnt'} with $ r(t; \eps) $, we have
\begin{align}
    \sfF_n &= \int_1^2 \sff_n(0) \diff\eps + o(1)
    = \int_1^2 \paren{ \wt{\sfF}_n(r(1, \eps)) - \int_0^1 \frac{r'(t,\eps)^2}{4} \diff t } \diff\eps + o(1) \notag \\
    &\le \int_1^2 \int_0^1 \wt{\sfF}_n(r'(t,\eps)) - \frac{r'(t,\eps)^2}{4} \diff t \diff \eps + o(1)
    \le \sup_{r\ge0} \wt{\sfF}_n(r) - \frac{r^2}{4} + o(1) , \notag 
\end{align}
where the second equality is by \Cref{eqn:ol_conc} and the penultimate inequality holds since $\wt{\sfF}_n(\cdot)$ is convex and non-decreasing. 
Passing to the limit, we obtain the upper bound
\begin{align}
    \limsup_{n\to\infty} \sfF_n &\le \sup_{r\ge0} \wt{\sfF}(r) - \frac{r^2}{4} , \notag 
\end{align}
as desired. 
\end{proof}

To establish \Cref{lem:sF}, it remains to verify that the RHSs of \Cref{eqn:sF_lim,eqn:sF_ub} are equal. 
We need the following lemma. 

\begin{lemma}
\label{lem:supinf}
Let $ f,g\colon\bbR_{\ge0}\to\bbR $ be non-decreasing, lower semi-continuous, convex functions with finite $ f(0), g(0) $ and monotone conjugates $ f^*, g^* $ (see \Cref{def:conjugate}). 
Then
\begin{align}
    \sup_{r\ge0} \sup_{q_1\ge0} \inf_{q_2\ge0} f(q_1 + r) + g(q_2) - q_1 q_2 - \frac{r^2}{2}
    &= \sup_{q_1,q_2\ge0} \frac{q_2^2}{2} + q_1q_2 - f^*(q_2) - g^*(q_1) . \notag 
\end{align}
\end{lemma}

\begin{proof}
Writing $ f_r(x) = f(x+r) $ and using \Cref{prop:dual}, we have
\begin{align}
    \sup_{q_1\ge0} \inf_{q_2\ge0} f(q_1 + r) + g(q_2) - q_1 q_2
    &= \sup_{q_1, q_2\ge0} q_1 q_2 - f_r^*(q_2) - g^*(q_1) \notag \\
    &= \sup_{q_1\ge0} f_r(q_1) - g^*(q_1) 
    = \sup_{q_1, q_2\ge0} q_2 (q_1+r) - f^*(q_2) - g^*(q_1) , \notag 
\end{align}
where we have used the fact that $ f^{**} = f $. 
Therefore, 
\begin{align}
    \sup_{r\ge0} \sup_{q_1\ge0} \inf_{q_2\ge0} f(q_1 + r) + g(q_2) - q_1 q_2 - \frac{r^2}{2}
    &= \sup_{q_1, q_2\ge0} \brace{ \sup_{r\ge0} \brace{ q_2r - \frac{r^2}{2} } + q_1q_2 - f^*(q_2) - g^*(q_1) } \notag \\
    &= \sup_{q_1,q_2\ge0} \frac{q_2^2}{2} + q_1q_2 - f^*(q_2) - g^*(q_1) , \notag 
\end{align}
as claimed.
\end{proof}

Applying \Cref{lem:supinf} immediately finishes the proof of \Cref{lem:sF}. 

\begin{proof}[Proof of \Cref{lem:sF}]
By \Cref{lem:supinf} and the definition \Cref{eqn:Fr}, 
\begin{align}
    \sup_{r\ge0} \wt{\sfF}(r) - \frac{r^2}{4}
    &= \sup_{q_u, q_v\ge0} \frac{\snr'}{4} q_u^2 + \frac{\alpha}{2} \snr q_u q_v - \psi_{\ol{\Xi}}^* \paren{ \frac{q_u}{2} } - \alpha \psi_{\ol{\Sigma}}^* \paren{ \frac{q_v}{2} } , \notag
\end{align}
where we have used the fact that for $ a,b\ge0 $, the monotone conjugate of $ g(x) \coloneqq b f(ax) $ is $ g^*(x) = b f^*(x/(ab)) $. 
\end{proof}

\subsection{Proof of \Cref{cor:weak_rec_thr}}
\label{app:pf_cor_weak_rec_thr}

Denote $ M \coloneqq \expt{ u^* {v^*}^\top \mid Y } $. 
By the Nishimori identity (\Cref{prop:nishimori}), 
\begin{align}
    \lim_{n\to\infty} \frac{1}{nd} \expt{ \normf{ \wt{u}^* (\wt{v}^*)^\top - \expt{ \wt{u}^* (\wt{v}^*)^\top \mid Y } }^2 }
    &= \expt{\ol{\Xi}^{-1}} \expt{\ol{\Sigma}^{-1}} - \lim_{n\to\infty} \frac{1}{nd} \expt{ \normf{ \Xi^{-1/2} M \Sigma^{-1/2} }^2 } , \label{eqn:MMSE_tilde_nish} \\
    \lim_{n\to\infty} \frac{1}{nd} \expt{ \normf{ u^* {v^*}^\top - \expt{ u^* {v^*}^\top \mid Y } }^2 }
    &= 1 - \lim_{n\to\infty} \frac{1}{nd} \expt{ \normf{ M }^2 } . \label{eqn:eqn:MMSE_nish}
\end{align}
\Cref{thm:IT} and \Cref{eqn:MMSE_tilde_nish} imply that
\begin{align}
    \lim_{n\to\infty} \frac{1}{nd} \expt{ \normf{ \Xi^{-1/2} M \Sigma^{-1/2} }^2 } &= q_u^* q_v^* , \label{eqn:quqv_pos} 
\end{align}
which, by \Cref{prop:bayes_fp_sol}, is positive if and only if \Cref{eqn:thr_bayes_change} holds. 

We first show \Cref{eqn:MMSE_uv_thr} assuming \Cref{eqn:thr_bayes_change}. 
Using assumption \Cref{eqn:supp}, super-multiplicativity of $ \sigma_{\min}(\cdot) $, and the fact that $ \normf{B C} \ge \normf{B} \sigma_{\min}(C) $, we have
\begin{align}
    \lim_{n\to\infty} \frac{1}{nd} \expt{ \normf{M}^2 }
    &\ge \lim_{n\to\infty} \frac{1}{nd} \sigma_n(\Xi^{1/2})^2 \expt{\normf{ \Xi^{-1/2} M \Sigma^{-1/2} }^2} \sigma_d(\Sigma^{1/2})^2 \notag \\
    &= (\inf\supp(\ol{\Xi})) (\inf\supp(\ol{\Sigma})) \lim_{n\to\infty} \frac{1}{nd} \expt{\normf{ \Xi^{-1/2} M \Sigma^{-1/2} }^2} , \notag 
\end{align}
which is positive by \Cref{eqn:quqv_pos}. 
This combined with \Cref{eqn:eqn:MMSE_nish} implies \Cref{eqn:MMSE_uv_thr}. 

We then show \Cref{eqn:MMSE_uv_thr} with $<$ replaced with $=$, assuming that \Cref{eqn:thr_bayes_change} is reversed. 
Using assumption \Cref{eqn:supp}, sub-multiplicativity of spectral norm, and the fact that $ \normf{BC} \le \normf{B} \normtwo{C} $, we have
\begin{align}
    \lim_{n\to\infty} \frac{1}{nd} \expt{\normf{M}^2}
    &\le \lim_{n\to\infty} \frac{1}{nd} \normtwo{\Xi^{1/2}}^2 \expt{\normf{\Xi^{-1/2} M \Sigma^{-1/2}}^2} \normtwo{\Sigma^{1/2}}^2 \notag \\
    &= (\sup\supp(\ol{\Xi})) (\sup\supp(\ol{\Sigma})) 
    \lim_{n\to\infty} \frac{1}{nd} \expt{\normf{\Xi^{-1/2} M \Sigma^{-1/2}}^2} , \notag 
\end{align}
which is $0$ by \Cref{eqn:quqv_pos}. 
This combined with \Cref{eqn:eqn:MMSE_nish} finishes the proof of the corollary for estimating $ u^* {v^*}^\top $. 
The proofs for estimating $ u^* {u^*}^\top, v^* {v^*}^\top $ are similar and omitted. 



\section{Analysis of the spectral estimator}
\label{app:analysis_spec}

\subsection{Bayes-AMP}
\label{sec:BAMP}

We propose an AMP algorithm that operates on $ \Xi^{-1} A \Sigma^{-1} $ and maintains a pair of iterates $ u^t\in\bbR^n, v^{t+1}\in\bbR^d $ for every $ t\ge0 $. 
Specifically, given for every $t\ge0$ a pair of denoising functions\footnote{Strictly speaking, for every $t\ge0$, we are given two sequences of functions $ g_t, f_{t+1} $ indexed by $n,d$ respectively. See \Cref{def:PL_fn} for a formal treatment of function sequences. } $ g_t\colon\bbR^n\to\bbR^n, f_{t+1}\colon\bbR^d\to\bbR^d $, the iterates are initialized at $ \wt{u}^{-1} = 0_n $ and some $ \wt{v}^0\in\bbR^d $ of user's choice, and are updated for every $ t\ge0 $ according to the following rules:
\begin{align}
\begin{split}
u^t &= \Xi^{-1} A \Sigma^{-1} \wt{v}^t - b_t \Xi^{-1} \wt{u}^{t-1} , \quad 
\wt{u}^t = g_t(u^t) , \quad 
c_t = \frac{1}{n} \tr((\nabla g_t(u^t)) \Xi^{-1}) , \\
v^{t+1} &= \Sigma^{-1} A^\top \Xi^{-1} \wt{u}^t - c_t \Sigma^{-1} \wt{v}^t , \quad
\wt{v}^{t+1} = f_{t+1}(v^{t+1}) , \quad
b_{t+1} = \frac{1}{n} \tr((\nabla f_{t+1}(v^{t+1})) \Sigma^{-1}) , 
\end{split}
\label{eqn:BAMP} 
\end{align}
where $ \nabla g_t(u^t)\in\bbR^{n\times n}, \nabla f_{t+1}(v^{t+1})\in\bbR^{d\times d} $ denote the Jacobians of $ g_t, f_{t+1} $ at $ u^t, v^{t+1} $, respectively. 
For any fixed $ t\ge0 $, the $n,d\to\infty$ limit of the iterates $ u^t, v^{t+1} $ can be described by a deterministic recursion known as the state evolution. 
To define the latter, we need a sequence of preliminary definitions. 

First, define random vectors $ (U^*, W_{U,0}, W_{U,1}, \cdots, W_{U,t})\in(\bbR^n)^{t+2} $ and $ (V^*, W_{V,1}, W_{V,2}, \cdots, W_{V,t+1})\in(\bbR^d)^{t+2} $ with joint distributions specified below:
\begin{align}
\matrix{ U^* \\ \sigma_0 W_{U,0} \\ \sigma_1 W_{U,1} \\ \vdots \\ \sigma_t W_{U,t} }
&\sim P^{\ot n} \ot \cN(0_{n(t+1)}, \Phi_t \ot I_n) , \qquad 
\matrix{ V^* \\ \tau_1 W_{V,1} \\ \tau_2 W_{V,2} \\ \vdots \\ \tau_{t+1} W_{V,t+1} }
\sim Q^{\ot d} \ot \cN(0_{d(t+1)}, \Psi_{t+1} \ot I_d) , 
\label{eqn:all_joint}
\end{align}
where we recall that for $ A\in\bbR^{m\times n}, B\in\bbR^{p\times q} $, their Kronecker product 
is 
\begin{align}
    A \ot B
    &= \matrix{
    A_{1,1} B & \cdots & A_{1,n} B \\
    \vdots & \ddots & \vdots \\
    A_{m,1} B & \cdots & A_{m,n} B
    } \in\bbR^{mp \times nq} . \notag 
\end{align}
The covariance matrices $ \Phi_t = (\Phi_{r,s})_{0\le r,s\le t} , \Psi_{t+1} = (\Psi_{r+1,s+1})_{0\le r,s\le t} \in\bbR^{(t+1)\times(t+1)} $ are given by the $ (t+1)\times(t+1) $ principal minors of two infinite-dimensional symmetric matrices $ \Phi \coloneqq (\Phi_{r,s})_{r,s\ge0}, \Psi \coloneqq (\Psi_{r+1,s+1})_{r,s\ge0} $ whose elements are in turn given recursively below:
\begin{align}
\begin{split}
\Phi_{0,0} &= \plim_{n\to\infty} \frac{1}{n} (\wt{v}^0)^\top \Sigma^{-1} \wt{v}^0 , \\
\Phi_{0,s} &= \lim_{n\to\infty} \frac{1}{n} \expt{ f_0(V^*)^\top \Sigma^{-1} f_s(V_s) } , \quad s\ge1 , \\
\Phi_{r,s} &= \lim_{n\to\infty} \frac{1}{n} \expt{ f_r(V_r)^\top \Sigma^{-1} f_s(V_s) } , \quad r,s\ge1 , \\
\Psi_{r+1,s+1} &= \lim_{n\to\infty} \frac{1}{n} \expt{ g_r(U_r)^\top \Xi^{-1} g_s(U_s) } , \quad r,s\ge0 . 
\end{split}
\label{eqn:Phi_Psi}
\end{align}
Furthermore, for $ t\ge0 $, $ \sigma_t , \tau_{t+1} > 0 $ are defined as
\begin{align}
\begin{split}
\sigma_0^2 &\coloneqq \Phi_{0, 0} = \plim\limits_{n\to\infty} \frac{1}{n} (\wt{v}^0)^\top \Sigma^{-1} \wt{v}^0 , \\
\sigma_t^2 &\coloneqq \Phi_{t, t} = \lim\limits_{n\to\infty} \frac{1}{n} \expt{ f_t(V_t)^\top \Sigma^{-1} f_t(V_t) } , \quad t\ge1 , \\
\tau_{t+1}^2 &\coloneqq \Psi_{t+1,t+1} = \lim_{n\to\infty} \frac{1}{n} \expt{ g_t(U_t)^\top \Xi^{-1} g_t(U_t) } , \quad t\ge0 . 
\end{split}
\label{eqn:sigma_tau}
\end{align}
With the above definitions, note that each $ W_{U,t}, W_{V,t+1} $ is marginally distributed as $ \cN(0_n, I_n), \cN(0_d, I_d) $, respectively.

Next, define two sequences of deterministic scalars $ (\mu_t, \nu_{t+1})_{t\ge0} $:
\begin{align}
\begin{split}
\mu_0 &= \lim\limits_{n\to\infty} \frac{\lambda}{n} \expt{ \inprod{\Sigma^{-1} V^*}{f_0(V^*)} } , \\
\mu_t &= \lim\limits_{n\to\infty} \frac{\lambda}{n} \expt{ \inprod{\Sigma^{-1} V^*}{f_t(V_t)} } , \quad t\ge1, \\
\nu_{t+1} &= \lim_{n\to\infty} \frac{\lambda}{n} \expt{ \inprod{\Xi^{-1} U^*}{g_t(U_t)} } , \quad t\ge0 ,
\end{split}
\label{eqn:mu_nu}
\end{align}
where $ f_0 $ is determined by the initializer $ \wt{v}^0 $; see \Cref{asmp:init_bayes} below. 

With these, for $ t\ge0 $, define random vectors
\begin{align}
U_t &= \mu_t \Xi^{-1} U^* + \sigma_t \Xi^{-1/2} W_{U,t} , \qquad
V_{t+1} = \nu_{t+1} \Sigma^{-1} V^* + \tau_{t+1} \Sigma^{-1/2} W_{V,t+1} . 
\label{eqn:UV_bayes}
\end{align}

Finally, we need the notion of (uniformly) pseudo-Lipschitz functions. 

\begin{definition}[Pseudo-Lipschitz functions]
\label{def:PL_fn}
A function $ \phi\colon\bbR^{k\times m}\to\bbR^{\ell\times m} $ is called pseudo-Lipschitz of order $j\ge1$ if there exists $L>0$ such that 
\begin{align}
    \frac{1}{\sqrt{\ell}} \normf{\phi(x) - \phi(y)} &\le \frac{L}{\sqrt{k}} \normf{x - y} \brack{ 1 + \paren{\frac{1}{\sqrt{k}} \normf{x}}^{j-1} + \paren{\frac{1}{\sqrt{k}} \normf{y}}^{j-1} } , \label{eqn:def_PL} 
\end{align}
for every $ x, y\in\bbR^{k\times m} $. 
\end{definition}

We will consider sequences of functions $\phi_i\colon\bbR^{k_i\times m} \to \bbR^{\ell_i\times m}$ indexed by $ i\to\infty $ though the index $i$ is often suppressed. 
A sequence of functions $ (\phi_i\colon\bbR^{k_i\times m} \to \bbR^{\ell_i\times m})_{i\ge1} $ (with increasing dimensions $ (k_i)_{i\ge1}, (\ell_i)_{i\ge1} $) is called uniformly pseudo-Lipschitz of order $j$ if there exists a constant $L$ such that for every $i\ge1$, \Cref{eqn:def_PL} holds. 

The assumptions below are imposed on the initializer $ \wt{v}^0 $ and denoising function $ (g_t, f_{t+1})_{t\ge0} $. 

\begin{enumerate}[label=(A\arabic*)]
\setcounter{enumi}{\value{asmpctr}}

    \item \label[asmp]{asmp:init_bayes} $ \wt{v}^0 $ is independent of $ \wt{W} $ but may depend on $ v^* $.\footnote{Practically one can think of the dependence of $ \wt{v}^0 $ on $ v^* $ being given by some side information. However, here AMP is used solely as a proof technique, and we can consider initializers with impractical access to $ v^* $. } Assume that 
    \begin{align}
    \plim_{d\to\infty} \frac{1}{d} \normtwo{\wt{v}^0}^2 , \qquad 
    \plim_{d\to\infty} \frac{1}{n} (\wt{v}^0)^\top \Sigma^{-1} \wt{v}^0 \notag 
    \end{align}
    exist and are finite. 
    There exists a uniformly pseudo-Lipschitz function $ f_0\colon\bbR^d\to\bbR^d $ of order $1$ such that 
    \begin{align}
    \lim_{d\to\infty} \frac{1}{d} \expt{ \inprod{f_0(V^*)}{f_0(V^*)} } &\le \plim_{d\to\infty} \frac{1}{d} \normtwo{\wt{v}^0}^2 \notag 
    \end{align}
    and for every uniformly pseudo-Lipschitz function $ \phi\colon\bbR^d\to\bbR^d $ of finite order, the following two limits exist, are finite and equal:
    \begin{align}
    \plim_{d\to\infty} \frac{1}{d} \inprod{\wt{v}^0}{\phi(V^*)}
    &= \lim_{d\to\infty} \frac{1}{d} \expt{ \inprod{f_0(V^*)}{\phi(V^*)} } . \notag 
    \end{align}
    Let $ \wt{\nu}\in\bbR, \wt{\tau}\in\bbR_{\ge0} $. 
    For any $s\ge1$, 
    \begin{align}
    \lim_{d\to\infty} \frac{1}{d} \expt{ f_0(V^*)^\top \Sigma^{-1} f_s(\wt{\nu} \Sigma^{-1} V^* + \wt{\tau} \Sigma^{-1/2} W_V) } \notag 
    \end{align}
    exists and is finite, where $ W_V\sim\cN(0_d, I_d) $ is independent of $ V^* $. 

    \item \label[asmp]{asmp:denois_fn_bayes}
    Let $ \wt{\nu}\in\bbR $, and $ T\in\bbR^{2\times2} $ be positive definite. 
    For any $ r,s\ge1 $, 
    \begin{align}
    &\lim_{n\to\infty} \frac{\lambda}{n} \expt{ \inprod{\Sigma^{-1} V^*}{f_r(\wt{\nu} \Sigma^{-1} V^* + \Sigma^{-1/2} N)} } , \notag \\
    &\lim_{d\to\infty} \frac{1}{d} \expt{ f_r(\wt{\nu} \Sigma^{-1} V^* + \Sigma^{-1/2} N)^\top \Sigma^{-1} f_s(\wt{\nu} \Sigma^{-1} V^* + \Sigma^{-1/2} N') } \notag 
    \end{align}
    exist and are finite, where $ (N,N') \sim \cN(0_{2d}, T \ot I_d) $ is independent of $ V^* $. 

    Let $ \wt{\mu}\in\bbR $, and $ S\in\bbR^{2\times2} $ be positive definite. 
    For any $ r,s\ge0 $, 
    \begin{align}
    &\lim_{n\to\infty} \frac{\lambda}{n} \expt{ \inprod{\Xi^{-1} U^*}{g_r(\wt{\mu} \Xi^{-1} U^* + \Xi^{-1/2} M)} } , \notag \\
    &\lim_{d\to\infty} \frac{1}{d} \expt{ g_r(\wt{\mu} \Xi^{-1} U^* + \Xi^{-1/2} M)^\top \Xi^{-1} g_s(\wt{\mu} \Xi^{-1} U^* + \Xi^{-1/2} M') } \notag 
    \end{align}
    exist and are finite, where $ (M,M') \sim \cN(0_{2n}, S \ot I_n) $ is independent of $ U^* $. 

\setcounter{asmpctr}{\value{enumi}}
\end{enumerate}

We now give the state evolution result for the AMP in \Cref{eqn:BAMP}, which is proved in \Cref{app:pf_BAMP_SE}. 

\begin{proposition}[State evolution for AMP in \Cref{eqn:BAMP}]
\label{prop:SE_bayes}
For every $ t\ge0 $, let $ (g_t\colon\bbR^n\to\bbR^n)_{n\ge1} $ and $ (f_{t+1}\colon\bbR^d\to\bbR^d)_{d\ge1} $ be uniformly pseudo-Lipschitz of finite order subject to \Cref{asmp:denois_fn_bayes}. 
Consider the AMP iteration in \Cref{eqn:BAMP} defined by $ (g_t, f_{t+1})_{t\ge0} $ and initialized at $ \wt{u}^{-1} = 0_n $ and some $ \wt{v}^0\in\bbR^d $ subject to \Cref{asmp:init_bayes}. 
For any fixed $ t\ge0 $, let $ (\phi\colon\bbR^{(t+2)n} \to \bbR)_{n\ge1} $ and $ (\psi\colon\bbR^{(t+2)d} \to \bbR)_{d\ge1} $ be uniformly pseudo-Lipschitz of finite order. 
Then, 
\begin{subequations}
\begin{align}
\begin{split}
\plim_{n\to\infty} \phi(u^*, u^0, u^1, \cdots, u^t) - \expt{\phi(U^*, U_0, U_1, \cdots, U_t)} &= 0 , 
\end{split} \label{eqn:SE_result_u_bayes} \\
\begin{split}
\plim_{d\to\infty} \psi(v^*, v^1, v^2, \cdots, v^{t+1}) - \expt{\psi(V^*, V_1, V_2, \cdots, V_{t+1})} &= 0 ,     
\end{split} \label{eqn:SE_result_v_bayes}
\end{align}
\label{eqn:SE_result_uv}
\end{subequations}
where $ (U_s, V_{s+1})_{0\le s\le t} $ are defined in \Cref{eqn:UV_bayes}. 
\end{proposition}


Given $ U_t, V_{t+1} $, the Bayes-optimal (in terms of mean square error) choice of $ (g_t, f_{t+1})_{t\ge0} $ is given by the conditional expectations. 
Specifically, for any $ t\ge0 $ and $ u\in\bbR^n, v\in\bbR^d $,
\begin{align}
g_t^*(u) &\coloneqq \expt{ U^* \mid U_t = u } , \qquad 
f_{t+1}^*(v) \coloneqq \expt{ V^* \mid V_{t+1} = v } . \label{eqn:gf*}
\end{align}
We call \Cref{eqn:BAMP} with $ g_t = g_t^*, f_{t+1} = f_{t+1}^* $ the Bayes-AMP. 

If $ P = Q = \cN(0,1) $, by \Cref{eqn:all_joint,eqn:UV_bayes}, $ (U^*, U_t) $ and $ (V^*, V_{t+1}) $ are jointly Gaussian with mean zero and covariance:
\begin{align}
\matrix{
    I_n & \mu_t \Xi^{-1} \\
    \mu_t \Xi^{-1} & \mu_t^2 \Xi^{-2} + \sigma_t^2 \Xi^{-1}
} \in\bbR^{2n\times2n} , \qquad 
\matrix{
    I_d & \nu_{t+1} \Sigma^{-1} \\
    \nu_{t+1} \Sigma^{-1} & \nu_{t+1}^2 \Sigma^{-2} + \tau_{t+1}^2 \Sigma^{-1}
} \in\bbR^{2d\times2d} , \notag 
\end{align}
respectively. 
Therefore using \Cref{lem:cond-distr-gauss}, $ g_t^*, f_{t+1}^* $ admit the following explicit formulas: 
\begin{align}
\begin{split}
g_t^*(u) &= \mu_t \Xi^{-1} (\mu_t^2 \Xi^{-2} + \sigma_t^2 \Xi^{-1})^{-1} u 
= \mu_t (\mu_t^2 \Xi^{-1} + \sigma_t^2 I_n)^{-1} u , \\
f_{t+1}^*(v) &= \tau_{t+1} \Sigma^{-1} (\tau_{t+1}^2 \Sigma^{-2} + \tau_{t+1}^2 \Sigma^{-1})^{-1} v 
= \nu_{t+1} (\nu_{t+1}^2 \Sigma^{-1} + \tau_{t+1}^2 I_d)^{-1} v . 
\end{split}
\label{eqn:g*_f*}
\end{align}

Under the above choice, the state evolution recursion for $ \mu_t, \sigma_t, \nu_{t+1}, \tau_{t+1} $ in \Cref{eqn:mu_nu,eqn:sigma_tau} becomes: for all $ t\ge1 $, 
\begin{subequations}
\begin{align}
\begin{split}
\mu_t &= \lim_{n\to\infty} \frac{\lambda}{n} \expt{ \inprod{\Sigma^{-1} V^*}{ \nu_t (\nu_t^2 \Sigma^{-1} + \tau_t^2 I_d)^{-1} V_t } } 
= \frac{\lambda}{\delta} \expt{ \frac{\nu_t^2 \ol{\Sigma}^{-2}}{\nu_t^2 \ol{\Sigma}^{-1} + \tau_t^2} } , 
\end{split}
\label{eqn:SE_bayes_mu} \\
\begin{split}
\nu_{t+1} &= \lim_{n\to\infty} \frac{\lambda}{n} \expt{ \inprod{\Xi^{-1} U^*}{\mu_t (\mu_t^2 \Xi^{-1} + \sigma_t^2 I_n)^{-1} U_t} } 
= \lambda \expt{ \frac{\mu_t^2 \ol{\Xi}^{-2}}{\mu_t^2 \ol{\Xi}^{-1} + \sigma_t^2} } , 
\end{split}
\label{eqn:SE_bayes_nu} \\
\begin{split}
\sigma_t^2 &= \lim_{n\to\infty} \frac{1}{n} \expt{ V_t^\top (\nu_t^2 \Sigma^{-1} + \tau_t^2 I_d)^{-1} \nu_t  \Sigma^{-1} \nu_t (\nu_t^2 \Sigma^{-1} + \tau_t^2 I_d)^{-1} V_t } \\
&= \frac{1}{\delta} \expt{ \frac{\nu_t^4 \ol{\Sigma}^{-3}}{(\nu_t^2 \ol{\Sigma}^{-1} + \tau_t^2)^2} } 
+ \frac{1}{\delta} \expt{ \frac{\nu_t^2 \tau_t^2 \ol{\Sigma}^{-2}}{(\nu_t^2 \ol{\Sigma}^{-1} + \tau_t^2)^2} } \\
&= \frac{1}{\delta} \expt{ \frac{\nu_t^2 \ol{\Sigma}^{-2}}{\nu_t^2 \ol{\Sigma}^{-1} + \tau_t^2} } , 
\end{split}
\label{eqn:SE_bayes_sigma} \\
\begin{split}
\tau_{t+1}^2 &= \lim_{n\to\infty} \frac{1}{n} \expt{ U_t^\top (\mu_t^2 \Xi^{-1} + \sigma_t^2 I_n)^{-1} \mu_t  \Xi^{-1} \mu_t (\mu_t^2 \Xi^{-1} + \sigma_t^2 I_n)^{-1} U_t } \\
&= \expt{ \frac{\mu_t^4 \ol{\Xi}^{-3}}{(\mu_t^2 \ol{\Xi}^{-1} + \sigma_t^2)^2} } 
+ \expt{ \frac{\mu_t^2 \sigma_t^2 \ol{\Xi}^{-2}}{(\mu_t^2 \ol{\Xi}^{-1} + \sigma_t^2)^2} } \\
&= \expt{ \frac{\mu_t^2 \ol{\Xi}^{-2}}{\mu_t^2 \ol{\Xi}^{-1} + \sigma_t^2} } , 
\end{split}
\label{eqn:SE_bayes_tau} 
\end{align}
\label{eqn:SE_bayes}
\end{subequations}
where we have used the definitions \Cref{eqn:UV_bayes} of $ U_t , V_{t+1} $, the joint distribution \Cref{eqn:all_joint} of $ (U^*, W_{U,t}), (V^*, V_{t+1}) $, the convergence of the empirical spectral distributions of $ \Sigma, \Xi $, and \Cref{prop:quadratic_form}. 

Inspecting the expressions, we realize that
\begin{align}
\mu_t &= \lambda \sigma_t^2 , \qquad 
\nu_{t+1} = \lambda \tau_t^2 . \label{eqn:SE_relation}
\end{align}
This allows us to only track the recursion of $ \mu_t, \nu_{t+1} $:
\begin{align}
\mu_t &= \frac{\lambda}{\delta} \expt{ \frac{\lambda\nu_t \ol{\Sigma}^{-2}}{\lambda\nu_t \ol{\Sigma}^{-1} + 1} } , \qquad 
\nu_{t+1} = \lambda \expt{ \frac{\lambda \mu_t \ol{\Xi}^{-2}}{\lambda \mu_t \ol{\Xi}^{-1} + 1} } . \notag 
\end{align}
Thus, the fixed point $ (\mu^*, \nu^*) $ of the above recursion must satisfy:
\begin{align}
\begin{split}
\mu^* &= \frac{\lambda}{\delta} \expt{ \frac{\lambda\nu^* \ol{\Sigma}^{-2}}{\lambda\nu^* \ol{\Sigma}^{-1} + 1} } , \qquad 
\nu^* = \lambda \expt{ \frac{\lambda \mu^* \ol{\Xi}^{-2}}{\lambda \mu^* \ol{\Xi}^{-1} + 1} } . 
\end{split}
\label{eqn:BAMP_SE_FP}
\end{align}
Note that upon a change of variable
\begin{align}
q_u &\coloneqq \frac{\nu^*}{\lambda} , \qquad 
q_v \coloneqq \frac{\delta \mu^*}{\lambda} , \label{eqn:quv_to_mu_nu} 
\end{align}
the fixed point equation \Cref{eqn:BAMP_SE_FP} coincides with that in the characterization of the free energy; see \Cref{eqn:fp_it}.

\subsection{Spectral estimator from Bayes-AMP}
\label{sec:AMP_spec}

Under Gaussian priors, the Bayes-AMP algorithm specified by \Cref{eqn:BAMP,eqn:g*_f*} naturally suggests a spectral estimator with respect to a matrix which is a non-trivial transformation of $A$. 
In what follows, we provide a heuristic derivation of this spectral estimator. 
Its performance guarantee (\Cref{thm:spec}) will be proved in \Cref{app:pf_spec}. 

Suppose, informally, that $ \mu_t$, $\sigma_t$, $\nu_{t+1}$, $\tau_{t+1}$, $u^t$, $v^{t+1}$, $c_t$, $b_{t+1} $ converge (under the sequential limits $ n\to\infty, t\to\infty $) to $ \mu^*, \sigma^*, \nu^*, \tau^*, u, v, c^*, b^* $, respectively, in the sense that, e.g.,
\begin{align}
\lim_{t\to\infty} \plim_{n\to\infty} \frac{1}{\sqrt{n}} \normtwo{u^t - u} &= 0 . \notag 
\end{align}
Recall that $ (\mu^*, \nu^*) $ solves the fixed point equation \Cref{eqn:BAMP_SE_FP}, and from \Cref{eqn:SE_relation}, the following relation holds:
\begin{align}
\mu^* = \lambda {\sigma^*}^2, \qquad 
\nu^* = \lambda {\tau^*}^2 . 
\label{eqn:sigma_tau_fp}
\end{align}
So denoting
\begin{align}
G &\coloneqq \lambda (\lambda \mu^* \Xi^{-1} + I_n)^{-1} \in\bbR^{n\times n} , \qquad 
F \coloneqq \lambda (\lambda \nu^* \Sigma^{-1} + I_d)^{-1} \in\bbR^{d\times d} , \label{eqn:FG_bayes}
\end{align}
by the design of $ g_t^*, f_{t+1}^* $ in \Cref{eqn:g*_f*}, 
we have that $ \wt{u}^t, \wt{v}^{t+1} $ converge to 
\begin{align}
\wt{u} &= \mu^* \paren{ {\mu^*}^2 \Xi^{-1} + {\sigma^*}^2 I_n }^{-1} u = G u , \qquad 
\wt{v} = \nu^* \paren{ {\nu^*}^2 \Sigma^{-1} + {\tau^*}^2 I_d }^{-1} v = F v , \notag 
\end{align}
respectively, 
and the limiting Onsager coefficients $ b^*,c^* $ are given by: 
\begin{align}
\begin{split}
b^* &= \lim_{n\to\infty} \frac{1}{n} \tr(F \Sigma^{-1}) 
= \frac{1}{\delta} \expt{ \frac{\lambda}{\lambda \nu^* + \ol{\Sigma}} } , \quad
c^* = \lim_{n\to\infty} \frac{1}{n} \tr(G \Xi^{-1})
= \expt{ \frac{\lambda}{\lambda\mu^* + \ol{\Xi}} } . 
\end{split}
\label{eqn:bc}
\end{align} 

At the fixed point of \Cref{eqn:BAMP}, we have 
\begin{align}
\begin{split}
u &= \Xi^{-1} A \Sigma^{-1} \wt{v} - b^* \Xi^{-1} \wt{u}
= \Xi^{-1} A \Sigma^{-1} F v - b^* \Xi^{-1} G u , \\
v &= \Sigma^{-1} A^\top \Xi^{-1} \wt{u} - c^* \Sigma^{-1} \wt{v}
= \Sigma^{-1} A^\top \Xi^{-1} G u - c^* \Sigma^{-1} F v . 
\end{split}
\label{eqn:rearrange}
\end{align}
Upon rearrangement, \Cref{eqn:rearrange} is equivalent to 
\begin{align}
\wh{G} u &= \Xi^{-1} A \Sigma^{-1} F v , \qquad 
\wh{F} v = \Sigma^{-1} A^\top \Xi^{-1} G u , 
\label{eqn:rewrite}
\end{align}
where 
\begin{align}
\wh{G} &\coloneqq I_n + b^* \Xi^{-1} G \in\bbR^{n\times n} , \qquad 
\wh{F} \coloneqq I_d + c^* \Sigma^{-1} F \in\bbR^{d\times d} . \label{eqn:FG_hat} 
\end{align}
We further introduce the notation: 
\begin{align}
\wt{G} &\coloneqq \wh{G} G^{-1} \in\bbR^{n\times n} , \qquad 
\wt{F} \coloneqq \wh{F} F^{-1} \in\bbR^{d\times d} , \label{eqn:FG_tilde}
\end{align}
so that \Cref{eqn:rewrite} can be rewritten as 
\begin{align}
\wt{G} G u &= \Xi^{-1} A \Sigma^{-1} F v , \qquad 
\wt{F} F v = \Sigma^{-1} A^\top \Xi^{-1} G u , \notag 
\end{align}
or
\begin{align}
\wt{G}^{1/2} G u &= \wt{G}^{-1/2} \Xi^{-1} A \Sigma^{-1} \wt{F}^{-1/2} \cdot \wt{F}^{1/2} F v , \qquad 
\wt{F}^{1/2} F v = \wt{F}^{-1/2} \Sigma^{-1} A^\top \Xi^{-1} \wt{G}^{-1/2} \cdot \wt{G}^{1/2} G u . \notag 
\end{align}
The key observation is that this is a pair of singular vector equations for the matrix 
\begin{align}
A^* &\coloneqq \wt{G}^{-1/2} \Xi^{-1} A \Sigma^{-1} \wt{F}^{-1/2} \in\bbR^{n\times d} \label{eqn:A*}
\end{align}
with respect to left/right singular vectors (up to rescaling)
\begin{align}
\wt{G}^{1/2} G u \in\bbR^n , \qquad \wt{F}^{1/2} F v \in\bbR^d \notag 
\end{align}
and singular value $1$. 
Using the definitions \Cref{eqn:FG_bayes,eqn:FG_hat,eqn:FG_tilde}, we verify that the two expressions of $ A^* $ in \Cref{eqn:A*,eqn:A*_main} are equal. 

By the state evolution result (\Cref{prop:SE_bayes}), $ u, v $ behave (in the sense of \Cref{eqn:SE_result_uv}) as 
\begin{align}
    \mu^* \Xi^{-1} u^* + \sigma^* \Xi^{-1/2} W_U , \qquad 
    \nu^* \Sigma^{-1} v^* + \tau^* \Sigma^{-1/2} W_V, \notag 
\end{align}
for $ W_U \sim \cN(0_n, I_d), W_V \sim \cN(0_d, I_d) $ independent of each other and of $ u^*, v^* $. 
This suggests that
\begin{align}
    \Xi (\wt{G}^{1/2} G)^{-1} u_1(A^*) , \qquad 
    \Sigma (\wt{F}^{1/2} F)^{-1} v_1(A^*) \label{eqn:uv_hat_heu} 
\end{align}
are effective estimates of $ u^*, v^* $. 
Simple algebra reveals that the above vectors, upon suitable rescaling, are precisely $ \wh{u}, \wh{v} $ in \Cref{eqn:def_spec}. 

\subsection{Right edge of the bulk}
\label{sec:bulk}

Before proceeding with the proof of \Cref{thm:spec}, we provide a characterization of $\sigma_2(A^*)$, i.e., the right edge of the bulk of the spectrum of $ A^* $. 
This bulk is related to the spectrum of the non-spiked random matrix 
\begin{align}
    W^* \coloneqq 
    \lambda \paren{ \lambda(\mu^* + b^*) I_n + \Xi }^{-1/2} \wt{W} \paren{ \lambda(\nu^* + c^*) I_d + \Sigma }^{-1/2}
    . \label{eqn:W*}
\end{align}
We first present a characterization of $ \sigma_1(W^*) $ and then relate it to $ \sigma_2(A^*) $. 
%
Define random variables:
\begin{align}
    \ol{\Xi}^* &\coloneqq \frac{\lambda}{\lambda(\mu^* + b^*) + \ol{\Xi}}
    , \qquad 
    \ol{\Sigma}^* \coloneqq \frac{\lambda}{\lambda(\nu^* + c^*) + \ol{\Sigma}} . \label{eqn:Xi_Sigma_bar_star} 
\end{align}
Define functions $ c,s \colon (\sup\supp(\ol{\Xi}^*), \infty) \to (0,\infty) $ as 
\begin{align}
    c(\alpha) &= \expt{ \frac{\ol{\Xi}^*}{\alpha - \ol{\Xi}^*} } , \quad 
    s(\alpha) = \sup\supp(\ol{\Sigma}^*) c(\alpha) . \notag
\end{align}
Define the implicit function $ \beta\colon(\sup\supp(\ol{\Xi}^*), \infty) \to (0,\infty) $ as, for any $ \alpha\in(\sup\supp(\ol{\Xi}^*), \infty) $, the unique solution in $ (s(\alpha), \infty) $ to 
\begin{align}
    1 &= \frac{1}{\delta} \expt{ \frac{\ol{\Sigma}^*}{\beta - c(\alpha) \ol{\Sigma}^*} } . \notag 
\end{align}
(The existence and uniqueness of the solution is easy to see.)
Next, define $ \psi \colon (\sup\supp(\ol{\Xi}^*), \infty) \to (0,\infty) $ as $ \psi(\alpha) = \alpha \beta(\alpha) $. 
It is known that $\psi$ is differentiable and the set of its critical points is a nonempty finite set \cite{zhang2023spectral}. 
Let $ \alpha^\circ \in (\sup\supp(\ol{\Xi}^*), \infty) $ be the largest critical point, i.e., 
\begin{align}
    1 &= \frac{1}{\delta} \expt{ \frac{{\ol{\Sigma}^*}^2}{(\beta(\alpha) - c(\alpha) \ol{\Sigma}^*)^2} }
    \expt{ \frac{{\ol{\Xi}^*}^2}{(\alpha - \ol{\Xi}^*)^2} } \notag 
\end{align}
Finally, denote 
\begin{equation}\label{eq:sigma2}
\sigma_2^* \coloneqq \sqrt{\psi(\alpha^\circ)}.    
\end{equation}

The characterization of $ \sigma_1(W^*) $ requires an extra technical assumption on the random variables $ \ol{\Xi}, \ol{\Sigma} $, which is the same as in 
\cite{zhang2023spectral}. 

\begin{enumerate}[label=(A\arabic*)]
\setcounter{enumi}{\value{asmpctr}}

    \item \label[asmp]{asmp:edge}
    For any $c>0$, 
    \begin{align}
        \lim_{\beta\downarrow s} \expt{ \frac{\ol{\Sigma}^*}{\beta - c \ol{\Sigma}^*} }
        = \lim_{\beta\downarrow s} \expt{ \paren{\frac{\ol{\Sigma}^*}{ \beta - c \ol{\Sigma}^* }}^2 }
        = \infty . \notag 
    \end{align}
    where $ s \coloneqq c \cdot \sup\supp(\ol{\Sigma}^*) $. 
    Furthermore, 
    \begin{align}
        \lim_{\alpha\downarrow\sup\supp(\ol{\Xi}^*)} \expt{\frac{\ol{\Xi}^*}{\alpha - \ol{\Xi}^*}}
        = \infty . \notag 
    \end{align}
    
\setcounter{asmpctr}{\value{enumi}}
\end{enumerate}

\begin{lemma}
\label{lem:sigma1_W*}
Let \Cref{asmp:edge} hold. 
Consider $ W^* $ defined in \Cref{eqn:W*}. 
Then, we have
\begin{align}
\plim_{n\to\infty} \sigma_1(W^*) &= \sigma_2^* . \notag 
\end{align}
\end{lemma}

\begin{proof}
Note that $ W^* {W^*}^\top $ is a separable covariance matrix. 
Its largest eigenvalue is characterized in \cite{CouilletHachem}. 
The explicit formulas we need are due to \cite{zhang2023spectral}. 
To apply their results, one simply observes that the covariances (as in the context of separable covariance matrices) of $ W^* $ are
\begin{align}
    \Xi^* &\coloneqq \sqrt{\lambda} \paren{ \lambda(\mu^* + b^*) I_n + \Xi }^{-1/2} , \qquad 
    \Sigma^* \coloneqq \sqrt{\lambda} \paren{ \lambda(\nu^* + c^*) I_d + \Sigma }^{-1/2}, \notag
\end{align}
whose limiting spectral distributions are given by the distributions of $ \ol{\Xi}^*, \ol{\Sigma}^* $ in \Cref{eqn:Xi_Sigma_bar_star}. 
\end{proof}

\begin{lemma}
\label{lem:sigma2}
Consider $ A^* $ defined in \Cref{eqn:A*_main}. Then 
\begin{align}
    \plim_{n\to\infty} \sigma_2(A^*) &= \sigma_2^* . \notag 
\end{align}
\end{lemma}

\begin{proof}
By Weyl's inequality, $ \sigma_3(W^*) \le \sigma_1(A^*) \le \sigma_1(W^*) $. 
We have already shown in \Cref{lem:sigma1_W*} that $ \sigma_1(W^*) $ converges to $ \sigma_2^* $. 
The almost sure weak convergence of the empirical spectral distribution of $ W^* $ \cite[Theorem 1.2.1]{Zhang_Thesis_RMT} implies that $ \sigma_3(W^*) $ (and indeed $ \sigma_k(W^*) $ for any constant $k$ relative to $n,d$) must also converge to the same limit $ \sigma_2^* $. 
\end{proof}


\subsection{Proof of \Cref{thm:spec}}
\label{app:pf_spec}

We suppose throughout the proof that the condition \Cref{eqn:thr_bayes} holds. 
Then, by \Cref{prop:bayes_fp_sol} and the change of variable \Cref{eqn:quv_to_mu_nu}, the fixed point equation \Cref{eqn:BAMP_SE_FP} admits a unique non-trivial solution $ (\mu^*, \nu^*)\in\bbR_{>0}^2 $. 
Construct matrices $ F, G $ as in \Cref{eqn:FG_bayes} using such $ \mu^*,\nu^* $. 
Define also the random variables
\begin{align}
    \ol{G} &\coloneqq \frac{\lambda}{\lambda\mu^* \ol{\Xi}^{-1} + 1} , \quad 
    \ol{F} \coloneqq \frac{\lambda}{\lambda\nu^* \ol{\Sigma}^{-1} + 1} \notag 
\end{align}
whose distributions are the limiting spectral distributions of $ G, F $, respectively. 

Now consider the denoising functions: 
\begin{align}
f_{t+1}(v^{t+1}) &= F v^{t+1} , \quad 
g_t(u^t) = G u^t . \notag 
\end{align}
With this choice, the AMP iteration \Cref{eqn:BAMP} becomes
\begin{align}
u^t &= \Xi^{-1} A \Sigma^{-1} F v^t - b \Xi^{-1} G u^{t-1} , \quad 
v^{t+1} = \Sigma^{-1} A^\top \Xi^{-1} G u^t - c \Sigma^{-1} F v^t
, \label{eqn:AMP_spec_bayes} 
\end{align}
where
\begin{align}
b &= \frac{1}{n} \tr(F \Sigma^{-1}) , \quad 
c = \frac{1}{n} \tr(G \Xi^{-1}) . \notag 
\end{align}
Note that as $n\to\infty$, $b,c$ converge to $ b^*,c^* $ in \Cref{eqn:bc}. 

Recall from \Cref{eqn:sigma_tau_fp} the definition of $ \sigma^*, \tau^* $. 
We initialize \Cref{eqn:AMP_spec_bayes} with 
\begin{align}
\wt{u}^{-1} &= 0_n , \quad 
\wt{v}^0 = F(\nu^* \Sigma^{-1} v^* + \tau^* \Sigma^{-1/2} w) , \notag 
\end{align}
where $ w\sim\cN(0_d, I_d) $ is independent of everything else. 
Accordingly, one can take $ f_0 $ in \Cref{asmp:init_bayes} to be $ f_0(v) = \nu^* F \Sigma^{-1} v $. 
Under the above AMP initializer, the state evolution initializers in \Cref{eqn:mu_nu,eqn:Phi_Psi} specialize to 
\begin{align}
\mu_0 &= \lim_{n\to\infty} \frac{\lambda}{n} \expt{ {V^*}^\top \Sigma^{-1} F \Sigma^{-1} V^* } \nu^*
= \frac{\lambda}{\delta} \expt{ \ol{F} \ol{\Sigma}^{-2} } \nu^*
= \mu^* , \notag \\ 
\sigma_0^2 &= \plim_{n\to\infty} \frac{{\nu^*}^2}{n} {v^*}^\top \Sigma^{-1} F \Sigma^{-1} F \Sigma^{-1} v^*
+ \frac{{\tau^*}^2}{n} w^\top \Sigma^{-1/2} F \Sigma^{-1} F \Sigma^{-1/2} w \notag \\
&= \frac{1}{\delta} \expt{ \ol{F}^2 \ol{\Sigma}^{-3} } {\nu^*}^2
+ \frac{1}{\delta} \expt{ \ol{F}^2 \ol{\Sigma}^{-2} } {\tau^*}^2
= {\sigma^*}^2 , \notag 
\end{align}
where the last equalities for both chains of computation are by \Cref{eqn:SE_bayes}. 
Since the parameters $ \mu_t, \sigma_t $ are initialized at the non-trivial fixed point $ (\mu_0, \sigma_0) = (\mu^*, \sigma^*) $, the state evolution recursion \Cref{eqn:SE_bayes} will stay at the fixed point $ (\mu_t, \sigma_t, \nu_{t+1}, \tau_{t+1}) = (\mu^*, \sigma^*, \nu^*, \tau^*) $ across all $t\ge0$.

\begin{lemma}
\label{lem:align}
Let
\begin{align}
\wh{u}^t &\coloneqq \wt{G}^{1/2} G u^t , \quad 
\wh{v}^{t+1} \coloneqq \wt{F}^{1/2} F v^{t+1} , \label{eqn:uv_hat_bayes}
\end{align}
where $ \wt{F}, \wt{G} $ are defined in \Cref{eqn:FG_tilde}. 
Suppose the condition \Cref{eqn:thr_bayes} holds. 
Then 
\begin{align}
\lim_{t\to\infty} \plim_{n\to\infty} \frac{\abs{\inprod{\wh{u}^t}{u_1(A^*)}}}{\normtwo{\wh{u}^t}} =
\lim_{t\to\infty} \plim_{n\to\infty} \frac{\abs{\inprod{\wh{v}^{t+1}}{v_1(A^*)}}}{\normtwo{\wh{v}^{t+1}}} = 1 \label{eqn:align_bayes}
\end{align}
and 
\begin{align}
\plim_{n\to\infty} \sigma_1(A^*) &= 1 , \label{eqn:sinval_bayes}
\end{align}
where $ A^* $ is defined in \Cref{eqn:A*}. 
\end{lemma}

\begin{proof}
Denoting 
\begin{align}
e_1^t &\coloneqq u^t - u^{t-1} , \quad 
e_2^{t+1} \coloneqq v^{t+1} - v^t ,\label{eqn:e12_bayes}
\end{align}
for any $t\ge1$ and using the notation $ \wh{F}, \wh{G} $ in \Cref{eqn:FG_hat}, we have from \Cref{eqn:AMP_spec_bayes} that 
\begin{align}
\wh{G} u^t &= \Xi^{-1} A \Sigma^{-1} F v^t + b^* \Xi^{-1} G e_1^t 
+ (b^* - b) \Xi^{-1} G u^{t-1} , \notag \\
\wh{F} v^{t+1} &= \Sigma^{-1} A^\top \Xi^{-1} G u^t + c^* \Sigma^{-1} F e_2^{t+1}
+ (c^* - c) \Sigma^{-1} F v^t . \notag 
\end{align}
Recalling the notation $ \wt{F}, \wt{G} $ from \Cref{eqn:FG_tilde} and multiplying $ \wt{G}^{-1/2} $ (resp.\ $ \wt{F}^{-1/2} $) on both sides of the first (resp.\ second) equation above, we arrive at
\begin{align}
\wt{G}^{1/2} G u^t &= A^* \cdot \wt{F}^{1/2} F v^t + b^* \wt{G}^{-1/2} \Xi^{-1} G e_1^t 
+ (b^* - b) \wt{G}^{-1/2} \Xi^{-1} G u^{t-1} , \notag \\ 
\wt{F}^{1/2} F v^{t+1} &= {A^*}^\top \cdot \wt{G}^{1/2} G u^t + c^* \wt{F}^{-1/2} \Sigma^{-1} F e_2^{t+1}
+ (c^* - c) \wt{F}^{-1/2} \Sigma^{-1} F v^t . \notag 
\end{align}
Using the definition of $ \wh{u}^t, \wh{v}^{t+1} $ in \Cref{eqn:uv_hat_bayes}, we rewrite the above as 
\begin{align}
\wh{u}^t &= A^* \wh{v}^t + e_u^t , \quad 
\wh{v}^{t+1} = {A^*}^\top \wh{u}^t + e_v^{t+1} , \label{eqn:uv+err}
\end{align}
where 
\begin{align}
\begin{split}
e_u^t &\coloneqq b \wt{G}^{-1/2} \Xi^{-1} G e_1^t
+ (b^* - b) \wt{G}^{-1/2} \Xi^{-1} G u^{t-1} , \\
e_v^{t+1} &\coloneqq c \wt{F}^{-1/2} \Sigma^{-1} F e_2^{t+1}
+ (c^* - c) \wt{F}^{-1/2} \Sigma^{-1} F v^t . 
\end{split}
\label{eqn:euv}
\end{align}
Let us focus on $ \wh{u}^t $ and only prove the first equality in \Cref{eqn:align_bayes}. 
The proof of the second one is similar and will be omitted. 
Eliminating $ \wh{v}^t $ from \Cref{eqn:uv+err} gives
\begin{align}
\wh{u}^t &= A^* {A^*}^\top \wh{u}^{t-1} + A^* e_v^t + e_u^t . \notag 
\end{align}
Unrolling this recursion for $s$ steps, we get:
\begin{align}
\wh{u}^{t+s} &= \paren{A^* {A^*}^\top}^s \wh{u}^t + \wh{e}_u^{t,s} , \label{eqn:pow_itr_bayes}
\end{align}
where 
\begin{align}
\wh{e}_u^{t,s} &\coloneqq \sum_{r = 1}^s \paren{A^* {A^*}^\top}^{s-r} (A^* e_v^{t+r} + e_u^{t+r}) . \label{eqn:eu_ts} 
\end{align}

Taking $ \frac{1}{n} \normtwo{\cdot}^2 $ on both sides of \Cref{eqn:pow_itr_bayes} and take the sequential limits of $ n\to\infty $, $ t\to\infty $, $ s\to\infty $, we have the left hand side:
\begin{align}
\lim_{s\to\infty} \lim_{t\to\infty} \plim_{n\to\infty} \frac{1}{n} \normtwo{\wh{u}^{t+s}}^2
&= \lim_{t\to\infty} \plim_{n\to\infty} \frac{1}{n} \normtwo{\wh{u}^t}^2 \notag \\
&= \lim_{t\to\infty} \plim_{n\to\infty} \frac{1}{n} \normtwo{ \wt{G}^{1/2} G u^t }^2 \notag \\
&= \lim_{t\to\infty} \plim_{n\to\infty} \frac{1}{n} \normtwo{ \wh{G}^{1/2} G^{1/2} u^t }^2 \notag \\
&= \lim_{t\to\infty} \mu_t^2 \expt{ \ol{\Xi}^{-2} (1 + b^* \ol{\Xi}^{-1} \ol{G}) \ol{G} } + \sigma_t^2 \expt{ \ol{\Xi}^{-1} (1 + b^* \ol{\Xi}^{-1} \ol{G}) \ol{G} } \notag \\
&= {\mu^*}^2 \expt{ \ol{\Xi}^{-2} (1 + b^* \ol{\Xi}^{-1} \ol{G}) \ol{G} } + \frac{\mu^*}{\lambda} \expt{ \ol{\Xi}^{-1} (1 + b^* \ol{\Xi}^{-1} \ol{G}) \ol{G} } \notag \\
&= \lambda {\mu^*}^2 \expt{ \frac{ \lambda (\mu^* + b^*) + \ol{\Xi} }{ \ol{\Xi} \paren{ \lambda \mu^* + \ol{\Xi} }^2 }  }
+ \mu^* \expt{ \frac{\lambda (\mu^* + b^*) + \ol{\Xi}}{\paren{ \lambda \mu^* + \ol{\Xi} }^2} } \in (0,\infty) . 
\label{eqn:norm_lim}
\end{align}
Next, we have that 
\begin{align}
\lim_{t\to\infty} \plim_{n\to\infty} \frac{1}{n} \normtwo{e_1^t}^2
&= \lim_{t\to\infty} \plim_{d\to\infty} \frac{1}{d} \normtwo{e_2^{t+1}}^2
= 0 . 
\label{eqn:e12=0} 
\end{align}
To prove the first statement on $ e_1^t $, the strategy is to write the LHS of \Cref{eqn:e12=0} in terms of the state evolution parameters and prove that the latter quantities converge. 
We start with
\begin{align}
\plim_{n\to\infty} \frac{1}{n} \normtwo{ u^t - u^{t-1} }^2
&= \paren{ \mu_t^2 \expt{\ol{\Sigma}^{-2}} + \sigma_t^2 \expt{\ol{\Sigma}^{-1}} }
+ \paren{ \mu_{t-1}^2 \expt{\ol{\Sigma}^{-2}} + \sigma_{t-1}^2 \expt{\ol{\Sigma}^{-1}} } \notag \\
&\quad - 2 \paren{ \mu_t \mu_{t-1} \expt{\ol{\Sigma}^{-2}} + \Phi_{t, t-1} \expt{\ol{\Sigma}^{-1}} } \notag \\
&= 2 {\sigma^*}^2 \expt{\ol{\Sigma}^{-1}} - 2 \Phi_{t,t-1} \expt{\ol{\Sigma}^{-1}} , \notag 
\end{align}
where the first equality is by \Cref{prop:quadratic_form_corr} and the joint distribution of $ W_{U,t}, W_{U,t-1} $ in \Cref{eqn:all_joint}, and the second one holds since the state evolution parameters stay at the fixed point upon initialization. 
So it remains to verify that $ \Phi_{t,t-1} \to {\sigma^*}^2 $ as $ t\to\infty $. 
To see this, note that according to the state evolution recursion \Cref{eqn:Phi_Psi}, 
\begin{align}
\Phi_{t,t-1} &= \lim_{n\to\infty} \frac{1}{n} \expt{ V_t^\top F^\top \Sigma^{-1} F V_{t-1} }
= \frac{{\nu^*}^2}{\delta} \expt{ \ol{F}^2 \ol{\Sigma}^{-3} } + \frac{\Psi_{t,t-1}}{\delta} \expt{ \ol{F}^2 \ol{\Sigma}^{-2} } , \notag \\
\Psi_{t+1,t} &= \lim_{n\to\infty} \frac{1}{n} \expt{ U_t^\top G^\top \Xi^{-1} G U_{t-1} }
= {\mu^*}^2 \expt{ \ol{G}^2 \ol{\Xi}^{-3} }
+ \Phi_{t,t-1} \expt{\ol{G}^2 \ol{\Xi}^{-2}} . \notag 
\end{align}
Eliminating $ \Psi_{t,t-1} $ from the first equation, we arrive at 
\begin{align}
\Phi_{t,t-1} &=
\frac{{\nu^*}^2}{\delta} \expt{ \ol{F}^2 \ol{\Sigma}^{-3} } + \frac{1}{\delta} \expt{ \ol{F}^2 \ol{\Sigma}^{-2} }
\paren{ {\mu^*}^2 \expt{ \ol{G}^2 \ol{\Xi}^{-3} }
+ \Phi_{t-1,t-2} \expt{\ol{G}^2 \ol{\Xi}^{-2}} } \notag \\
&= \frac{1}{\delta} \expt{ \frac{\lambda^2 \ol{\Sigma}^{-3}}{(\lambda\nu^* \ol{\Sigma}^{-1} + 1)^2} } {\nu^*}^2 + \frac{1}{\delta} \expt{ \frac{\lambda^2 \ol{\Sigma}^{-2}}{(\lambda\nu^* \ol{\Sigma}^{-1} + 1)^2} } \notag \\
&\qquad \times \paren{
\expt{ \frac{\lambda^2 \ol{\Xi}^{-3}}{(\lambda\mu^* \ol{\Xi}^{-1} + 1)^2} } {\mu^*}^2
+ \expt{ \frac{\lambda^2 \ol{\Xi}^{-2}}{(\lambda\mu^* \ol{\Xi}^{-1} + 1)^2} } \Phi_{t-1,t-2}
} , \notag 
\end{align}
whose only fixed point is $ {\sigma^*}^2 $ by the relations in \Cref{eqn:SE_bayes}. 
This concludes the proof of the first equality in \Cref{eqn:e12=0}. The proof of the second equality is analogous and, hence, omitted.

By using \Cref{eqn:e12=0} and the fact that $ b\to b^*, c\to c^* $ as $n\to\infty$ (see \Cref{eqn:bc}), we obtain
\begin{align}\label{eq:euev}
    \lim_{t\to\infty} \plim_{n\to\infty} \frac{1}{n} \normtwo{e_u^t}^2 &= 
    \lim_{t\to\infty} \plim_{d\to\infty} \frac{1}{d} \normtwo{e_v^{t+1}}^2 = 0 . 
\end{align}
Note that the operator norm of $ A^* $ is almost surely bounded uniformly in $n$ by Weyl's inequality, sub-multiplicativity of matrix norms and the Bai--Yin law \cite{BaiYin}. 
This together with the triangle inequality of the $\ell_2$-norm and \Cref{eq:euev} implies that
\begin{align}
\lim_{s\to\infty} \lim_{t\to\infty} \plim_{n\to\infty} \frac{1}{n} \normtwo{\wh{e}_u^{t,s}}^2 &= 0 , \label{eqn:err_bayes}
\end{align}
From this, it follows that the right hand side of \Cref{eqn:pow_itr_bayes} (upon taken the rescaled squared norm and the sequential limits) equals
\begin{align}
\lim_{s\to\infty} \lim_{t\to\infty} \plim_{n\to\infty} \frac{1}{n} \normtwo{ \paren{A^* {A^*}^\top}^s \wh{u}^t }^2 . \notag 
\end{align}

We then compute the above term by taking the SVD of $ A^* $. 
Define two spectral projectors that are orthogonal to each other:
\begin{align}
    \Pi &\coloneqq u_1(A^*) u_1(A^*)^\top , \quad 
    \Pi^\perp \coloneqq I_n - \Pi . \notag 
\end{align}
We have 
\begin{align}
    \frac{1}{n} \normtwo{ \paren{ A^* {A^*}^\top }^s \wh{u}^t }^2
    &= \frac{1}{n} \normtwo{ \paren{ A^* {A^*}^\top }^s \Pi \wh{u}^t }^2
    + \frac{1}{n} \normtwo{ \paren{ A^* {A^*}^\top }^s \Pi^\perp \wh{u}^t }^2 . \label{eqn:proj12}
\end{align}
Using the spectral decomposition 
\begin{align}
    \paren{ A^* {A^*}^\top }^s
    &= \sum_{i = 1}^n \sigma_i(A^*)^{2s} u_i(A^*) u_i(A^*)^\top , \notag 
\end{align}
we can write the first term in \Cref{eqn:proj12} as
\begin{align}
    \frac{1}{n} \normtwo{ \paren{ A^* {A^*}^\top }^s \Pi \wh{u}^t }^2
    &= \frac{1}{n} \normtwo{ \sum_{i = 1}^n \sigma_i(A^*)^{2s} u_i(A^*) u_i(A^*)^\top u_1(A^*) u_1(A^*)^\top \wh{u}^t }^2
    = \sigma_1(A^*)^{4s} \frac{\inprod{u_1(A^*)}{\wh{u}^t}^2}{n} . \label{eqn:proj1} 
\end{align}

For the second term in \Cref{eqn:proj12}, we have
\begin{align}
    \frac{1}{n} \normtwo{ \paren{ A^* {A^*}^\top }^s \Pi^\perp \wh{u}^t }^2
    &= \frac{1}{n} \normtwo{ \paren{ A^* {A^*}^\top \Pi^\perp }^s \wh{u}^t }^2 \notag \\
    &\le \frac{\normtwo{\wh{u}^t}^2}{n} \max_{u\in\bbS^{n-1}} \normtwo{ \paren{ A^* {A^*}^\top \Pi^\perp }^s u }^2 \notag \\
    &= \frac{\normtwo{\wh{u}^t}^2}{n} \sigma_1\paren{ \paren{ A^* {A^*}^\top \Pi^\perp }^s }^2 \notag \\
    &= \frac{\normtwo{\wh{u}^t}^2}{n} \sigma_1\paren{ A^* {A^*}^\top \Pi^\perp }^{2s} \notag \\
    &= \frac{\normtwo{\wh{u}^t}^2}{n} \sigma_2\paren{ A^* {A^*}^\top }^{2s} \notag \\
    &= \frac{\normtwo{\wh{u}^t}^2}{n} \sigma_2\paren{ A^* }^{4s} , \notag 
\end{align}
where penultimate line follows since
\begin{align}
    A^* {A^*}^\top \Pi^\perp
    &= \sum_{i = 2}^n \sigma_i(A^*)^2 u_i(A^*) u_i(A^*) ^\top . \notag 
\end{align}

From \Cref{lem:sigma2} and the assumption that $\sigma_2^*<1$, we know
\begin{align}
    \plim_{n\to\infty} \sigma_2(A^*) &= \sigma_2^* < 1 . \notag 
\end{align}
This implies: 
\begin{multline}
    \lim_{s\to\infty} \lim_{t\to\infty} \plimsup_{n\to\infty} \frac{1}{n} \normtwo{ \paren{ A^* {A^*}^\top }^s \Pi^\perp \wh{u}^t }^2
    \le \lim_{s\to\infty} \lim_{t\to\infty} \plimsup_{n\to\infty} \frac{\normtwo{\wh{u}^t}^2}{n} \sigma_2\paren{ A^* }^{4s} \\
    \le \paren{ \lim_{t\to\infty} \plimsup_{n\to\infty} \frac{\normtwo{\wh{u}^t}^2}{n} } 
    \paren{ \lim_{s\to\infty} \plimsup_{n\to\infty} \sigma_2\paren{ A^* }^{4s} }
    = 0 , \label{eqn:proj2} 
\end{multline}
where the last equality holds since the limit in the first parentheses is finite by \Cref{eqn:norm_lim}. 

Now \Cref{eqn:proj12,eqn:proj1,eqn:proj2} jointly show
\begin{align}
    \lim_{s\to\infty} \lim_{t\to\infty} \plim_{n\to\infty} \frac{1}{n} \normtwo{ \paren{A^* {A^*}^\top}^s \wh{u}^t }^2
    &= \lim_{s\to\infty} \lim_{t\to\infty} \plim_{n\to\infty}
    \sigma_1(A^*)^{4s} \frac{\inprod{u_1(A^*)}{\wh{u}^t}^2}{n} \notag \\
    &= \paren{ \lim_{s\to\infty} \plim_{n\to\infty} \sigma_1(A^*)^{4s} }
    \paren{ \lim_{t\to\infty} \plim_{n\to\infty} \frac{\inprod{u_1(A^*)}{\wh{u}^t}^2}{n} } . \notag 
\end{align}
Combining this with \Cref{eqn:norm_lim} brings us to the following identity:
\begin{align}
    1 &= \paren{ \lim_{s\to\infty} \plim_{n\to\infty} \sigma_1(A^*)^{4s} }
    \paren{ \lim_{t\to\infty} \plim_{n\to\infty} \frac{\inprod{u_1(A^*)}{\wh{u}^t}^2}{\normtwo{\wh{u}^t}^2} } , \notag 
\end{align}
which necessarily implies
\begin{align}
    \plim_{n\to\infty} \sigma_1(A^*) &= 1 , \qquad 
    \lim_{t\to\infty} \plim_{n\to\infty} \frac{\inprod{u_1(A^*)}{\wh{u}^t}^2}{\normtwo{\wh{u}^t}^2} = 1 , \notag 
\end{align}
as desired. 
\end{proof}

With \Cref{lem:align}, we can complete the proof of \Cref{thm:spec}. 

\begin{proof}[Proof of \Cref{thm:spec}]
The characterization \Cref{eqn:sing_char} of the top two singular values have been obtained in \Cref{lem:align,lem:sigma2}. 
It remains to compute the overlaps which can be done using \Cref{lem:align} and the state evolution (\Cref{prop:SE_bayes}). 
Recall the estimators $ \wh{u}, \wh{v} $ in \Cref{eqn:def_spec} and their heuristic derivation in \Cref{eqn:uv_hat_heu}. 
Then 
\begin{align}
    \plim_{n\to\infty} \frac{\inprod{\wh{u}}{u^*}^2}{\normtwo{\wh{u}}^2 \normtwo{u^*}^2}
    &= \plim_{t\to\infty} \plim_{n\to\infty} \frac{ \inprod{\Xi u^t}{u^*}^2 }{ \normtwo{\Xi u^t}^2 \normtwo{u^*}^2 } 
    = \frac{\lambda \mu^* }{\lambda \mu^* + 1} = \eta_u^2 , \notag 
\end{align}
establishing the first equality in \Cref{eqn:ol_char}. 
The second equality in \Cref{eqn:ol_char} and other quantities in \Cref{eqn:mmse_v_char,eqn:mmse_uv_char} can be similarly obtained. 
The proof is completed. 
\end{proof}



\section{Proof of \Cref{prop:SE_bayes}}
\label{app:pf_BAMP_SE}

Recall $ \wt{u}^*, \wt{v}^* $ from \Cref{eqn:uv*_tilde} and let
\begin{align}
    \wt{A} &\coloneqq \Xi^{-1/2} A \Sigma^{-1/2}
    = \frac{\lambda}{n} \wt{u}^* (\wt{v}^*)^\top + \wt{W} . 
    \label{eqn:A_tilde}
\end{align}


\subsection{Auxiliary AMP and its state evolution}
\label{sec:AMP}

For $ (\brv{g}_t\colon\bbR^n\to\bbR^n, \brv{f}_{t+1}\colon\bbR^d\to\bbR^d)_{t\ge0} $, the iterates of the auxiliary AMP, initialized at $ \mathring{u}^{-1} = 0_d $ and some $ \mathring{v}^0\in\bbR^d $, are updated according to the following rules for every $ t\ge0 $: 
\begin{align}
\begin{split}
\brv{u}^t &= \wt{A} \mathring{v}^t - \brv{b}_t \mathring{u}^{t-1} , \quad
\mathring{u}^t = \brv{g}_t(\brv{u}^t) , \quad 
\brv{c}_t = \frac{1}{n} \sum_{i = 1}^n \partial_i \brv{g}_t(\brv{u}^t)_i , \\
\brv{v}^{t+1} &= \wt{A}^\top \mathring{u}^t - \brv{c}_t \mathring{v}^t , \quad 
\mathring{v}^{t+1} = \brv{f}_{t+1}(\brv{v}^{t+1}) , \quad 
\brv{b}_{t+1} = \frac{1}{n} \sum_{i = 1}^d \partial_i \brv{f}_{t+1}(\brv{v}^{t+1})_i . 
\end{split}
\label{eqn:AMP}
\end{align}

The state evolution result associated with the above auxiliary AMP iteration asserts that the distributions of $ (\brv{u}^0, \brv{u}^1, \cdots, \brv{u}^t), (\brv{v}^1, \brv{v}^2, \cdots, \brv{v}^{t+1}) $ converge (in the sense of \Cref{eqn:SE_result}) respectively to the laws of the random vectors $ (\brv{U}_0, \brv{U}_1, \cdots, \brv{U}_t), (\brv{V}_1, \brv{V}_2, \cdots \brv{V}_{t+1}) $ defined below: 
\begin{align}
\brv{U}_t &= \brv{\mu}_t \wt{U}^* + \brv{\sigma}_{t} \brv{W}_{U,t} \in\bbR^n , \quad 
\brv{V}_{t+1} = \brv{\nu}_{t+1} \wt{V}^* + \brv{\tau}_{t+1} \brv{W}_{V,t+1} \in\bbR^d , \label{eqn:UV}
\end{align}
where 
\begin{align}
&&
\matrix{
    U^* \\ 
    \brv{\sigma}_{0} \brv{W}_{U,0} \\
    \brv{\sigma}_{1} \brv{W}_{U,1} \\
    \vdots \\
    \brv{\sigma}_{t} \brv{W}_{U,t}
} &\sim P^{\ot n} \ot \cN(0_{n(t+1)}, \brv{\Phi}_t \ot I_n) , &
\matrix{
    V^* \\ 
    \brv{\tau}_{1} \brv{W}_{V,1} \\ 
    \brv{\tau}_{2} \brv{W}_{V,2} \\
    \vdots \\
    \brv{\tau}_{t+1} \brv{W}_{V,t+1}
} &\sim Q^{\ot d} \ot \cN(0_{d(t+1)}, \brv{\Psi}_{t+1} \ot I_d) , && \label{eqn:joint_distr} \\
&&
\wt{U}^* &\coloneqq \Xi^{-1/2} U^*\in\bbR^n , &
\wt{V}^* &\coloneqq \Sigma^{-1/2} V^*\in\bbR^d . && \label{eqn:UV_tilde} 
\end{align}
The parameters $ \brv{\mu}_t, \brv{\nu}_{t+1} \in \bbR, \brv{\Phi}_t = (\brv{\Phi}_{r,s})_{0\le r,s\le t}, \brv{\Psi}_{t+1} = (\brv{\Psi}_{r+1,s+1})_{0\le r,s\le t} \in \bbR^{(t+1)\times (t+1)} $ are defined recursively through the following state evolution equations:
\begin{align}
\brv{\mu}_0 &= \lambda \lim_{d\to\infty} \frac{1}{n} \expt{\inprod{\wt{V}^*}{\brv{f}_0(\wt{V}^*)}} , \label{eqn:SE_mu0} \\
\brv{\mu}_t &= \lambda \lim_{d\to\infty} \frac{1}{n} \expt{\inprod{\wt{V}^*}{\brv{f}_t(\brv{V}_t)}} , \quad t\ge1 , \label{eqn:SE_mu} \\
\brv{\nu}_{t+1} &= \lambda \lim_{n\to\infty} \frac{1}{n} \expt{\inprod{\wt{U}^*}{\brv{g}_t(\brv{U}_t)}} , \quad t\ge0 ,\label{eqn:SE_nu} \\
\brv{\Phi}_{0,0} &= \plim_{n\to\infty} \frac{1}{n} \normtwo{\mathring{v}^0}^2 , \label{eqn:Phi_11} \\
\brv{\Phi}_{0,s} &= \lim_{n\to\infty} \frac{1}{n} \expt{ \inprod{\brv{f}_0(\wt{V}^*)}{\brv{f}_s(\brv{V}_s)} } , \quad 1\le s\le t , \label{eqn:Phi_1s} \\
\brv{\Phi}_{r, s} &= \lim_{n\to\infty} \frac{1}{n} \expt{\inprod{\brv{f}_r(\brv{V}_r)}{\brv{f}_s(\brv{V}_s)}} , \quad 1\le r,s\le t , \label{eqn:Phi_rs} \\
\brv{\Psi}_{r+1, s+1} &= \lim_{n\to\infty} \frac{1}{n} \expt{\inprod{\brv{g}_r(\brv{U}_r)}{\brv{g}_s(\brv{U}_s)}} , \quad 0\le r,s\le t , \label{eqn:Psi_rs}
\end{align}
where $ \brv{f}_0 $ is determined by $ \mathring{v}^0 $ through \Cref{asmp:init} below. 
In particular, 
\begin{align}
\brv{\sigma}_0^2 &= \plim_{n\to\infty} \frac{1}{n} \normtwo{\mathring{v}^0}^2 , \label{eqn:SE_sigma_init} \\
\brv{\sigma}_{t}^2 &= \lim_{n\to\infty} \frac{1}{n} \expt{\inprod{\brv{f}_t(\brv{V}_t)}{\brv{f}_t(\brv{V}_t)}} , \quad t \ge 1 ,\label{eqn:SE_sigma} \\
\brv{\tau}_{t+1}^2 &= \lim_{n\to\infty} \frac{1}{n} \expt{\inprod{\brv{g}_t(\brv{U}_t)}{\brv{g}_t(\brv{U}_t)}} , \quad t\ge0 . \label{eqn:SE_tau}
\end{align}

We require the following assumptions to guarantee the existence and finiteness of the state evolution parameters defined above. 

\begin{enumerate}[label=(A\arabic*)]
\setcounter{enumi}{\value{asmpctr}}

    \item \label[asmp]{asmp:init} $ \mathring{v}^0 $ is independent of $ \wt{W} $ but may depend on $ \wt{v}^* $. Assume that 
    \begin{align}
    \plim_{d\to\infty} \frac{1}{d} \normtwo{\mathring{v}^0}^2 \notag
    \end{align}
    exists and is finite. 
    There exists a uniformly pseudo-Lipschitz function $ \brv{f}_0\colon\bbR^d\to\bbR^d $ of order $1$ such that 
    \begin{align}
    \lim_{d\to\infty} \frac{1}{d} \expt{ \inprod{\brv{f}_0(\wt{V}^*)}{\brv{f}_0(\wt{V}^*)} } &\le \plim_{d\to\infty} \frac{1}{d} \normtwo{\mathring{v}^0}^2 \notag 
    \end{align}
    and for every uniformly pseudo-Lipschitz function $ \phi\colon\bbR^d\to\bbR^d $ of finite order, the following two limits exist, are finite and equal:
    \begin{align}
    \plim_{d\to\infty} \frac{1}{d} \inprod{\mathring{v}^0}{\phi(\wt{v}^*)}
    &= \lim_{d\to\infty} \frac{1}{d} \expt{ \inprod{\brv{f}_0(\wt{V}^*)}{\phi(\wt{V}^*)} } . \notag 
    \end{align}
    Let $ \wt{\nu}\in\bbR, \wt{\tau}\in\bbR_{\ge0} $. 
    For any $s\ge1$, 
    \begin{align}
    \lim_{d\to\infty} \frac{1}{d} \expt{ \brv{f}_0(\wt{V}^*)^\top \brv{f}_s(\wt{\nu} \wt{V}^* + \wt{\tau} W_V) } \notag 
    \end{align}
    exists and is finite, where $ W_V\sim\cN(0_d, I_d) $ is independent of $ \wt{V}^* $. 

    \item \label[asmp]{asmp:denois_fn}
    Let $ \wt{\nu}\in\bbR $, and $ T\in\bbR^{2\times2} $ be positive definite. 
    For any $ r,s\ge1 $, 
    \begin{align}
    &\lim_{n\to\infty} \frac{\lambda}{n} \expt{ \inprod{\wt{V}^*}{\brv{f}_r(\wt{\nu} \wt{V}^* + N)} } , \quad 
    \lim_{d\to\infty} \frac{1}{d} \expt{ \brv{f}_r(\wt{\nu} \wt{V}^* + N)^\top \brv{f}_s(\wt{\nu} \wt{V}^* + N') } \notag 
    \end{align}
    exist and are finite, where $ (N,N') \sim \cN(0_{2d}, T \ot I_d) $ is independent of $ \wt{v}^* $. 

    Let $ \wt{\mu}\in\bbR $, and $ S\in\bbR^{2\times2} $ be positive definite. 
    For any $ r,s\ge0 $, 
    \begin{align}
    &\lim_{n\to\infty} \frac{\lambda}{n} \expt{ \inprod{\wt{U}^*}{\brv{g}_r(\wt{\mu} \wt{U}^* + M)} } , \quad
    \lim_{d\to\infty} \frac{1}{d} \expt{ \brv{g}_r(\wt{\mu} \wt{U}^* + M)^\top \brv{g}_s(\wt{\mu} \wt{U}^* + M') } \notag 
    \end{align}
    exist and are finite, where $ (M,M') \sim \cN(0_{2n}, S \ot I_n) $ is independent of $ \wt{U}^* $. 

\setcounter{asmpctr}{\value{enumi}}
\end{enumerate}

\begin{proposition}[State evolution for auxiliary AMP \Cref{eqn:AMP}]
\label{prop:SE}
For every $ t\ge0 $, let $ (\brv{g}_t\colon\bbR^{n}\to\bbR^n)_{n\ge1} $ and $ (\brv{f}_{t+1}\colon\bbR^{d}\to\bbR^d)_{d\ge1} $ be uniformly pseudo-Lipschitz of finite order subject to \Cref{asmp:denois_fn}. 
Consider the auxiliary AMP iteration in \Cref{eqn:AMP} defined by $ (\brv{g}_t, \brv{f}_{t+1})_{t\ge0} $ and initialized at $ \mathring{u}^{-1} = 0_n $ and some $ \mathring{v}^0\in\bbR^d $ subject to \Cref{asmp:init}. 
For any fixed $ t\ge0 $, let $ (\phi\colon\bbR^{(t+2)n} \to \bbR)_{n\ge1} $ and $ (\psi\colon\bbR^{(t+2)d} \to \bbR)_{n\ge1} $ be uniformly pseudo-Lipschitz functions of finite order. 
Then, 
\begin{subequations}
\begin{align}
\begin{split}
\plim_{n\to\infty} \phi(\wt{u}^*, \brv{u}^0, \brv{u}^1, \cdots, \brv{u}^t) - \expt{\phi(\wt{U}^*, \brv{U}_0, \brv{U}_1, \cdots, \brv{U}_t)} &= 0 , 
\end{split} \label{eqn:SE_result_u} \\
\begin{split}
\plim_{d\to\infty} \psi(\wt{v}^*, \brv{v}^1, \brv{v}^2, \cdots, \brv{v}^{t+1}) - \expt{\psi(\wt{V}^*, \brv{V}_1, \brv{V}_2, \cdots, \brv{V}_{t+1})} &= 0 , 
\end{split} \label{eqn:SE_result_v}
\end{align} \label{eqn:SE_result}
\end{subequations}
where $ (\brv{U}_s, \brv{V}_{s+1})_{0\le s\le t} $ are defined in \Cref{eqn:UV}. 
\end{proposition}



\begin{proof}[Proof of \Cref{prop:SE}]
By definitions of the auxiliary AMP \Cref{eqn:AMP} and the matrix $ \wt{A} $ in \Cref{eqn:A_tilde}, we have that for every $ t\ge0 $, 
\begin{align}
\brv{u}^t &= \wt{A} \mathring{v}^t - \brv{b}_t \mathring{u}^{t-1} 
= \frac{\lambda}{n} \inprod{\wt{v}^*}{\mathring{v}^t} \, \wt{u}^* + \wt{W} \mathring{v}^t - \brv{b}_t \mathring{u}^{t-1}
, \notag \\
\brv{v}^{t+1} &= \wt{A}^\top \mathring{u}^t - \brv{c}_t \mathring{v}^t
= \frac{\lambda}{n} \inprod{\wt{u}^*}{\mathring{u}^t} \, \wt{v}^* + \wt{W}^\top \mathring{u}^t - \brv{c}_t \mathring{v}^t
. \notag 
\end{align}
For every $ t\ge0 $, let us consider a pair of related iterates $ p^t, q^{t+1} $ with initialization 
\begin{align}
\wt{p}^{-1} &= 0_n , \quad 
\wt{q}^0 = \mathring{v}^0 \label{eqn:pq_init}
\end{align}
and update rules: 
\begin{align}
\begin{split}
p^t &= \wt{W} \wt{q}^t - \ell_t \wt{p}^{t-1} , \quad  
\wt{p}^t = \brv{g}_t(p^t + \brv{\mu}_t \wt{u}^*) , \quad 
m_t = \frac{1}{n} \sum_{i = 1}^n \partial_i \brv{g}_t(p^t + \brv{\mu}_t \wt{u}^*)_i , 
 \\
q^{t+1} &= \wt{W}^\top \wt{p}^t - m_t \wt{q}^t , \quad 
\wt{q}^{t+1} = \brv{f}_{t+1}(q^{t+1} + \brv{\nu}_{t+1} \wt{v}^*) , \quad 
\ell_{t+1} = \frac{1}{n} \sum_{i = 1}^d \partial_i \brv{f}_{t+1}(q^{t+1} + \brv{\nu}_{t+1} \wt{v}^*)_i 
,  
\end{split}
\label{eqn:AMP_pq}
\end{align}
where $ \brv{\mu}_t, \brv{\nu}_{t+1} $ are as in \Cref{eqn:SE_mu0,eqn:SE_mu,eqn:SE_nu}. 

Informally, the above iterates are related to $ \brv{u}^t, \brv{v}^{t+1} $ via
\begin{align}
p^t \, \textnormal{`='} \, \brv{u}^t - \brv{\mu}_t \wt{u}^* , \quad 
q^{t+1} \, \textnormal{`='} \, \brv{v}^{t+1} - \brv{\nu}_{t+1} \wt{v}^* , 
\end{align}
where the `equalities' hold only in the large $n$ limit. 
These relations will be made formal in the rest of the proof.

The algorithm \Cref{eqn:AMP_pq} takes the form of a standard AMP iteration with non-separable denoising functions as in \cite{AMP_nonsep,GraphAMP} for which the following state evolution result applies. 
For any $ t\ge0 $ and uniformly pseudo-Lipschitz functions $ \phi, \psi $ of finite order, it holds that 
\begin{align}
\begin{split}
\plim_{n\to\infty} \phi(\wt{u}^*, p^0, \cdots, p^t) - \expt{ \phi(\wt{U}^*, \brv{\sigma}_0 \brv{W}_{U,0}, \cdots, \brv{\sigma}_t \brv{W}_{U,t}) } &= 0 , \\
\plim_{n\to\infty} \psi(\wt{v}^*, q^1, \cdots, q^{t+1}) - \expt{ \psi(\wt{V}^*, \brv{\tau}_1 \brv{W}_{V,1}, \cdots, \brv{\tau}_{t+1} \brv{W}_{V,t+1}) } &= 0 , 
\end{split}
\label{eqn:SE_red}
\end{align}
where $ (\wt{U}^*, \brv{\sigma}_0 \brv{W}_{U,0}, \cdots, \brv{\sigma}_t \brv{W}_{U,t}) $ and $ (\wt{V}^*, \brv{\tau}_1 \brv{W}_{V,1}, \cdots, \brv{\tau}_{t+1} \brv{W}_{V,t+1}) $ are defined in \Cref{eqn:joint_distr,eqn:UV_tilde}. 
Note that \cite{GraphAMP} allows additional randomness independent of $ \wt{W} $ that goes into the denoising functions. 
So the asymptotic guarantee in \Cref{eqn:SE_red} holds for the joint tuple involving $ \wt{U}^*, \wt{V}^* $ as well. 

\Cref{eqn:SE_red} immediately implies
\begin{align}
\begin{split}
\plim_{n\to\infty} \phi({u}^*, p^0 + \brv{\mu}_0 \wt{u}^*, \cdots, p^t + \brv{\mu}_t \wt{u}^*) - \expt{ \phi({U}^*, \brv{U}_0, \cdots, \brv{U}_t) } &= 0 , \\
\plim_{n\to\infty} \psi({v}^*, q^1 + \brv{\nu}_1 \wt{v}^*, \cdots, q^{t+1} + \brv{\nu}_{t+1} \wt{v}^*) - \expt{ \psi({V}^*, \brv{V}_1, \cdots, \brv{V}_{t+1}) } &= 0 , 
\end{split}
\label{eqn:SE_imply}
\end{align}
where we recall the definition of $ \brv{U}_t, \brv{V}_{t+1} $ in \Cref{eqn:UV}. 
We will show that 
\begin{align}
\begin{split}
\plim_{n\to\infty} \phi({u}^*, p^0 + \brv{\mu}_0 \wt{u}^*, \cdots, p^t + \brv{\mu}_t \wt{u}^*) - \phi({u}^*, \brv{u}^0, \cdots, \brv{u}^t) &= 0 , \\
\plim_{n\to\infty} \psi({v}^*, q^1 + \brv{\nu}_1 \wt{v}^*, \cdots, q^{t+1} + \brv{\nu}_{t+1} \wt{v}^*) - \psi({v}^*, \brv{v}^1, \cdots, \brv{v}^{t+1}) &= 0 , 
\end{split}
\label{eqn:SE_todo}
\end{align}
which, when combined with \Cref{eqn:SE_imply}, concludes the proof of \Cref{prop:SE}. 

To show \Cref{eqn:SE_todo}, suppose that $ \phi, \psi $ are uniformly $L$-pseudo-Lipschitz of order $k$. 
Then by the triangle inequality, 
\begin{align}
& \abs{ \phi({u}^*, p^0 + \brv{\mu}_0 \wt{u}^*, \cdots, p^t + \brv{\mu}_t \wt{u}^*) - \phi({u}^*, \brv{u}^0, \cdots, \brv{u}^t) } \notag \\
&\le L \paren{ \sum_{s = 0}^t \frac{1}{\sqrt{n}} \normtwo{p^s + \brv{\mu}_s \wt{u}^* - \brv{u}^s} }\notag\\
& \quad
\times \brack{ 1 + \paren{ \frac{1}{\sqrt{n}} \normtwo{{u}^*} + \sum_{s = 0}^t \frac{1}{\sqrt{n}} \normtwo{p^s + \brv{\mu}_s \wt{u}^*} }^{k-1} + \paren{ \frac{1}{\sqrt{n}} \normtwo{{u}^*} + \sum_{s = 0}^t \frac{1}{\sqrt{n}} \normtwo{\brv{u}^s} }^{k-1} } \notag \\
&\le L' \paren{ \sum_{s = 0}^t \frac{1}{\sqrt{n}} \normtwo{p^s + \brv{\mu}_s \wt{u}^* - \brv{u}^s} }\notag\\
& \quad
\times \brack{ 1 + \paren{ \frac{1}{\sqrt{n}} \normtwo{{u}^*} }^{k-1} + \sum_{s = 0}^t \paren{ \frac{1}{\sqrt{n}} \normtwo{p^s + \brv{\mu}_s \wt{u}^*} }^{k-1} + \sum_{s = 0}^t \paren{ \frac{1}{\sqrt{n}} \normtwo{\brv{u}^s} }^{k-1} } , \notag 
\end{align}
for some $L'$ depending only on $t,k,L$. 
Similar manipulation gives 
\begin{align}
& \abs{ \psi({v}^*, q^1 + \brv{\nu}_1 \wt{v}^*, \cdots, q^{t+1} + \brv{\nu}_{t+1} \wt{v}^*) - \psi({v}^*, \brv{v}^1, \cdots, \brv{v}^{t+1}) } \notag \\
&\le L' \paren{ \sum_{s = 1}^{t+1} \frac{1}{\sqrt{d}} \normtwo{q^s + \brv{\nu}_s \wt{v}^* - \brv{v}^s} }\notag\\
& \quad
\times \brack{ 1 + \paren{ \frac{1}{\sqrt{d}} \normtwo{{v}^*} }^{k-1} + \sum_{s = 1}^{t+1} \paren{ \frac{1}{\sqrt{d}} \normtwo{q^s + \brv{\nu}_s \wt{v}^*} }^{k-1} + \sum_{s = 1}^{t+1} \paren{ \frac{1}{\sqrt{d}} \normtwo{\brv{v}^s} }^{k-1} } . \notag 
\end{align}
Clearly, \Cref{eqn:SE_todo} holds if for every $ t\ge0 $, 
\begin{align}
\plim_{n\to\infty} \frac{1}{\sqrt{n}} \normtwo{p^t + \brv{\mu}_t \wt{u}^*} &< \infty , \label{eqn:finite_p} \\
\plim_{n\to\infty} \frac{1}{\sqrt{n}} \normtwo{\brv{u}^t} &< \infty , \label{eqn:finite_u} \\
\plim_{n\to\infty} \frac{1}{n} \normtwo{ \brv{u}^t - (p^t + \brv{\mu}_t \wt{u}^*) }^2 &= 0 , \label{eqn:ind_u} \\
\plim_{d\to\infty} \frac{1}{\sqrt{d}} \normtwo{q^{t+1} + \brv{\nu}_{t+1} \wt{v}^*} &< \infty , \label{eqn:finite_q} \\
\plim_{d\to\infty} \frac{1}{\sqrt{d}} \normtwo{\brv{v}^{t+1}} &< \infty , \label{eqn:finite_v} \\
\plim_{d\to\infty} \frac{1}{d} \normtwo{ \brv{v}^{t+1} - (q^{t+1} + \brv{\nu}_{t+1} \wt{v}^*) }^2 &= 0 , \label{eqn:ind_v} 
\end{align}
which, together with the following statements
\begin{align}
\brv{\mu}_t < \infty, \quad \brv{\sigma}_t &< \infty , 
\label{eqn:mu_sigma_finite} \\
\brv{\nu}_{t+1} < \infty , \quad 
\brv{\tau}_{t+1} &< \infty , 
\label{eqn:nu_tau_finite} 
\end{align}
will be shown in the sequel by induction on $ t\ge0 $. 

\paragraph{Base case.}
Consider $ t = 0 $. 
From \Cref{eqn:SE_imply}, 
\begin{align}
\plim_{n\to\infty} \frac{1}{n} \normtwo{p^0 + \brv{\mu}_0 \wt{u}^*}^2
&= \plim_{n\to\infty} \frac{1}{n} \expt{ \normtwo{\brv{U}_0}^2 }
= \brv{\sigma}_0^2 + \brv{\mu}_0^2  \expt{ \ol{\Xi}^{-1} } , \label{eqn:p0}
\end{align}
where the last equality is by \Cref{eqn:UV}. 
Due to \Cref{eqn:SE_sigma_init} and \Cref{asmp:init}, both $ \brv{\mu}_0 $ and $ \brv{\sigma}_0 $ are finite, so \Cref{eqn:mu_sigma_finite} holds for $t=0$. 
Consequently, \Cref{eqn:finite_p} also holds for $t=0$. 

Since by \Cref{eqn:pq_init}, 
\begin{align}
(p^0 + \brv{\mu}_0 \wt{u}^*) - \brv{u}^0
&= \brv{\mu}_0 \wt{u}^* - \frac{\lambda}{n} \inprod{\wt{v}^*}{\mathring{v}^0} \wt{u}^* , \notag 
\end{align}
therefore \Cref{eqn:ind_u} for $t=0$ follows from \Cref{eqn:SE_mu0} and \Cref{asmp:init}. 
This in turn implies, when combined with the finiteness of \Cref{eqn:p0}, that 
\begin{align}
\plim_{n\to\infty} \frac{1}{\sqrt{n}} \normtwo{\brv{u}^0} &< \infty , \label{eqn:u0} 
\end{align}
verifying \Cref{eqn:finite_u} for $t=0$. 
Since $ \brv{g}_0 $ is uniformly pseudo-Lipschitz of finite order, so is the function $ \frac{1}{n} \sum_{i = 1}^n \partial_i \brv{g}_0(\cdot)_i $. 
\Cref{eqn:ind_u,eqn:finite_p,eqn:finite_u} (for $t=0$) together imply
\begin{align}
\plim_{n\to\infty} \abs{ m_0 } &< \infty , \quad
\plim_{n\to\infty} \abs{ m_0 - \brv{c}_0 } = 0 . \label{eqn:m0_b0}
\end{align}

Using the the pseudo-Lipschitzness of $ \brv{g}_0 $ again, 
\begin{align}
\plim_{n\to\infty} \frac{1}{\sqrt{n}} \normtwo{ \wt{p}^0 - \mathring{u}^0 }
&= \plim_{n\to\infty} \frac{1}{\sqrt{n}} \normtwo{ \brv{g}_0(p^0 + \brv{\mu}_0\wt{u}^*) - \brv{g}_0(\brv{u}^0) } \notag \\
&\le \plim_{n\to\infty} L \frac{\normtwo{(p^0 + \brv{\mu}_0\wt{u}^*) - \brv{u}^0}}{\sqrt{n}} \brack{ 1 + \paren{\frac{\normtwo{p^0 + \brv{\mu}_0\wt{u}^*}}{\sqrt{n}}}^{k-1} + \paren{\frac{\normtwo{\brv{u}^0}}{\sqrt{n}}}^{k-1} } \notag \\
&= 0 . \label{eqn:pu0_tilde}
\end{align}
The last equality holds because of \Cref{eqn:ind_u} (for $t=0$) and the finiteness of \Cref{eqn:p0,eqn:u0}. 

To show \Cref{eqn:ind_v} for $t=0$, we use \Cref{eqn:AMP_pq,eqn:AMP,eqn:pq_init} to write
\begin{align}
(q^1 + \brv{\nu}_1 \wt{v}^*) - \brv{v}^1
&= \underbrace{\wt{W}^\top (\wt{p}^0 - \mathring{u}^0)}_{T_1} 
+ \underbrace{\paren{ \brv{\nu}_1 - \frac{\lambda}{n} \inprod{\mathring{u}^0}{\wt{u}^*} }}_{T_2} \wt{v}^*
- \underbrace{(m_0 - \brv{c}_0)}_{T_3} \mathring{v}^0 . \notag 
\end{align} 
By \Cref{eqn:pu0_tilde} and the Bai--Yin law \cite{BaiYin}, 
\begin{align}
\plim_{d\to\infty} \frac{1}{d} \normtwo{ T_1 }^2
&= 0 . \label{eqn:T1}
\end{align}
Using \Cref{eqn:pu0_tilde} again, 
\begin{align}
\plim_{n\to\infty} \frac{\lambda}{n} \inprod{\mathring{u}^0}{\wt{u}^*}
&= \plim_{n\to\infty} \frac{\lambda}{n} \inprod{\wt{p}^0}{\wt{u}^*}
= \plim_{n\to\infty} \frac{\lambda}{n} \inprod{\brv{g}_0(p^0 + \brv{\mu}_0 \wt{u}^*)}{\wt{u}^*} \notag \\
&= \lim_{n\to\infty} \frac{\lambda}{n} \expt{ \inprod{\brv{g}_0(\brv{U}_0)}{\wt{u}^*} }
= \brv{\nu}_1 , \label{eqn:ind_nu1}
\end{align}
where in the second line, the first equality is by \Cref{eqn:SE_todo} and the pseudo-Lipschitzness of $ \brv{g}_0 $, and the second equality is by the definition \Cref{eqn:SE_nu}. 
We further show the finiteness of $ \brv{\nu}_1 $. 
Note that
\begin{align}
\brv{\nu}_1 &\le \lim_{n\to\infty} \lambda  \expt{ \frac{1}{n} \normtwo{\brv{g}_0(\brv{U}_0)}^2 }^{1/2} \expt{ \frac{1}{n} \normtwo{\wt{U}^*}^2}^{1/2} . \notag 
\end{align}
The first term can be bounded as
\begin{align}
\plim_{n\to\infty} \frac{1}{n} \expt{ \normtwo{\brv{g}_0(\brv{U}_0)}^2 }
&\le \plim_{n\to\infty} L' \expt{ \paren{ 1 + \paren{\frac{1}{\sqrt{n}} \normtwo{\brv{U}_0}}^k }^2 }
\le \plim_{n\to\infty} 2L' \paren{
    1 + \expt{ \paren{ \frac{1}{n} \normtwo{\brv{U}_0}^2 }^k }
} , \label{eqn:g0_finite}
\end{align}
where the last step is the elementary inequality $ (a+b)^2 \le 2(a^2 + b^2) $. 
The RHS above is finite since  
\begin{align}
\frac{1}{n} \normtwo{\brv{U}_0}^2 &= \frac{\brv{\mu}_0^2}{n} \normtwo{\wt{U}^*}^2 + \frac{\brv{\sigma}_0^2}{n} \normtwo{\brv{W}_{V,0}}^2
\end{align}
whose all moments are finite by finiteness of $ \brv{\mu}_0, \brv{\sigma}_0 $. 
This shows the first bound in \Cref{eqn:nu_tau_finite} for $ t=0 $. 
Recalling \Cref{eqn:ind_nu1}, we then have 
\begin{align}
\plim_{n\to\infty} \abs{ T_2 } &= 0 . \label{eqn:T2}
\end{align}
By \Cref{eqn:m0_b0}, 
\begin{align}
\plim_{n\to\infty} \abs{ T_3 } &= 0 . \label{eqn:T3}
\end{align}
Therefore, \Cref{eqn:T1,eqn:T2,eqn:T3} altogether verify \Cref{eqn:ind_v} for $ t=0 $. 

We then show \Cref{eqn:finite_q} for $t=0$. 
Since $ \brv{\nu}_1<\infty $, it suffices to only consider $ q^1 $. 
According to \Cref{eqn:AMP_pq,eqn:pq_init}, 
\begin{align}
q^1 &= \wt{W}^\top \wt{p}^0 - m_0 \mathring{v}^0 . \notag 
\end{align}
Pseudo-Lipschitzness of $ \brv{g}_0 $, finiteness \Cref{eqn:finite_p} (for $t=0$) and finiteness \Cref{eqn:m0_b0} jointly imply 
\begin{align}
\plim_{n\to\infty} \frac{1}{\sqrt{n}} \normtwo{ q^1 } &< \infty , \notag 
\end{align}
from which \Cref{eqn:finite_q} follows. 
This combined with \Cref{eqn:ind_v} (for $t=0$) also implies \Cref{eqn:finite_v} (for $ t=0 $). 

Finally, we are left with the second inequality in \Cref{eqn:nu_tau_finite}. 
From the definition \Cref{eqn:SE_tau}, 
\begin{align}
\brv{\tau}_1^2 &= \lim_{n\to\infty} \frac{1}{n} \expt{ \normtwo{\brv{g}_0(\brv{U}_0)}^2 } \notag 
\end{align}
which has already been shown to be finite; see \Cref{eqn:g0_finite}. 
So the base case is finished. 

\paragraph{Induction step.}
Assume that \Cref{eqn:finite_p,eqn:finite_u,eqn:ind_u,eqn:finite_q,eqn:finite_v,eqn:ind_v} all hold up to the $t$-th step (for an arbitrary $ t\ge1 $). 
We now show that they hold for $t+1$. 
The idea is similar to the base case.
We briefly lay down the key steps for \Cref{eqn:finite_p,eqn:finite_u,eqn:ind_u,eqn:mu_sigma_finite}, and omit the verification of \Cref{eqn:finite_q,eqn:finite_v,eqn:ind_v,eqn:nu_tau_finite}. 

Using \Cref{eqn:SE_imply,eqn:UV}, 
\begin{align}
\plim_{n\to\infty} \frac{1}{n} \normtwo{p^{t+1} + \brv{\mu}_{t+1} \wt{u}^*}^2
&= \brv{\sigma}_{t+1}^2 + \brv{\mu}_{t+1}^2 \expt{\ol{\Xi}^{-1}} . \label{eqn:finite_p_ind}
\end{align}
Using the definition \Cref{eqn:SE_sigma,} and the pseudo-Lipschitzness of $ \brv{f}_{t+1} $, 
\begin{align}
\brv{\sigma}_{t+1}^2 &= \lim_{n\to\infty} \frac{1}{n} \expt{ \normtwo{\brv{f}_{t+1}(\brv{V}_{t+1})}^2 }
\le \lim_{n\to\infty} 2L' \paren{ 1 + \expt{ \paren{\frac{1}{n} \normtwo{\brv{V}_{t+1}}^2}^{k} } }
, \notag 
\end{align}
for some $L'$ depending only on $k,L$. 
The inequality is obtained in a similar way to \Cref{eqn:g0_finite}. 
By induction hypothesis \Cref{eqn:nu_tau_finite} and the compactness of $\supp(\ol{\Sigma})$, all moments of
\begin{align}
\frac{1}{n} \normtwo{\brv{V}_{t+1}}^2 &= \frac{\brv{\nu}_{t+1}^2}{n} \normtwo{\wt{V}^*}^2 + \frac{\brv{\tau}_{t+1}^2}{n} \normtwo{\brv{W}_{V,t+1}}^2 \notag 
\end{align}
are finite. 
Therefore $ \brv{\sigma}_{t+1}^2 < \infty $, giving the second bound in \Cref{eqn:mu_sigma_finite}. 
Similarly, using the definition \Cref{eqn:SE_mu} and Cauchy--Schwarz, 
\begin{align}
\brv{\mu}_{t+1} &= \lim_{d\to\infty} \frac{\lambda}{n} \expt{\inprod{\wt{V}^*}{\brv{f}_{t+1}(\brv{V}_{t+1})}}
\le \lim_{d\to\infty} \frac{L'\lambda}{\sqrt{n}} \expt{ \normtwo{\wt{V}^*} \paren{ 1 + \paren{ \frac{1}{\sqrt{n}} \normtwo{\brv{V}_{t+1}} }^k } } \notag \\
&\le \lim_{d\to\infty} L' \lambda \expt{ \frac{1}{n} \normtwo{\wt{V}^*}^2 }^{1/2} \paren{ 2 \expt{ 1 + \paren{ \frac{1}{n} \normtwo{\brv{V}_{t+1}}^2 }^k } }^{1/2}
, \notag 
\end{align}
which is again finite for the same reason as $ \brv{\sigma}_{t+1} $, giving the first bound in \Cref{eqn:mu_sigma_finite}. 
Therefore \Cref{eqn:finite_p_ind} is also finite, verifying \Cref{eqn:finite_p} for $t+1$. 

We then show \Cref{eqn:ind_u} for $t+1$. 
Using the recursions \Cref{eqn:AMP,eqn:AMP_pq}, 
\begin{align}
p^{t+1} + \brv{\mu}_{t+1} \wt{u}^* - \brv{u}^{t+1}
&= \underbrace{\wt{W} (\wt{q}^{t+1} - \mathring{v}^{t+1})}_{T_1'}
+ \underbrace{\paren{ \brv{\mu}_{t+1} - \frac{\lambda}{n} \inprod{\wt{v}^*}{\mathring{v}^{t+1}} }}_{T_2'} \wt{u}^*
- \underbrace{(\ell_{t+1} \wt{p}^t - \brv{b}_{t+1} \mathring{u}^t)}_{T_3'} . \notag 
\end{align}
Consider $ T_1' $. 
Since \Cref{eqn:ind_v,eqn:finite_q,eqn:finite_v} are assumed to hold, by pseudo-Lipschitzness of $ \brv{f}_{t+1} $, 
\begin{align}
\plim_{n\to\infty} \frac{1}{\sqrt{d}} \normtwo{ \wt{q}^{t+1} - \mathring{v}^{t+1} } &= 0 . \label{eqn:qv_tilde_t+1}
\end{align}
This with the Bai-Yin law \cite{BaiYin} gives
\begin{align}
\plim_{n\to\infty} \frac{1}{n} \normtwo{T_1'}^2 &= 0 , 
\label{eqn:T1'}
\end{align}
Consider $ T_2' $. 
Recall that $ \brv{\mu}_{t+1} < \infty $. 
Using \Cref{eqn:qv_tilde_t+1,eqn:SE_imply,eqn:SE_mu} and following the argument leading to \Cref{eqn:ind_nu1}, we have
\begin{align}
\plim_{n\to\infty} \abs{ T_2' } = 0 .
\label{eqn:T2'}
\end{align}
Consider $ T_3' $. \
By the triangle inequality,  
\begin{align}
\normtwo{ \ell_{t+1} \wt{p}^t - \brv{b}_{t+1} \mathring{u}^t }
&\le \abs{ \ell_{t+1} - \brv{b}_{t+1} } \normtwo{ \wt{p}^t }
+ \abs{ \brv{b}_{t+1} } \normtwo{ \wt{p}^t - \mathring{u}^t } . \notag 
\end{align}
Since \Cref{eqn:ind_v,eqn:finite_q,eqn:finite_v} are assumed to hold, by pseudo-Lipschitzness of $ \frac{1}{n} \sum_{i=1}^d \partial_i \brv{f}_{t+1}(\cdot)_i $, 
\begin{align}
\plim_{n\to\infty} \abs{\brv{b}_{t+1}} &< \infty , \quad 
\plim_{n\to\infty} \abs{ \ell_{t+1} - \brv{b}_{t+1} } = 0 . \label{eqn:T3_1}
\end{align}
Similarly, pseudo-Lipschitzness of $ \brv{g}_t $ and the hypothesis \Cref{eqn:finite_p,eqn:finite_u,eqn:ind_u} ensures
\begin{align}
\plim_{n\to\infty} \frac{1}{\sqrt{n}} \normtwo{\wt{p}^t} &< \infty , \quad 
\plim_{n\to\infty} \frac{1}{\sqrt{n}} \normtwo{ \wt{p}^t - \mathring{u}^t } = 0 , 
\label{eqn:T3_2}
\end{align}
So combining \Cref{eqn:T3_1,eqn:T3_2}, we have
\begin{align}
\plim_{n\to\infty} \frac{1}{n} \normtwo{ T_3' }^2 &= 0 . \label{eqn:T3'}
\end{align}
\Cref{eqn:T1',eqn:T2',eqn:T3'} altogether verify \Cref{eqn:ind_u} for $t+1$, and therefore also \Cref{eqn:finite_u} by \Cref{eqn:finite_p}. 

The verification of \Cref{eqn:nu_tau_finite,eqn:finite_q,eqn:finite_v,eqn:ind_v} for $ t+1 $ is completely analogous and we do not repeat similar arguments. 
The proof is finally completed. 
\end{proof}


\subsection{Proof of \Cref{prop:SE_bayes}}
\label{app:pf_SE_bayes}

We will prove \Cref{prop:SE_bayes} by reducing the AMP iteration \Cref{eqn:BAMP} (and its associated state evolution \Cref{eqn:mu_nu,eqn:Phi_Psi,eqn:sigma_tau,eqn:all_joint,eqn:UV_bayes}) to the auxiliary AMP \Cref{eqn:AMP} (and its associated state evolution \Cref{eqn:SE_mu0,eqn:SE_mu,eqn:SE_nu,eqn:Phi_11,eqn:Phi_1s,eqn:Phi_rs,eqn:Psi_rs,eqn:UV,eqn:joint_distr,eqn:UV_tilde,eqn:SE_sigma,eqn:SE_tau}). 

Under the following change of variables
\begin{align}
&&
u^t &\coloneqq \Xi^{-1/2} \brv{u}^t , & 
v^{t+1} &\coloneqq \Sigma^{-1/2} \brv{v}^{t+1} , && \label{eqn:change_uv} \\
&&
f_{t+1}(v^{t+1}) &\coloneqq \Sigma^{1/2} \brv{f}_{t+1}(\Sigma^{1/2} v^{t+1}) , & 
g_t(u^t) &\coloneqq \Xi^{1/2} \brv{g}_t(\Xi^{1/2} u^t) , && \label{eqn:change_fg} 
\end{align}
\Cref{eqn:AMP} becomes
\begin{align}
&&
u^t &= \Xi^{-1} A \Sigma^{-1} \wt{v}^t - b_t \Xi^{-1} \wt{u}^t , & 
\wt{u}^t &= g_t(u^t) , && \notag \\
&&
v^{t+1} &= \Sigma^{-1} A^\top \Xi^{-1} \wt{u}^t - c_b \Sigma^{-1} \wt{v}^t , & 
\wt{v}^{t+1} &= f_{t+1}(v^{t+1}) , && \notag 
\end{align}
where $ b_{t+1}, c_t $ are equal to $ \brv{b}_{t+1}, \brv{c}_t $, respectively, but are expressed using the derivatives of $ f_{t+1}, g_t $. 
Specifically, 
\begin{align}
c_t &= \frac{1}{n} \sum_{i = 1}^n \frac{\partial}{\partial \brv{u}^t_i} \brv{g}_t(\brv{u}^t)_i
= \frac{1}{n} \sum_{i = 1}^n \frac{\partial}{\partial \brv{u}^t_i} \paren{ \Xi^{-1/2} g_t(\Xi^{-1/2} \brv{u}^t) }_i \notag \\
&= \frac{1}{n} \sum_{i = 1}^n \sum_{j = 1}^n \sum_{k = 1}^n (\Xi^{-1/2})_{i,j} (\Xi^{-1/2})_{k,i} \frac{\partial g_t(u^t)_j}{\partial u^t_k}
= \frac{1}{n} \tr( (\nabla g_t) \Xi^{-1} ) . \notag 
\end{align}
The second equality follows since by \Cref{eqn:change_fg}, 
\begin{align}
\brv{g}_t(\brv{u}^t) &= \Xi^{-1/2} g_t(\Xi^{-1/2} \brv{u}^t) . \notag 
\end{align}
The third equality is by the chain rule for derivatives (\Cref{prop:chain_derivative}). A similar computation gives
\begin{align}
b_{t+1} &= \frac{1}{n} \tr( (\nabla f_{t+1}) \Sigma^{-1} ) . \notag 
\end{align}
We now see that under the change of variables \Cref{eqn:change_uv,eqn:change_fg}, the AMP iteration \Cref{eqn:BAMP} can be cast as \Cref{eqn:AMP}. 
Therefore, applying the same change of variables to the state evolution of \Cref{eqn:AMP} will produce the state evolution of \Cref{eqn:BAMP}. 
We describe the required modifications below. 

The state evolution result in \Cref{prop:SE} for the AMP in \Cref{eqn:AMP} says that the iterates $ v^*, \brv{v}^1, \brv{v}^2, \cdots, \brv{v}^{t+1} \in\bbR^d $ and $ u^*, \brv{u}^0, \brv{u}^1, \cdots ,\brv{u}^t \in\bbR^n $ converge (in the sense of \Cref{eqn:SE_result_u,eqn:SE_result_v}) respectively to $ V^*, \brv{V}_1, \brv{V}_2, \cdots, \brv{V}_{t+1} \in\bbR^d $ and $ U^*, \brv{U}_1, \brv{U}_2, \cdots, \brv{U}_t \in\bbR^n $. 
Recall that AMPs \Cref{eqn:AMP,eqn:BAMP} operate on the following matrices respectively: 
\begin{align}
\wt{A} 
&= \frac{\lambda}{n} (\Xi^{-1/2} u^*) (\Sigma^{-1/2} v^*)^\top + \wt{W} , \quad 
\Xi^{-1} A \Sigma^{-1} = \frac{\lambda}{n} (\Xi^{-1} u^*) (\Sigma^{-1} v^*)^\top + \Xi^{-1/2} \wt{W} \Sigma^{-1/2} . \notag
\end{align}
In view of \Cref{eqn:change_uv}, to obtain the analogous state evolution result for the AMP in \Cref{eqn:BAMP}, the definition \Cref{eqn:UV} of $ \brv{U}_t, \brv{V}_{t+1} $ 
should be multiplied by $ \Xi^{-1/2} , \Sigma^{-1/2} $ respectively. 
This gives the new definition of $ U_t, V_{t+1} $ in \Cref{eqn:UV_bayes}. 
By the relation \Cref{eqn:change_fg}, the parameters $ \brv{\mu}_t, \brv{\nu}_{t+1} $ in $ \brv{U}_t, \brv{V}_{t+1} $ should be modified as follows: replace $ \brv{f}_t(\brv{V}_t), \brv{g}_t(\brv{U}_t) $ in the recursive equations \Cref{eqn:SE_mu0,eqn:SE_mu,eqn:SE_nu} with $ \Sigma^{-1/2} f_t(V_t), \Xi^{-1/2} g_t(U_t) $. 
This gives the new definition of $ \mu_t, \nu_t $ in \Cref{eqn:mu_nu}. 
Similar operations map equations \Cref{eqn:Phi_11,eqn:Phi_1s,eqn:Phi_rs,eqn:Psi_rs,eqn:SE_sigma,eqn:SE_tau} to equations \Cref{eqn:Phi_Psi,eqn:sigma_tau}. 
Finally, under the new definition of $ U_t, V_{t+1} $, the convergence result \Cref{eqn:SE_result_u,eqn:SE_result_v} translates to \Cref{eqn:SE_result_u_bayes,eqn:SE_result_v_bayes}, which completes the proof.



\section{Auxiliary lemmas}
\label{sec:aux}

\begin{proposition}[Gaussian integral]
\label{prop:gauss_int}
Let $ A\in\bbR^{d\times d} $ be a positive-definite matrix and $ b\in\bbR^d $. 
Then
\begin{align}
\int_{\bbR^d} \exp\paren{ -\frac{1}{2} x^\top A x + b^\top x } \diff x = \sqrt{\frac{(2\pi)^d}{\det(A)}} \exp\paren{ \frac{1}{2} b^\top A^{-1} b } . \notag 
\end{align}
\end{proposition}

\begin{proposition}
\label{prop:quadratic_form}
Let $ V \sim Q^{\ot d} $ where $ Q $ is a fixed distribution on $\bbR$ with mean $0$. 
Let $ B\in\bbR^{d\times d} $ denote a sequence (indexed by $d$) of deterministic matrices such that the empirical spectral distribution of $ \frac{1}{d} B $ converges to the law of a random variable $ \ol{B} $. 
Then 
\begin{align}
\lim_{d\to\infty} \frac{1}{d} \expt{V^\top B V} &= \expt{ \ol{V}^2 } \expt{\ol{B}} \notag 
\end{align}
where $ \ol{V} \sim Q $. 
\end{proposition}

\begin{proposition}
\label{prop:quadratic_form_corr}
Let 
\begin{align}
    (W_1, W_2) &\sim \cN\paren{ \matrix{ 0_d \\ 0_d } , \matrix{
    \sigma_1^2 & \rho \\
    \rho & \sigma_2^2
    } \ot I_d } . \notag 
\end{align}
Let $ B\in\bbR^{d\times d} $ denote a sequence (indexed by $d$) of deterministic matrices such that the empirical spectral distribution of $ \frac{1}{d} B $ converges to the law of a random variable $ \ol{B} $.
Then
\begin{align}
    \lim_{d\to\infty} \frac{1}{d} \expt{ W_1^\top B W_2 }
    &= \rho \expt{ \ol{B} } . \notag 
\end{align}
\end{proposition}

\begin{proposition}[Nishimori identity]
\label{prop:nishimori}
Let $ (X, Y) $ be two random variables. 
Let $ k\ge1 $ and $ X_1, \cdots, X_k $ be $k$ i.i.d.\ samples (given $Y$) from the distribution $ \law(X \mid Y) $. 
Denote by $ \bracket{\cdot} , \expt{\cdot} $ the expectations with respect to $ \law(X \mid Y) $ and $ \law(X, Y) $, respectively. 
Then for all continuous bounded function $f$, it holds that 
\begin{align}
\expt{\bracket{ f(Y, X_1, \cdots, X_k) }} &= \expt{ \bracket{f(Y, X_1, \cdots, X_{k-1}, X)} } . \notag 
\end{align}
\end{proposition}

\begin{proposition}[Conditional distribution of Gaussians]
\label{lem:cond-distr-gauss}
Let $d\ge2$ and $ 1\le p\le d-1 $ be integers. 
Let 
\begin{align}
    \begin{bmatrix} G_1 \\ G_2 \end{bmatrix} &\sim \cN\paren{
        \begin{bmatrix} \mu_1 \\ \mu_2 \end{bmatrix}, 
        \begin{bmatrix} \Sigma_{1,1} & \Sigma_{1,2} \\ 
        \Sigma_{1,2}^\top & \Sigma_{2,2} \end{bmatrix}
    } \notag 
\end{align}
be a $d$-dimensional Gaussian random vector, where $ G_1\in\bbR^p, G_2\in\bbR^{d-p}, \mu_1\in\bbR^p, \mu_2\in\bbR^{d-p}, \Sigma_{1,1}\in\bbR^{p\times p}, \Sigma_{1,2}\in\bbR^{p\times(d-p)}, \Sigma_{2,2}\in\bbR^{(d-p)\times(d-p)} $. 
Then for any $ g_2\in\bbR^{d-p} $, the distribution of $ G_1 $ conditioned on $G_2 = g_2$ is given by $ G_1 \mid \{G_2 = g_2\} \sim \cN(\mu_1', \Sigma_1') $ where 
\begin{align}
    \mu_1' &= \mu_1 + \Sigma_{1,2} \Sigma_{2,2}^{-1} (g_2 - \mu_2) \in\bbR^p , \quad 
    \Sigma_1' = \Sigma_{1,1} - \Sigma_{1,2} \Sigma_{2,2}^{-1} \Sigma_{1,2}^\top \in\bbR^{p\times p} \notag 
\end{align}
and $\Sigma_{2,2}^{-1}$ denotes the generalized inverse of $\Sigma_{2,2}$. 
\end{proposition}



\begin{proposition}[Chain rule of derivatives]
\label{prop:chain_derivative}
Let $ A\in\bbR^{n\times n} $ and $ f\colon\bbR^n\to\bbR^n $. 
Let $ x\in\bbR^n $ and $ \wt{x} \coloneqq Ax $. 
Then for any $ i, j\in[n] $,
\begin{align}
\frac{\partial}{\partial x_i} (A f(A x))_i
&= \sum_{j = 1}^n \sum_{k = 1}^n A_{i,j} A_{k,i} \frac{\partial f(\wt{x})_j}{\partial\wt{x}_k} , \notag 
\end{align}
where $ f(\wt{x})_j\in\bbR $ denotes the $j$-th ($ j\in[n] $) output of $ f(\wt{x})\in\bbR^n $. 
\end{proposition}

\begin{proof}
The proof follows from elementary applications of the chain rule for derivatives. 
Writing $ A = \matrix{ a_1 & \cdots & a_n }^\top $ and $ f = \matrix{ f_1 & \cdots & f_n }^\top $, we have 
\begin{align}
\frac{\partial}{\partial x_i} f_j(Ax)
&= \frac{\partial}{\partial x_i} f_j\paren{ \inprod{a_1}{x}, \cdots, \inprod{a_n}{x} }
= \sum_{k = 1}^n \partial_k f_j(Ax) \frac{\partial \inprod{a_k}{x}}{\partial x_i} 
= \sum_{k = 1}^n A_{k,i} \partial_k f_j(\wt{x}) , \notag 
\end{align}
where $ \partial_k f_j $ denotes the partial derivative of $ f_j\colon\bbR^n\to\bbR $ with respect to its $k$-th argument. 
Then, 
\begin{align}
\frac{\partial}{\partial x_i} (A f(Ax))_i
&= \sum_{j = 1}^n A_{i,j} \frac{\partial}{\partial x_i} f_j(Ax)
= \sum_{j = 1}^n \sum_{k = 1}^n A_{i,j} A_{k,i} \partial_k f_j(\wt{x}) , \notag 
\end{align}
as claimed. 
\end{proof}

\begin{proposition}[Stein's lemma \cite{stein_proc}]
\label{prop:stein}
Let $ W\sim\cN(0,\sigma^2) $ and let $ f\colon\bbR\to\bbR $ be such that both expectations below exist. 
Then $ \expt{W f(W)} = \sigma^2 \expt{f'(W)} $. 
\end{proposition}

\begin{proposition}[Gaussian \poincare inequality {\cite[Theorem 3.20]{BLM_book}}]
\label{prop:gauss_poincare}
Let $ X\sim\cN(0_n, I_n) $ and $ f\colon\bbR^n\to\bbR $ a differentiable function. 
Then 
\begin{align}
    \var{ f(X) } &\le \expt{ \normtwo{ \nabla f(X) }^2 } . \notag 
\end{align}
\end{proposition}

\begin{proposition}[Bounded difference inequality {\cite[Corollary 3.2]{BLM_book}}]
\label{prop:bdd_diff}
Let $ \cU\subset\bbR $ and $ f\colon\cU^n\to\bbR $ a function such that there exist $ c = (c_1, \cdots, c_n) \in \bbR_{\ge0}^n $ satisfying for all $i\in[n]$, 
\begin{align}
    \sup_{(x_1, \cdots, x_n, x_i')\in\cU^{n+1}} \abs{ f(x_1, \cdots, x_{i-1}, x_i, x_{i+1}, \cdots, x_n) - f(x_1, \cdots, x_{i-1}, x_i', x_{i+1}, \cdots, x_n) }
    &\le c_i . \notag 
\end{align}
Then if $ X\in\cU^n $ is a random vector consisting of independent elements, we have $ \var{ f(X) } \le \normtwo{c}^2 / 4 $. 
\end{proposition}

\begin{proposition}[{\cite[Lemma 3.2]{Panchenko_book}}]
\label{prop:panchenko}
If $ f,g $ are differentiable convex functions, then for any $a\in\bbR$ and $a'>0$, 
\begin{align}
    \abs{ f'(a) - g'(a) }
    &\le g'(a + a') - g'(a - a') + B / a' , \notag 
\end{align}
where 
\begin{align}
    B &\coloneqq \abs{ f(a + a') - g(a + a') }
    + \abs{ f(a - a') - g(a - a') }
    + \abs{ f(a) - g(a) } . \notag 
\end{align}
\end{proposition}

\begin{definition}[Monotone conjugate]
\label{def:conjugate}
Let $ f\colon\bbR_{\ge0}\to\bbR $ be a non-decreasing convex function.
Its monotone conjugate $ f^* $ is defined as
\begin{align}
    f^*(x) &= \sup_{y\ge0} xy - f(y) . \notag 
\end{align}
\end{definition}

\begin{proposition}[{\cite[Proposition C.1]{miolane_thesis}}]
\label{prop:deriv}
Let $I\subset\bbR$ be an interval and $ (f_n\colon I\to\bbR)_n $ be a sequence of convex functions converging pointwise to $f$. 
Then for all $t\in I$ such that the following quantities exist, 
\begin{align}
    \lim_{s\uparrow t} f'(s) &\le \liminf_{n\to\infty} \lim_{s\uparrow t} f_n'(s) \le\limsup_{n\to\infty} \lim_{s\downarrow t} f_n'(s) \le \lim_{s\downarrow t} f'(s) . \notag 
\end{align}
\end{proposition}

\begin{proposition}[{\cite[Proposition C.6]{miolane_thesis}}]
\label{prop:dual}
Let $ f,g\colon\bbR_{\ge0}\to\bbR $ be strictly convex differentiable functions and 
\begin{align}
    \cC &\coloneqq \brace{ (q_1, q_2)\in\bbR_{\ge0}^2 \colon q_2 = f'(q_1) , q_1 = g'(q_2) } . \notag 
\end{align}
Then
\begin{align}
    \sup_{(q_1, q_2)\in\cC} f(q_1) + g(q_2) - q_1 q_2
    = \sup_{q_1, q_2\ge0} q_1 q_2 - f^*(q_2) - g^*(q_1)
    = \sup_{q_1\ge0} \inf_{q_2\ge0} f(q_1) + g(q_2) - q_1 q_2 
    , \notag 
\end{align}
and $ \sup_{(q_1, q_2)\in\cC} $ and $ \sup_{q_1, q_2\ge0} $ are achieved at the same $ (q_1^*, q_2^*) $. 
\end{proposition}



\end{document}